\documentclass[a4paper]{amsart}

\usepackage{amsmath, amssymb, amsthm}
\usepackage[margin=3cm]{geometry}
\usepackage{mathtools}
\usepackage{color}
\usepackage{comment}
\usepackage{amsfonts, latexsym}
\usepackage{extpfeil}
\usepackage{tikz-cd}
\usepackage{mathrsfs} 
\usepackage{array, booktabs, float}
\usepackage{enumitem}
\usepackage{hyperref}
\usepackage{adjustbox}

\numberwithin{equation}{section} 

\theoremstyle{plain} 
\newtheorem{theorem}{Theorem}[section]
\newtheorem{proposition}[theorem]{Proposition}
\newtheorem{lemma}[theorem]{Lemma}
\newtheorem{corollary}[theorem]{Corollary}

\theoremstyle{definition} 
\newtheorem{definition}[theorem]{Definintion}

\theoremstyle{remark} 
\newtheorem{remark}[theorem]{Remark}

\newcommand{\gMod}[1]{{#1}\textup{-gMod}}
\newcommand{\gmod}[1]{{#1}\textup{-gmod}}
\newcommand{\gproj}[1]{{#1}\textup{-gproj}}
\newcommand{\quantum}[1]{U_q^{#1}(\mathfrak{g})}

\newcommand{\fpd}[1]{{#1}\textup{-gMod}_{\mathrm{f.p.d.}}}
\newcommand{\Extform}[2]{\langle {#1}, {#2} \rangle_{\mathrm{Ext}}}
\newcommand{\Lusform}[2]{({#1}, {#2})}
\newcommand{\prroot}{\Phi_+^{\mathrm{re}}}
\newcommand{\minroot}{\Phi_+^{\mathrm{min}}}
\newcommand{\Mod}[1]{{#1}\textup{-Mod}}

\newcommand{\myDelta}{\mathbin{\scalebox{0.9}{$\Delta$}}}
\newcommand{\mynabla}{\mathbin{\scalebox{0.9}{$\nabla$}}}

\DeclareMathOperator{\Hom}{Hom}
\DeclareMathOperator{\wt}{wt}
\DeclareMathOperator{\Res}{Res}
\DeclareMathOperator{\height}{ht}
\DeclareMathOperator{\Ind}{Ind}
\DeclareMathOperator{\Coind}{Coind}
\DeclareMathOperator{\hd}{hd}
\DeclareMathOperator{\soc}{soc}

\DeclareMathOperator{\id}{id}
\DeclareMathOperator{\Id}{Id}
\DeclareMathOperator{\Ext}{Ext}
\DeclareMathOperator{\End}{End}
\DeclareMathOperator{\Cok}{Cok}
\DeclareMathOperator{\Ker}{Ker}
\DeclareMathOperator{\Image}{Im}

\DeclareMathOperator{\ext}{ext}
\DeclareMathOperator{\grend}{end}

\DeclareMathOperator{\spn}{span}

\DeclareMathOperator{\Ann}{Ann}
\DeclareMathOperator{\Infl}{Infl}
\DeclareMathOperator{\ch}{ch}
\DeclareMathOperator{\Map}{Map}

\usepackage[all]{xy}
\xyoption{line}
\xyoption{arrow}
\SelectTips{cm}{}

\newcommand{\cross}{
 \xybox{ 
        (-2,-2)*{};(2,2)*{} **\crv{(-2,-1) & (2,1)}?(0) ;
    (2,-2)*{};(-2,2)*{} **\crv{(2,-1) & (-2,1)}?(0);
     }}     
\newcommand{\sline}{\xybox{
  (0,2)*{}; (0,-2)*{} **\dir{-};
}}
\newcommand{\sdot}{\xybox{
  (0,2)*{}; (0,-2)*{} **\dir{-}?(.5)*{\scriptstyle\bullet};
}}
\newcommand{\stau}{\xybox{
(0,0)*{\cross};
(-4,4)*{\cross};
(4,4)*{\cross};
(0,8)*{\cross}; 
(-6,0)*{\sline};
(6,0)*{\sline};
(-6,8)*{\sline};
(6,8)*{\sline};
}}

\tikzset{cross/.style={preaction={-,draw=white,line width=6pt}}}

\title{Quantum imaginary schur-weyl duality}
\author{Haruto Murata}

\address{Graduate School of Mathematical Sciences, the University of Tokyo, 3-8-1 Komaba, Meguro-ku, Tokyo 153-8914, Japan.}
\email{muraharu@ms.u-tokyo.ac.jp}

\makeindex

\begin{document}
\begin{abstract}
We study quiver Hecke algebras of untwisted affine type $A$ with an arbitrary choice of parameters
and establish a duality with the Iwahori-Hecke algebra of the symmetric group.
The parameter $t$ of the Iwahori-Hecke algebra is explicitly determined by the parameters defining the quiver Hecke algebra.
This duality provides a deformation of the imaginary Schur-Weyl duality introduced by Kleshchev and Muth. 
Furthermore, we prove that the characters of simple modules in the imaginary strata are computed in terms of the dual canonical basis and Kazhdan-Lusztig polynomials,
and the characters of standard modules coincide with the PBW vectors of the corresponding quantum group under certain assumptions. 

In addition, we examine other untwisted affine types, 
where the quiver Hecke algebra is known to be independent of the choice of parameters and the imaginary Schur-Weyl duality with the symmetric group has been established. 
As in type $A$, we apply this duality to the computation of characters of simple and standard modules.
\end{abstract}

\maketitle

\tableofcontents

\section{Introduction}
\subsection{Overview}
In this paper, we investigate the quiver Hecke algebra $R$ of type $A_N^{(1)}$ over an algebraically closed field $\mathbf{k}$ with an arbitrary choice of parameters.
As shown in \cite{MR3384473}, the quiver Hecke algebra essentially depends on a single parameter $t \in \mathbf{k}^{\times}$ (see Corollary \ref{cor:Qchoice}), 
and it can be regarded as a one-parameter deformation of the affine Hecke algebra.
Specifically, Rouquier \cite{rouquier20082kacmoodyalgebras} proved that, 
after taking an appropriate completion, 
the quiver Hecke algebra at $t =1$ is isomorphic to the affine Hecke algebra of $\mathrm{GL}$ with quantum characteristic $N+1$. 
In addition, the cyclotomic quotients of the affine Hecke algebra, including Hecke algebras of finite type $A$ and $B$ are isomorphic to cyclotomic quotients of the quiver Hecke algebra at $t=1$ without the need for any completions \cite{MR2551762}.  
Since the quiver Hecke algebra has a grading, these isomorphisms can be used to incorporate non-trivial grading into representations of the affine Hecke algebra. 


In \cite{murata2025affinehighestweightstructures}, 
the author proved that the category of finitely generated graded modules, $\gMod{R}$, 
for an arbitrary parameter $t$ is stratified with respect to any convex order on the set $\Phi_+$ of positive roots of the affine Lie algebra of type $A_N^{(1)}$. 
This result generalizes \cite{MR3694676}, which established the stratification based on geometric results under the assumption that $t = 1$ and $\ch \mathbf{k} = 0$. 
Such a stratification allows us to reduce the study of $\gMod{R}$ to its strata labelled by the positive roots. 
It was further proved that any stratum corresponding to a real root is Morita equivalent to a polynomial ring, regardless of $t$ and $\ch \mathbf{k}$. 
Consequently, the central problem shifts to understanding the imaginary strata corresponding to the imaginary roots $n\delta \ (n \geq 1)$. 

When $t = 1$, the imaginary strata have been extensively studied \cite{MR3589160, MR3694676, MR3934461}. 
A key ingredient is the imaginary Schur-Weyl duality, 
which establishes that the endomorphism algebras of certain modules in the imaginary strata are isomorphic to the group algebras of symmetric groups. 
Based on this duality, 
Kleshchev-Muth \cite{MR3589160} proved that a certain full subcategory of the imaginary strata is Morita equivalent to the Schur algebra for a special convex order.  
This equivalence is referred to as the imaginary Howe-duality.
Furthermore, when $\ch \mathbf{k} =0$ (hence the Schur algebra is semisimple), 
McNamara \cite{MR3694676} showed that the characters of the indecomposable projective modules in the imaginary strata coincide with the imaginary root vectors in the corresponding quantum group \cite{MR2066942,MR3874704}.  
This implies that the characters of standard modules in $\gMod{R}$ are precisely the PBW vectors. 
Kleshchev-Muth \cite{MR3934461} further proved that the entire imaginary strata are Morita equivalent to the affine zigzag algebras, 
which was crucially used in their proof of Turner's conjecture \cite{MR3862945}. 

In this paper, we generalize some of these results to an arbitrary parameter $t$. 
The key new ingredient is quantum imaginary Schur-Weyl duality, 
which shows that the endomorphism algebras of certain modules in the imaginary strata are isomorphic to the Hecke algebras of symmetric groups with the same parameter $t$ as $R$. 
This duality leads to a quantum imaginary Howe duality between the imaginary strata and the $t$-Schur algebra. 
As a consequence, we obtain a natural parametrization of simple modules in $\gMod{R^e(n\delta)}$ by multipartitions $(\lambda^{(i)})_{i \in \mathring{I}}$ of $n$, 
where $\mathring{I} = \{1,\ldots, N\}$ denotes the set of nodes of the corresponding finite Dynkin diagram. 
We prove that the characters of the modules corresponding to Weyl modules of the $t$-Schur algebra belong to the dual canonical basis.
This implies that we can compute the characters of simple modules in the imaginary strata in terms of the dual canonical basis and Kazhdan-Lusztig polynomials if the characteristic of $\mathbf{k}$ is zero. 
Furthermore, when the quantum characteristic of $t$ is infinite (hence the $t$-Schur algebra is semisimple), 
we prove that the characters of indecomposable projective modules in the imaginary strata coincide with the imaginary root vectors. 
This again implies that the characters of standard modules in $\gMod{R}$ are the PBW vectors. 
In the rest of this introduction, we provide a precise account of our main results.


\subsection{Imaginary strata}
Let $\mathfrak{g}$ be the affine Lie algebra of type $A_N^{(1)}$ with $N \geq 1$, 
and let $I = \{0,1,\ldots, N\} \simeq \mathbb{Z}/(N+1)\mathbb{Z}$ be the set cyclically labelling the Dynkin nodes.  
Let $\delta$ be the null root. 
Let $\mathbf{k}$ be a field, 
and let $R(\beta) \ (\beta \in \mathsf{Q}_+$) be the quiver Hecke algebra of type $A_N^{(1)}$ over $\mathbf{k}$, 
where $\mathsf{Q}_+$ denotes the positive root lattice of $\mathfrak{g}$. 
The algebra $R(\beta)$ implicitly depends on a choice of homogeneous polynomials $Q_{i,j}(u,v)$ for $i,j \in I$.
As we will see in Corollary \ref{cor:Qchoice}, $R(\beta)$ is isomorphic to the quiver Hecke algebra associated with another choice of polynomials $Q'_{i,j}(u,v)$ satisfying
\[
Q'_{i,i+1}(u,v) = \begin{cases}
(u-v)(v-tu) & \text{if $N =1$}, \\
u-v & \text{if $N \geq 2$ and $i \neq N$}, \\
u-tv & \text {if $N \geq 2$ and $i = N$},
\end{cases}
\]
where $t \in \mathbf{k}^{\times}$ is determined by the coefficients of the original $Q_{i,j} \ (i,j\in I)$. 
Hence, we may assume that the polynomials $Q_{i,j}$ are of this form from the beginning. 

Let $w$ be an element of the finite Weyl group $\mathring{W} = \mathfrak{S}_N$. 
For $n \geq 0$, let $R^w(n\delta)$ be the quotient algebra of $R(n\delta)$ by the ideal generated by certain idempotents $e(\beta,\gamma) \ (\beta,\gamma \in \mathsf{Q}_+, \beta + \gamma = n\delta)$ satisfying the condition:
\[
\text{$\beta \not \in \spn_{\mathbb{Z}_{\geq 0}} (\Phi_+ \cap p^{-1}(w\mathring{\mathsf{Q}}_-))$ or $\gamma \not \in \spn_{\mathbb{Z}_{\geq 0}} (\Phi_+ \cap p^{-1}(w\mathring{\mathsf{Q}}_+))$}, 
\]
where $p \colon \mathsf{Q} \to \mathring{\mathsf{Q}}$ is the projection from the affine root lattice to the finite root lattice, 
and $\mathring{\mathsf{Q}}_+$ (resp.\ $\mathring{\mathsf{Q}}_-$) denotes the positive (resp.\ negative) part of $\mathring{\mathsf{Q}}$. 
The categories $\gMod{R^w(n\delta)}$ of finitely generated graded $R^w(n\delta)$-modules are the imaginary strata appearing in the stratification established in \cite{murata2025affinehighestweightstructures}; see Subsection \ref{sub:stratifications}. 
The element $w$ is referred to as the coarse type. 
We define the category $\gMod{R^w}$ as the direct sum $\bigoplus_{n \in \mathbb{Z}_{\geq 0}} \gMod{R^w(n\delta)}$, 
which is a monoidal category via the convolution product denoted by $\circ$.

We utilize the reflection functors constructed in \cite{murata2025diagrammaticapproachreflectionfunctors} to reduce the study of the imaginary strata to the case where the coarse type $w$ is the identity element $e$. 
In fact, we prove in Corollary \ref{cor:stdcoarsetype} that a suitable composition of reflection functors induces an equivalence of categories between $\gMod{R^w(n\delta)}$ and $\gMod{R^e(n\delta)}$, 
which commutes with the convolution product. 
Having established this reduction, we shall primarily focus on $\gMod{R^e}$. 

\subsection{Quantum imaginary Schur-Weyl duality}

For each $i \in \mathring{I} = \{1,\ldots,N\}$, there exists a unique (up to isomorphism) self-dual simple $R^e(\delta)$-module $L_i^e(\delta)$ such that 
\[
e(\delta-\alpha_i,\alpha_i) L_i^e(\delta) \neq 0. 
\]
Let $\Delta_i^e(\delta)$ be the projective cover of $L_i^e(\delta)$ in $\gMod{R^e(\delta)}$. 
For $\underline{n} = (n_i)_{i \in \mathring{I}} \in \mathbb{Z}_{\geq 0}^{\mathring{I}}$, 
we define 
\[
L^e(\underline{n}) = L_1^e(\delta)^{\circ n_1} \circ \cdots \circ L_N^e(\delta)^{\circ n_N}, \ \Delta^e(\underline{n}) = \Delta_1^e(\delta)^{\circ n_1} \circ \cdots \circ \Delta_N^e(\delta)^{\circ n_N}. 
\]
For $1 \leq i < j \leq N$, there is an isomorphism 
\[
L_j^e(\delta) \circ L_i^e(\delta) \to L_i^e(\delta) \circ L_j^e(\delta), 
\]
which is unique up to a scalar multiple.
This isomorphism uniquely lifts to a homogeneous isomorphism 
\[
\Delta_j^e(\delta) \circ \Delta_i^e(\delta) \to \Delta_i^e(\delta) \circ \Delta_j^e(\delta). 
\]
By composing these isomorphisms a minimal number of times to reorder the convolution factors,
we obtain isomorphisms:
\begin{equation*} 
L^e(\underline{m}) \circ L^e(\underline{n}) \simeq L^e(\underline{m}+\underline{n}), \ \Delta^e(\underline{m}) \circ \Delta^e(\underline{n}) \simeq \Delta^e(\underline{m}+\underline{n}) \ (\underline{m},\underline{n} \in \mathbb{Z}_{\geq 0}^{\mathring{I}}). 
\end{equation*}
These, in turn, induce the following homomorphisms among algebras of homogeneous endomorphisms: 
\begin{align} \label{eq:intro} 
&\grend_R(\Delta^e(\underline{m})) \otimes \grend_R(\Delta^e(\underline{n})) \xrightarrow{can} \grend_R(\Delta^e(\underline{m})\circ \Delta^e(\underline{n})) \simeq \grend_R(\Delta^e(\underline{m}+\underline{n})), \\ 
&\grend_R(L^e(\underline{m})) \otimes \grend_R(L^e(\underline{n})) \xrightarrow{can} \grend_R(L^e(\underline{m})\circ L^e(\underline{n})) \simeq \grend_R(L^e(\underline{m}+\underline{n})). \notag
\end{align}

For $n \in \mathbb{Z}_{\geq 0}$, 
let $H_n$ be the Hecke algebra of the symmetric group of degree $n$ over $\mathbf{k}$ with parameter $t \in \mathbf{k}^{\times}$. 
For $\underline{n} \in \mathbb{Z}_{\geq 0}^{\mathring{I}}$, 
we set
\[
H_{\underline{n}} = H_{n_1} \otimes \cdots \otimes H_{n_N}. 
\] 
Let $\lvert \underline{n} \rvert = \sum_{i \in \mathring{I}} n_i$. 
The main result of this paper is the following quantum imaginary Schur-Weyl duality:

\begin{theorem}[Theorem \ref{thm:heckeaction}] \label{thm:intro1}
  For each $\underline{n} \in \mathbb{Z}_{\geq 0}^{\mathring{I}}$, there exist isomorphisms of algebras
  \[
 \grend_{R(\lvert \underline{n} \rvert \delta)} (\Delta^e(\underline{n})) \simeq \grend_{R(\lvert \underline{n}\rvert\delta)} (L^e(\underline{n})) \simeq H_{\underline{n}},
  \]
  through which the homomorphisms in (\ref{eq:intro}) coincide with the canonical inclusion (see Subsection \ref{sub:heckeaction})
  \[
  H_{\underline{m}}\otimes H_{\underline{n}} \to H_{\underline{m}+\underline{n}}. 
  \]
\end{theorem}

The main challenge in proving this theorem lies in finding appropriate endomorphisms that satisfy the defining relations of $H_{\underline{n}}$. 
In the case of type $A_1^{(1)}$, we explicitly construct these endomorphisms and verify the relations by direct calculation. 
To reduce the general case of type $A_N^{(1)}$ to type $A_1^{(1)}$, 
we employ certain monoidal functors
\[
\mathcal{F}_i \colon \gMod{\underline{R}} \to \gMod{R} \ (i \in \mathring{I}), 
\]
which relate the quiver Hecke algebra $R$ of type $A_N^{(1)}$ to the quiver Hecke algebra $\underline{R}$ of type $A_1^{(1)}$, following the approaches in \cite{MR3670026} and \cite{MR3790066}. 

\subsection{Quantum imaginary Howe duality}

For $\underline{n} \in \mathbb{Z}_{\geq 0}^{\mathring{I}}$ with $\lvert \underline{n} \rvert = \sum_i n_i = n$, 
we define a finite-dimensional graded algebra
\[
\mathscr{S}_{\underline{n}} = R^e(n\delta)/\Ann (L^e(\underline{n})). 
\]
For $n \geq 0$, let $\mathscr{S}_n$ be the direct product of all $\mathscr{S}_{\underline{n}}$ over $\underline{n} \in \mathbb{Z}_{\geq 0}^{\mathring{I}}$ such that $\lvert \underline{n} \rvert = n$. 
By inflation, we obtain a functor $\gMod{\mathscr{S}_n} \to \gMod{R^e(n\delta)}$, which is shown to be fully faithful (Corollary \ref{cor:ff}). 

Let $\underline{n} \in \mathbb{Z}_{\geq 0}^{\mathring{I}}$. 
By Theorem \ref{thm:intro1} and the definition of $\mathscr{S}_{\underline{n}}$, $L^e(\underline{n})$ is a left $\mathscr{S}_{\underline{n}} \otimes H_{\underline{n}}$-module. 
By using the antiautomorphism of $H_{\underline{n}}$ that fixes all the standard generators, 
we regard $L^e(\underline{n})$ as an $(\mathscr{S}_{\underline{n}},H_{\underline{n}})$-bimodule. 
For $\underline{h} \in \mathbb{Z}_{\geq 0}^{\mathring{I}}$, we define the set
\[
X(\underline{h},\underline{n}) = \prod_{i \in \mathring{I}} X(h_i,n_i), 
\]
where $X(h_i,n_i)$ denotes the set of all compositions $\lambda^{(i)} = (\lambda_1^{(i)},\ldots,\lambda_{h_i}^{(i)})$ of $n_i$. 
For $\underline{\lambda} = (\lambda^{(i)}) \in X(\underline{h},\underline{n})$, 
we define an $\mathscr{S}_{\underline{n}}$-submodule $Z^{\underline{\lambda}}$ of $L^e(\underline{n})$ by
\[
Z^{\underline{\lambda}} = \{ v \in L^e(\underline{n}) \mid v \cdot \iota_i (T_w) = (-1)^{\ell(w)}v \ (i \in \mathring{I}, w \in \mathfrak{S}_{\lambda^{(i)}}) \}, 
\]
where $\iota_i \colon H_{n_i} \to H_{\underline{n}}$ is the inclusion into the $i$-th factor and $T_w$ is the standard basis element of $H_{n_i}$. 
Let $Z^{\underline{h},\underline{n}}$ be the direct sum of all $Z^{\underline{\lambda}}$ over $\underline{\lambda} \in X(\underline{h},\underline{n})$. 

For $\underline{n}, \underline{h} \in \mathbb{Z}_{\geq 0}^{\mathring{I}}$, 
we define an algebra
\[
S_{\underline{h},\underline{n}} = S_{h_1,n_1} \otimes \cdots \otimes S_{h_N,n_N},
\]
where $S_{h_i,n_i}$ is the $t$-Schur algebra over $\mathbf{k}$ with the same parameter $t \in \mathbf{k}^{\times}$ as that of $R$. 
We refer the reader to Subsection \ref{sub:tSchur} for the precise definition and notation concerning the $t$-Schur algebra. 

Based on Theorem \ref{thm:intro1}, we obtain the following quantum imaginary Howe duality generalizing \cite{MR3589160}. 

\begin{theorem}[Theorem \ref{thm:howe}] \label{thm:intro2}
(1) The module $Z^{\underline{h},\underline{n}}$ is a projective $\mathscr{S}_{\underline{n}}$-module, 
and there exists an isomorphism of graded $\mathbf{k}$-algebras
\[
\End_{\mathscr{S}_{\underline{n}}} ( Z^{\underline{h},\underline{n}})^{\mathrm{op}} \simeq S_{\underline{h},\underline{n}}, 
\]
where $S_{\underline{h},\underline{n}}$ is regarded as a graded algebra concentrated in degree zero. 

(2) If $h_i \geq n_i$ for all $i \in \mathring{I}$, then $Z^{\underline{h},\underline{n}}$ is a projective generator of $\mathscr{S}_{\underline{n}}$. 
In this case, the functors
\begin{align*}
\alpha_{\underline{h},\underline{n}} \coloneqq \Hom_{\mathscr{S}_{\underline{n}}}(Z^{\underline{h},\underline{n}},?) \colon \gMod{\mathscr{S}_{\underline{n}}} &\to \gMod{S_{\underline{h},\underline{n}}}, \\
\beta_{\underline{h},\underline{n}} \coloneqq Z^{\underline{h},\underline{n}}\otimes_{S_{\underline{h},\underline{n}}} ? \colon \gMod{S_{\underline{h},\underline{n}}} &\to \gMod{\mathscr{S}_{\underline{n}}}, 
\end{align*}
are mutually quasi-inverse equivalences. 
Moreover, the following diagram commutes up to natural isomorphisms: 
\[
\begin{tikzcd}
\gMod{\mathscr{S}_{\underline{n}}} \arrow[rr,yshift=0.7ex,"\alpha_{\underline{h},\underline{n}}"]\arrow[rd,"\gamma_{\underline{n}}"'] && \gMod{S_{\underline{h},\underline{n}}} \arrow[ll,yshift=-0.7ex,"\beta_{\underline{h},\underline{n}}"]\arrow[ld,"f_{\underline{h},\underline{n}}"] \\
& \gMod{H_{\underline{n}}} & 
\end{tikzcd}
\]
where 
\[
\gamma_{\underline{n}} = \Hom_{\mathscr{S}_{\underline{n}}}(L^e(\underline{n}),?), 
\]
and $f_{\underline{h},\underline{n}}$ is the Schur functor (see Theorem \ref{thm:howe} for its precise definition). 
\end{theorem}

It is well known that $S_{\underline{h},\underline{n}}$ is quasi-hereditary with respect to the poset $X_+(\underline{h},\underline{n})$ of multipartitions, 
where the partial order is the dominance order. 
Specifically, for each $\underline{\lambda}$, 
there exists a unique simple module $L_{\underline{h}}(\underline{\lambda})$ of highest weight $\underline{\lambda}$. 
The standard module for the $t$-Schur algebra is called the Weyl module and is denoted by $W_{\underline{h}}(\underline{\lambda})$, 
to distinguish it from the standard modules of the quiver Hecke algebra.  
For a multipartition $\underline{\lambda}$ of $\underline{n}$, we define $\mathscr{S}_{\underline{n}}$-modules
\[
L^e(\underline{\lambda}) = \beta_{\underline{h},\underline{n}}(L_{\underline{h}}(\underline{\lambda})), \ W^e(\underline{\lambda}) = \beta_{\underline{h},\underline{n}}(W_{\underline{h}}(\underline{\lambda})), 
\]
which are independent of $\underline{h} \in \mathbb{Z}_{\geq 0}^{\mathring{I}}$ satisfying $h_i \geq n_i$. 
As a corollary of Theorem \ref{thm:intro2}, 
we prove in Corollary \ref{cor:imaginarysimple} that the set 
\[
\{L^e(\underline{\lambda}) \mid \text{$\underline{\lambda} = (\lambda^{(i)})_{i \in \mathring{I}}$ is a multipartition} \}
\]
is a complete set of representatives of isomorphism classes of (self-dual) simple modules in $\gMod{R^e}$. 
Furthermore, we show in Theorem \ref{thm:monoidal} that the equivalence established in Theorem \ref{thm:intro2} (2) is monoidal; 
it intertwines the convolution product for the quiver Hecke algebra and the parabolic induction for the $t$-Schur algebra. 

\subsection{Characters}

By \cite{MR2525917,MR2763732}, there exists an isomorphism between the Grothendieck group of $\gMod{R}$ and the negative part of the quantum group of type $A_N^{(1)}$: 
\[
\chi \colon K(\gMod{R}) \otimes_{\mathbb{Z}[q,q^{-1}]} \mathbb{Q}(q) \to U_q^-(\mathfrak{g}). 
\]
For a module $M \in \gMod{R}$, we interpret $\chi(M) \in U_q^-(\mathfrak{g})$ as the character of $M$.  

We have the following result on the character of $W^e(\underline{\lambda})$: 

\begin{theorem}[{Theorem \ref{thm:Weylcharacter}}] \label{thm:intro2.5}
For any multipartition $\underline{\lambda}$, the character $\chi(W^e(\underline{\lambda})) \in U_q^-(\mathfrak{g})$ belongs to the dual canonical basis. 
\end{theorem}

This implies that the characters of simple modules in the imaginary strata are computed in terms of the dual canonical basis and the decomposition matrix of the $t$-Schur algebra. 
In particular, if the characteristic of $\mathbf{k}$ is zero, 
we obtain an algorithm to  compute them using the canonical basis of the Fock space representation or certain parabolic Kazhdan-Lusztig polynomials; see Remark \ref{rem:character}. 

Now, assume that the quantum characteristic of $t \in \mathbf{k}$ is infinite; 
that is, there exists no $e \in \mathbb{Z}_{\geq 1}$ such that $1 + t + \cdots + t^{e-1} = 0$. 
Under this assumption, it is well known that the $t$-Schur algebra is semisimple and the Schur functor $f_{\underline{n},\underline{n}} \colon \gMod{S_{\underline{h},\underline{n}}} \to \gMod{H_{\underline{n}}}$ is an equivalence. 
For a multipartition $\underline{\lambda}$ of $\underline{n}$, 
the module $Sp(\underline{\lambda}) = f_{\underline{n},\underline{n}}(L_{\underline{n}}(\underline{\lambda}))$ is the Specht module of $H_{\underline{n}}$, 
and we have an isomorphism
\[
L^e(\underline{\lambda}) \simeq L^e(\underline{n}) \otimes_{H_{\underline{n}}} Sp(\underline{\lambda}). 
\]
We define 
\[
\Delta^e(\underline{\lambda}) = \Delta^e(\underline{n}) \otimes_{H_{\underline{n}}} Sp(\underline{\lambda}). 
\]
Its character is described as follows:

\begin{theorem}[Theorem \ref{thm:characterformula}] \label{thm:intro3}
Assume that the quantum characteristic of $t$ is infinite. 
For any multipartition $\underline{\lambda}$, 
the module $\Delta^e(\underline{\lambda})$ is the projective cover of the simple module $L^e(\underline{\lambda})$ in $\gMod{R^e}$, 
and its character satisfies 
\[
\chi(\Delta^e(\underline{\lambda})) = S_{\underline{\lambda}}^e, 
\]
where $S_{\underline{\lambda}}^e$ is the imaginary root vector introduced by Beck-Nakajima \cite{MR2066942}; see Subsection \ref{sub:affinePBW}. 
\end{theorem}

This theorem generalizes \cite{MR3694676}. 
While our proof largely follows \cite{MR3694676},
additional care is required because the original argument relies on a geometric result \cite[Lemma 7.5]{MR3694676}, 
which is not available in our general setting. 
To overcome this difficulty, we utilize several consequences from the theory of $R$-matrices developed by Kang, Kashiwara, Kim, Oh, and Park \cite{MR3314831,MR3758148,MR3790066} (see Subsection \ref{sub:Rmatrix}). 

\subsection{Other types}

In the main body of this paper, 
we study quiver Hecke algebras of all untwisted affine types simultaneously in a uniform manner. 
Let $R$ be a quiver Hecke algebra of an untwisted affine type other than type $A$.

It is known that the isomorphism class of the quiver Hecke algebra is independent of the choice of polynomials $Q_{i,j}(u,v)$ (see Corollary \ref{cor:Qchoice}); 
thus we may adopt the specific polynomials presented there. 
The category $\gMod{R}$ for these types is also stratified, 
and it remains to study the imaginary strata $\gMod{R^w}$ for $w \in \mathring{W}$. 
Using reflection functors from \cite{murata2025diagrammaticapproachreflectionfunctors}, we prove that $\gMod{R^w}$ is equivalent to $\gMod{R^e}$. 

For $\gMod{R^e}$ associated with the specific choice of polynomials mentioned above, 
imaginary Schur-Weyl duality and imaginary Howe duality have already been established in \cite{MR3589160}. 
Hence, analogues of Theorem \ref{thm:intro1} and Theorem \ref{thm:intro2} hold with $t=1$; 
in this setting, the Hecke algebra always reduces to the group algebra of the symmetric group, and the $t$-Schur algebra becomes the classical Schur algebra. 

Furthermore, we prove analogues of Theorem \ref{thm:intro2.5} and \ref{thm:intro3} for these types. 
These character formulas are new results for positive characteristic or nonsymmetric types. 

\subsection*{Acknowledgement}
I am deeply grateful to my supervisor, Noriyuki Abe, for his continuous support and invaluable feedback throughout the course of  this research. 
This work was supported by JSPS KAKENHI Grant Number 25KJ1132. 

\section{Preliminaries} 

\subsection{Notations and Conventions} \label{sub:notation}

Let $\mathbf{k}$ be a unital commutative ring.
The tensor product over $\mathbf{k}$ is usually denoted by $\otimes$. 
Let $A$ be a $\mathbb{Z}$-graded $\mathbf{k}$-algebra.
Unless otherwise specified, all $A$-modules are left modules. 
Let $\gMod{A}$ (resp.\ $\gproj{A}, \gmod{A}, \fpd{A}$) denote the category of finitely generated graded $A$-modules (resp.\ finitely generated projective graded $A$-modules, graded $A$-modules that are finitely generated over $\mathbf{k}$, finitely generated graded $A$-modules with finite projective dimension) whose morphisms are degree-preserving A-module homomorphisms. 
For a graded $A$-module $X$, we define its grading shift $qX$ by $(qX)_d = X_{d-1}$. 
Then $\gMod{A}, \gproj{A}$, $\gmod{A}$ and $\fpd{A}$ are graded categories with the automorphism $q$. 
It makes the Grothendieck group $K(\gMod{A})$ a $\mathbb{Z}[q,q^{-1}]$-module. 
Let $K(\gMod{A})_{\mathbb{Q}(q)}$ denote its scalar extension to $\mathbb{Q}(q)$. 
For $X, Y \in \gMod{A}$, let $\hom_A(X,Y)$ denote the space of morphisms from $X$ to $Y$ in the category $\gMod{A}$. 
For $k \in \mathbb{Z}$, let $\ext_A^k(-,Y)$ denote the $k$-th right derived functor of $\hom_A(-,Y)$. 
When $A$ is Noetherian, the ordinary $\Hom_A(X,Y)$ and $\Ext_A^k(X,Y)$ are $\mathbb{Z}$-graded by setting
\[
\Hom_A(X,Y)_d = \hom_A(q^dX,Y), \ \Ext_A^k(X,Y)_d = \ext_A^k(q^dX,Y), 
\]
since $X$ is finitely generated. 

Now, assume that $\mathbf{k}$ is a field. 
The dimension of a $\mathbf{k}$-vector space is usually denoted by $\dim$. 
A graded $\mathbf{k}$-vector space $V = \bigoplus_{d \in \mathbb{Z}}V_d$ is Laurentian if  
every homogeneous component $V_d$ is finite dimensional and $V_d = 0$ for sufficiently small $d$.
In this case, we define a Laurent series
\[
\dim_q V = \sum_{d \in \mathbb{Z}}(\dim V_d) \ q^d \in \mathbb{Z}((q)). 
\]

Assume that $A$ is Laurentian. 
For $X \in \fpd{A}$ and $Y \in \gMod{A}$, we define 
\[
\Extform{X}{Y} = \sum_{k \geq 0} (-1)^k\dim_q \Ext_A^k(X,Y). 
\]
Note that this is well-defined since the right hand side is a finite sum. 
Let $L \in \gMod{A}$ be a simple module, and let $P \in \gproj{A}$ be a projective cover of $L$. 
For $X \in \gMod{A}$, we define 
\[
[X:L]_q = \dim_q \Hom_A (P,X). 
\]
The simple module $L$ is said to be a composition factor of $X$ if $\hom_A(P,X) \neq 0$, 
or equivalently, if $[X:L]_q$ has a nonzero constant term. 

For $n \in \mathbb{Z}_{\geq 0}$, let $\mathfrak{S}_n$ denote the symmetric group of degree $n$, acting on $\{1,2, \ldots, n\}$.
Let $w_n$ be the longest element of $\mathfrak{S}_n$, namely, $w_n(k) = n+1-k \ (1 \leq k \leq n)$. 
For $n', n'' \in \mathbb{Z}_{\geq 0}$ with $n' + n'' = n$, 
let $w[n',n'']$ be the element of $\mathfrak{S}_n$ defined by 
\[
w[n',n''] (k) = \begin{cases}
k+ n'' & \text{if $1 \leq k \leq n'$}, \\
k-n' & \text{if $n'< k \leq n$}. 
\end{cases}
\]
Let $\ell \colon \mathfrak{S}_n \to \mathbb{Z}_{\geq 0}$ be the length function. 

For a composition $\lambda = (\lambda_1,\ldots,\lambda_h)$ of $n$, 
we define a parabolic subgroup
\[
\mathfrak{S}_{\lambda} = \mathfrak{S}_{\lambda_1} \times \cdots \times \mathfrak{S}_{\lambda_h} \subset \mathfrak{S}_n. 
\]
Let $\mathfrak{S}_n^{\lambda}$ (resp.\ ${}^{\lambda}\mathfrak{S}_n$) be the set of minimal length coset representatives for $\mathfrak{S}_n/\mathfrak{S}_{\lambda}$ (resp.\ $\mathfrak{S}_{\lambda}\backslash \mathfrak{S}_n$).
For another composition $\mu$ of $n$, we define ${}^{\lambda}\mathfrak{S}_n^{\mu} = {}^{\lambda}\mathfrak{S}_n \cap \mathfrak{S}_n^{\mu}$. 

\subsection{Quantum groups} \label{sub:quantumgroups}
Let $\mathsf{A} = (a_{i,j})_{i,j \in I}$ be the Cartan matrix of untwisted affine type $X_N^{(1)}$ as in \cite[Chapter 4, Table Aff 1]{MR1104219}, 
and let $\mathfrak{g} = \mathfrak{g}(\mathsf{A})$ be the affine Lie algebra. 
We label $I = \{ 0,1, \ldots,N\}$, where $0 \in I$ denotes the left-most node in the Dynkin diagram of \cite[Chapter 4, Table Aff 1]{MR1104219}. 
In type $A_N^{(1)} \ (N \geq 2)$, we choose the labelling so that $a_{k-1,k} = -1 \ (1 \leq k \leq N)$ and $a_{N,0} = -1$. 
We fix a root datum $(\mathsf{A}, \mathsf{P}, \Pi, \Pi^{\lor}, (\cdot, \cdot))$ as follows. 
Let 
\[
\mathsf{P} = \bigoplus_{i \in I} \mathbb{Z} \varpi_i \oplus \mathbb{Z}\delta
\]
be the weight lattice. 
For each $i \in I$, we call $\varpi_i$ the fundamental weight. 
Let $\Pi = \{\alpha_i \mid i \in I \} \subset \mathsf{P}$ be the set of simple roots, where 
\[
\alpha_i = \sum_{j \in I} a_{j,i} \varpi_j + \delta_{i,0}\delta. 
\]
Put $\mathsf{P}^{\lor} = \Hom_{\mathbb{Z}}(\mathsf{P},\mathbb{Z})$ and let $\Pi^{\lor} = \{ h_i \mid i \in I \}$ be the set of simple coroots, where $h_i$ is determined by
\[
\langle h_i, \varpi_j \rangle = \delta_{i,j},\ \langle h_i, \delta \rangle = 0 \ (i, j \in I). 
\]
Let $(\cdot,\cdot)$ be the bilinear form on $\mathsf{P} \otimes \mathbb{Q}$ satisfying 
\[
\langle h_i, \lambda \rangle = 2\frac{(\alpha_i,\lambda)}{(\alpha_i,\alpha_i)} \ (i \in I, \lambda \in \mathsf{P}),\ \min \{(\alpha_i,\alpha_i) \mid i \in I \} = 2. 
\]
Let $W$ be the Weyl group, which is generated by the simple reflections $s_i \ (i \in I)$. 
We define 
\[
\mathsf{Q}=\sum_{i \in I} \mathbb{Z} \alpha_i,\ \mathsf{Q}_+ = \sum_{i \in I }\mathbb{Z}_{\geq 0}\alpha_i, \ \mathsf{Q}_- = - \mathsf{Q}_+ \subset \mathsf{P}. 
\]
Let $\height \colon \mathsf{Q}_+ \to \mathbb{Z}_{\geq 0}$ be the morphism of monoids given by $\height (\alpha_i) = 1 \ (i \in I)$. 
Let $\Phi$ (resp.\ $\Phi_+, \Phi_+^{\mathrm{re}}$) be the set of roots (resp.\ positive roots, positive real roots) of $\mathfrak{g}$. 

Let $\mathring{I} = \{1,\ldots,N\}$. 
Then, the Cartan matrix $(a_{i,j})_{i,j \in \mathring{I}}$ is of finite type. 
Let $\mathring{\mathfrak{g}}$ be the corresponding simple Lie algebra and let $\mathring{W}$ be its Weyl group. 
We define 
\[
\mathring{\mathsf{Q}}=\sum_{i \in \mathring{I}} \mathbb{Z} \alpha_i ,\ \mathring{\mathsf{Q}}_+ = \sum_{i \in \mathring{I}}\mathbb{Z}_{\geq 0}\alpha_i, \ \mathring{\mathsf{Q}}_- = - \mathring{\mathsf{Q}}_+ \subset \mathsf{P}. 
\]
Let $\mathring{\Phi}$ (resp.\ $\mathring{\Phi}_+$) be the set of roots (resp.\ positive roots) of $\mathring{\mathfrak{g}}$. 
Then, we have 
\[
\Phi_+^{\mathrm{re}} = \{ \alpha + n \delta \mid \alpha \in \mathring{\Phi}_+, n \in \mathbb{Z}_{\geq 0} \} \cup \{-\alpha + n\delta \mid \alpha \in \mathring{\Phi}_+, n \in \mathbb{Z}_{>0} \}. 
\]

Let $U_q(\mathfrak{g})$ be the quantum group associated with the root datum above, 
which is a $\mathbb{Q}(q)$-algebra generated by $e_i, f_i \ (i \in I)$ and $q^h \ (h \in \mathsf{P}^{\lor})$.
For $n \in \mathbb{Z}_{\geq 0}$, we define
\[
[n] = (q^n-q^{-n})/(q-q^{-1}), \ [n]! = [n] [n-1] \cdots [1]. 
\]
For $\alpha \in \Phi_+^{\mathrm{re}}$, we define
\[
q_{\alpha} = q^{(\alpha,\alpha)/2}, \ [n]_{\alpha} = (q_{\alpha}^n - q_{\alpha}^{-n})/(q_{\alpha}-q_{\alpha}^{-1}),\ [n]_{\alpha}! = [n]_{\alpha} [n-1]_{\alpha} \cdots [1]_{\alpha}.
\]
For $i \in I$, we write 
\[
t_i = q^{\frac{(\alpha_i,\alpha_i)}{2}h_i}, \ q_i = q_{\alpha_i}, \ [n]_i = [n]_{\alpha_i}, \ [n]_i! = [n]_{\alpha_i}!. 
\]
By setting $\wt e_i = \alpha_i, \wt f_i = -\alpha_i$ and $\wt q^h = 0$, the quantum group $U_q(\mathfrak{g})$ decomposes into a direct sum of weight spaces with weights in $\mathsf{Q}$. 

Let $\overline{(\cdot)}$ be the $\mathbb{Q}$-algebra involution of $U_q(\mathfrak{g})$ defined by 
\[
\overline{q^h} = q^{-h}, \ \overline{f_i} = f_i, \ \overline{e_i} = e_i, \ \overline{q} = q^{-1}.  
\]
The canonical basis, or the global basis, is a basis of $U_q^-(\mathfrak{g})$ fixed by this involution. 

We introduce a homomorphism $r \colon U_q^-(\mathfrak{g}) \to U_q^-(\mathfrak{g}) \otimes_{\mathbb{Q}(q)} U_q^-(\mathfrak{g})$ following \cite[Chapter 1]{MR2759715}. 
The tensor product $U_q^-(\mathfrak{g}) \otimes_{\mathbb{Q}(q)} U_q^-(\mathfrak{g})$ is regarded as a $\mathbb{Q}(q)$-algebra with multiplication 
\begin{equation} \label{eq:twistedmultiplication}
(u_1 \otimes u_2)(u'_1 \otimes u'_2) = q^{-(\wt u_2, \wt u'_1)} u_1u'_1 \otimes u_2u'_2
\end{equation}
for weight vectors $u_1, u_2, u'_1, u'_2 \in U_q^-(\mathfrak{g})$.
The homomorphism $r$ is defined as the homomorphism of $\mathbb{Q}(q)$-algebras satisfying $r(f_i) = f_i \otimes 1 + 1 \otimes f_i \ (i \in I)$. 
For $\alpha, \beta \in \mathsf{Q}_+$, 
let 
\[
r_{\alpha,\beta} \colon U_q^-(\mathfrak{g})_{-(\alpha+\beta)} \to U_q^-(\mathfrak{g})_{-\alpha} \otimes_{\mathbb{Q}(q)} U_q^-(\mathfrak{g})_{-\beta}
\]
be the entry of $r$ for the weight space decomposition. 
By \cite[Chapter 1]{MR2759715}, there exists a nondegenerate symmetric $\mathbb{Q}(q)$-bilinear form $(\cdot,\cdot)$ on $U_q^-(\mathfrak{g})$ determined by 
\[
(1,1) = 1, \ (f_i,f_j) = \delta_{i,j}/(1-q_i^2), \ (xx',y) = (x\otimes x', r(y)),
\]
for $i,j \in I$ and $x,x',y \in U_q^-(\mathfrak{g})$. 
The dual basis of the canonical basis with respect to this bilinear form is called the dual canonical basis. 

For $i \in I$, let $r_i$ and ${}_ir$ be $\mathbb{Q}$-linear endomorphisms of $U_q^-(\mathfrak{g})$ determined by 
\[
r_{\alpha_i,\beta-\alpha_i} (u) = f_i \otimes {}_ir (u), \ r_{\beta-\alpha_i,\alpha_i} (u) = r_i(u) \otimes f_i, 
\]
for $\beta \in \mathsf{Q}_+$ and $u \in U_q^-(\mathfrak{g})_{-\beta}$. 
We have 
\[
(x,f_iy) = \frac{1}{1-q_i^2}({}_ir(x),y), \ (x,yf_i) = \frac{1}{1-q_i^2} (r_i(x),y) \ (x,y \in U_q^-(\mathfrak{g})). 
\]
By \cite[Proposition 3.1.6]{MR2759715}, we have 
\begin{equation} \label{eq:bracket}
e_iu - ue_i = \frac{r_i(u)t_i - t_i^{-1}{}_ir(u)}{q_i-q_i^{-1}} \ (u \in U_q^-(\mathfrak{g})). 
\end{equation}

Let $c$ be the $\mathbb{Q}$-linear automorphism of $U_q^-(\mathfrak{g})$ defined by 
\[
(c(x),y) = \overline{(x,\overline{y})} \ (x, y \in U_q^-(\mathfrak{g})).  
\]
For $x \in U_q^-(\mathfrak{g})_{-\beta}$ with $\beta = \sum_{i \in I} k_i \alpha_i \ (k_i \geq 0)$, 
our $c(x)$ is $\prod_{i \in I}(-q_i^2)^{k_i}$-multiple of $\sigma(x)$ in \cite[3.1]{MR2914878}. 
By \cite[Proposition 3.6]{MR2914878}, we have 
\[
c(xy) = q^{(\beta,\gamma)}c(y)c(x) \ (x \in U_q^-(\mathfrak{g})_{-\beta}, y \in U_q^-(\mathfrak{g})_{-\gamma}). 
\]
By definition, $c$ fixes every element of the dual canonical basis. 

\subsection{Quiver Hecke algebras} \label{sub:quiverHecke}

Let $\mathbf{k}$ be a unital commutative ring. 
We fix a family of polynomials $Q = (Q_{i,j}(u,v))_{i,j \in I} \in \mathbf{k}[u,v]^{I \times I}$ of the form
\[
Q_{i,j}(u,v) = \begin{cases}
t_{i,j}u^{-a_{i,j}} + t_{j,i}v^{-a_{j,i}} + \sum_{a,b >0, a(\alpha_i,\alpha_i)+ b(\alpha_j,\alpha_j)= - 2(\alpha_i,\alpha_j)} t_{i,j;a,b} u^av^b & \text{if $i \neq j$}, \\
0 & \text {if $i = j$}
\end{cases}
\]
for some $t_{i,j} \in \mathbf{k}^{\times}$ and $t_{i,j;a,b} \in \mathbf{k}$ satisfying $Q_{i,j}(u,v) = Q_{j,i}(v,u)$. 

Note that, if $X_N^{(1)} = A_1^{(1)}$, the polynomial $Q_{0,1}$ is of the form
\[
Q_{0,1}(u,v) = t_{0,1}u^2 + suv + t_{1,0} v^2 
\]
for some $t_{0,1}, t_{1,0} \in \mathbf{k}^{\times}$ and $s = t_{0,1;1,1} = t_{1,0;1,1} \in \mathbf{k}$. 
If $X_N^{(1)} \neq A_1^{(1)}$ and $a_{i,j} <0$, the polynomial $Q_{i,j}$ is of the form 
\[
Q_{i,j}(u,v) = t_{i,j}u^{-a_{i,j}} + t_{j,i} v^{-a_{j,i}}, 
\] 
since either $-a_{i,j}$ or $-a_{j,i}$ is $1$. 
If $a_{i,j} = 0$, we have $Q_{i,j}(u,v) = t_{i,j} = t_{j,i}$. 

\begin{definition}
Let $\beta \in \mathsf{Q}_+$. 
Put $n = \height \beta$ and $I^{\beta} = \{ \nu \in I^n \mid \alpha_{\nu_1} + \cdots + \alpha_{\nu_n} = \beta \}$. 
The quiver Hecke algebra $R_Q(\beta)$ is a graded $\mathbf{k}$-algebra defined by the following generators and relations:  
\begin{itemize}
\item The generators are 
\[
e(\nu) \ (\nu \in I^{\beta}), \ x_k \ (1 \leq k \leq n),\ \tau_k \ (1 \leq k \leq n-1). 
\]
\item The relations are 
\begin{align*}
  & e(\nu)e(\nu') = \delta_{\nu, \nu'} e(\nu), \  \sum_{\nu \in I^{\beta}} e(\nu) = 1, \\
  & x_k e(\nu) = e(\nu) x_k, \ x_k x_{l} = x_{l} x_k, \\
  & \tau_k e(\nu) = e(s_k(\nu)) \tau_k \ (1 \leq k \leq n-1), \ \tau_k \tau_{l} = \tau_{l} \tau_k \ (1 \leq k, l \leq n-1, \lvert k-l \rvert \geq 2), \\
  & (\tau_k x_{k+1} - x_k \tau_k)e(\nu) = (x_{k+1} \tau_k - \tau_k x_k)e(\nu) = \delta_{\nu_k, \nu_{k+1}} e(\nu) \; (1 \leq k \leq n-1), \\
  & \tau_k^2 e(\nu) = Q_{\nu_k, \nu_{k+1}} (x_k, x_{k+1}) e(\nu) \; (1 \leq k \leq n-1), \\
  & (\tau_{k+1} \tau_k \tau_{k+1} - \tau_k \tau_{k+1} \tau_k) e(\nu) = \overline{Q}_{\nu_k, \nu_{k+1}, \nu_{k+2}}(x_k, x_{k+1}, x_{k+2})e(\nu) \; (1 \leq k \leq n-2), 
\end{align*}  
where 
\[
\overline{Q}_{i,i',i''}(u,u',u'') = \begin{cases} 
\dfrac{Q_{i,i'}(u,u')-Q_{i,i'}(u'',u')}{u-u''} & \text{if $i= i'' \neq i'$}, \\
0 & \text{otherwise}. 
\end{cases}
\] 
\item The degree is given by 
\[
\deg e(\nu) = 0, \ \deg x_k e(\nu) = (\alpha_{\nu_k},\alpha_{\nu_k}), \ \deg \tau_k e(\nu) = -(\alpha_{\nu_k},\alpha_{\nu_{k+1}}). 
\]
\end{itemize}
\end{definition}

\begin{lemma}[{\cite[Lemma 2.2]{MR3384473}}] \label{lem:modifyQ}
Let $(c_{i,j})_{i,j \in I}$ be a symmetric matrix with entries in $\mathbf{k}^{\times}$. 
We define a family of polynomials $Q' = (Q'_{i,j}(u,v))$ by $Q'_{i,j}(u,v) = c_{i,j}^2Q_{i,j}(c_{i,i}u,c_{j,j}v)$. 
Then, we have an isomorphism of graded $\mathbf{k}$-algebras $R_{Q'}(\beta) \to R_{Q}(\beta)$ given by 
\[
e(\nu) \mapsto e(\nu), \ x_k e(\nu) \mapsto c_{\nu_k,\nu_k}^{-1}x_ke(\nu), \tau_l e(\nu) \mapsto c_{\nu_l, \nu_{l+1}}\tau_l e(\nu), 
\]
for $\nu \in I^{\beta}, 1 \leq k \leq \height \beta$ and $1 \leq l \leq \height \beta-1$. 
\end{lemma}

\begin{corollary} \label{cor:Qchoice}
Assume that $\mathbf{k}$ is an algebraically closed field. 
We define a family of polynomials $Q' = (Q'_{i,j}(u,v))$ based on $Q$ as follows:
\begin{itemize}
\item In type $A_1^{(1)}$, 
\[
Q'_{0,1}(u,v) = (u-v)(v-tu), 
\]
where $t\in \mathbf{k}^{\times}$ is a root of the quadratic equation 
\[
t_{0,1}t_{1,0}(1+t)^2 - t_{0,1;1,1}^2 t = 0. 
\]
\item In type $A_N^{(1)} \ (N \geq 2)$, 
\begin{align*}
Q'_{i,i+1}(u,v) &= u-v \ (0 \leq i \leq N-1), \ Q'_{N,0} (u,v) = u-tv, \\
Q'_{i,j}(u,v) &= 1 \ (i,j \in I, a_{i,j} = 0), 
\end{align*}
where $t \in \mathbf{k}^{\times}$ is given by 
\[
t = \prod_{i \in I} \left(-\frac{t_{i+1,i}}{t_{i,i+1}} \right).
\]
Here, we identify $I = \mathbb{Z}/(N+1)\mathbb{Z}$. 
\item Otherwise, 
\[
Q'_{i,j}(u,v) = \begin{cases}
u^{-a_{i,j}} - v^{-a_{j,i}} & \text{if $a_{i,j} <0, 0 \leq i < j \leq N$}, \\
1 & \text{if $a_{i,j}= 0$}. \\
\end{cases}
\]
\end{itemize}
Then, $R_Q(\beta)$ is isomorphic to $R_{Q'}(\beta)$ for any $\beta \in \mathsf{Q}_+$. 
\end{corollary}

\begin{proof}
We apply Lemma \ref{lem:modifyQ} as follows. 

In type $A_1^{(1)}$, let $\zeta \in \mathbf{k}$ be a square root of $-1$, 
and put $c_{0,0} = \zeta t_{0,1;1,1} t/(1+t), \ c_{1,1} = -\zeta t_{0,1}$ and let $c_{0,1} = c_{1,0}$ be a square root of $t_{0,1}^{-2}t_{1,0}^{-1}$. 

In type $A_N^{(1)}$, we put
\begin{align*}
c_{i,i} = \prod_{1 \leq j \leq i} \left(-\frac{t_{j-1,j}}{t_{j,j-1}}\right) \ (0 \leq i \leq N), 
\end{align*}
and let $c_{i,i+1} = c_{i+1,i}$ be a square root of $c_{i,i}^{-1}t_{i,i+1}^{-1}$ for $0 \leq i \leq N$.
If $j \not \in \{i-1,i,i+1\}$, let $c_{i,j}$ be a square root of $t_{i,j}^{-1}$. 

Otherwise, it is easy to find appropriate $(c_{i,j})_{i,j \in I}$ since the Dynkin graph is a tree (with multiple edges allowed) and 
each $Q_{i,j}(u,v)$ is whether constant or of the form $t_{i,j}u^{-a_{i,j}} + t_{j,i}v^{-a_{j,i}}$. 
\end{proof}

Hereafter, we make the following choice for $\mathbf{k}$ and $t \in \mathbf{k}^{\times}$. 
For type $A_N^{(1)}$, the pair $(\mathbf{k},t)$ is either
\begin{itemize}
  \item $\mathbf{k}$ is a field and $t \in \mathbf{k}^{\times}$ is arbitrary, 
  \item $\mathbf{k}$ is a Laurent polynomial ring $F[z, z^{-1}]$ over a field $F$, and $t = z$. 
\end{itemize}
For other types, $\mathbf{k}$ is either a field or $\mathbb{Z}$, and $t = 1$.

We mainly consider polynomials $Q$ having the same form as $Q'$ in Corollary \ref{cor:Qchoice} for the parameter $t$ specified above, 
and write $R_Q(\beta) = R(\beta)$ for simplicity. 
When we wish to emphasize the coefficient ring $\mathbf{k}$, 
we write $R_{\mathbf{k}}(\beta)$. 

\subsection{Basic representation theory of quiver Hecke algebras}

In this paper, every $R(\beta)$-module is assumed to be a graded left module unless otherwise specified. 
Since $\mathbf{k}$ is Noetherian and $R(\beta)$ is a finitely generated module over the polynomials ring $\mathbf{k}[x_1, \ldots,x_{\height \beta}]$, the quiver Hecke algebra $R(\beta)$ is also Noetherian. 
We define $\gMod{R} = \bigoplus_{\beta \in \mathsf{Q}_+} \gMod{R(\beta)}$.
For  $X \in \gMod{R(\beta)}$, we set $\wt (X) = -\beta$.
The categories $\gproj{R}, \gmod{R}$ and $\fpd{R}$ are similarly defined. 
For $X = \bigoplus_{\beta \in \mathsf{Q}_+} X_{\beta}, \ Y = \bigoplus_{\beta \in \mathsf{Q}_+} Y_{\beta} \in \gMod{R}$, we define
\[
\hom_{R}(X,Y) = \bigoplus_{\beta \in \mathsf{Q}_+} \hom_{R(\beta)} (X_{\beta},Y_{\beta}). 
\]
Similarly, we use notations $\Hom_R(X,Y), \ext_R^k(X,Y)$ and $\Ext_R^k(X,Y)$. 
We will also use analogous notations for certain quotient algebras of $R(\beta) \ (\beta \in \mathsf{Q}_+)$. 

For each $w \in \mathfrak{S}_n$, fix a reduced expression $w = s_{i_1} \cdots s_{i_l}$ and define 
\[
\tau_w = \tau_{i_1} \cdots \tau_{i_l}. 
\]
While it depends on the choice of the reduced expression in general, 
any choice will suffice for our purposes. 
For $n', n'' \in \mathbb{Z}_{\geq 1}$ with $n' + n'' =n$, 
any reduced expression of the element $w[n',n''] \in \mathfrak{S}_n$ from Subsection \ref{sub:notation} is obtained from another by only using the relation $s_ks_l = s_ls_k$ for $1 \leq k < l \leq n-1$ satisfying $l-k \geq 2$. 
Hence, the element $\tau_{w[n',n'']}$ is independent of the choice of the reduced expression. 

Let $\beta,\gamma \in \mathsf{Q}_+$ and put $m = \height (\beta), n = \height (\gamma)$. 
We define the idempotent $e(\beta, \gamma)$ of $R(\beta + \gamma)$ by
\[
e(\beta, \gamma) = \sum_{\nu \in I^{\beta}, \nu' \in I^{\gamma}} e(\nu, \nu'). 
\]
For multiple $\beta_1, \ldots, \beta_r \in \mathsf{Q}_+$, the idempotent $e(\beta_1, \ldots,\beta_r)$ is defined in the same way. 
The $R(\beta+\gamma)$-module $R(\beta+ \gamma)e(\beta,\gamma)$ is also a right $(R(\beta) \otimes R(\gamma))$-module as follows:
\begin{align*}
 &ue(\beta,\gamma) (e(\nu)\otimes e(\nu')) = ue(\nu,\nu') \ (\nu \in I^{\beta}, \nu' \in I^{\gamma}), \\
&ue(\beta,\gamma) (x_k \otimes 1) = ue(\beta,\gamma)x_k \ (1 \leq k \leq m), \\
 &ue(\beta,\gamma) (1 \otimes x_k) = ue(\beta,\gamma) x_{k+m} \ (1 \leq k \leq n), \\
 &ue(\beta,\gamma) (\tau_k \otimes 1) = ue(\beta,\gamma)\tau_k \ (1 \leq k \leq m-1), \\ 
 &ue(\beta, \gamma) (1 \otimes \tau_k) = ue(\beta,\gamma) \tau_{k+m} \ (1\leq k \leq n-1). 
\end{align*}
It is both projective as a left $R(\beta + \gamma)$-module and free as a right $(R(\beta) \otimes R(\gamma))$-module. 
Similar property holds for $e(\beta, \gamma)R(\beta + \gamma)$. 
These bimodules yields three functors
\begin{align*}
  \Ind_{\beta, \gamma} &= R(\beta + \gamma)e(\beta, \gamma) \otimes_{R(\beta) \otimes R(\gamma)} ? \colon \gMod{(R(\beta) \otimes R(\gamma))} \to \gMod{R(\beta + \gamma)}, \\
  \Res_{\beta, \gamma} &= \Hom_{R(\beta + \gamma)}(R(\beta + \gamma)e(\beta, \gamma), ?)\colon \gMod{R(\beta + \gamma)} \to \gMod{(R(\beta) \otimes R(\gamma))}, \\
  \Coind_{\beta,\gamma} &= \Hom_{R(\beta)\otimes R(\gamma)} (e(\beta,\gamma)R(\beta+\gamma), ?) \colon \gMod{(R(\beta)\otimes R(\gamma))} \to \gMod{R(\beta+\gamma)}. 
\end{align*}
We have two adjoint pairs, $(\Ind_{\beta, \gamma}, \Res_{\beta, \gamma})$ and $(\Res_{\beta,\gamma}, \Coind_{\beta,\gamma})$. 
These three functors are exact and send projective modules to projective ones. 
Hence, the adjunctions also hold for $\Ext^k$ for any $k \geq 0$.  
For multiple $\beta_1, \ldots, \beta_r \in \mathsf{Q}_+$, the functors $\Ind_{\beta_1, \ldots, \beta_r}, \Res_{\beta_1, \ldots, \beta_r}$ and $\Coind_{\beta_1,\ldots,\beta_r}$ are defined in the same way and they have analogous properties. 

We usually write $M \circ N$ instead of $\Ind_{\beta, \gamma} (M \otimes N)$ and refer to it as the convolution product of $M$ and $N$.
It gives a monoidal structure on $\gMod{R}$ with the unit object $\mathbf{k} \in \gMod{R(0)}$.
When $M \otimes N$ is regarded as a subspace of $M \circ N$, it is denoted by $M \boxtimes N$. 
For $u \in M, v\in N$, the element $u \otimes v \in M \boxtimes N$ is denoted by $u \boxtimes v$. 
Similarly, when $R(\alpha) \otimes R(\beta)$ is regarded as a subspace of $R(\alpha+\beta)$, it is denoted by $R(\alpha) \boxtimes R(\beta)$.

There is a $\mathbf{k}$-algebra anti-involution $\varphi$ of $R(\beta)$ that fixes all the generators $e(\nu), x_k$ and $\tau_k$. 
Using it, we get a duality functor $D$ on $\gmod{R(\beta)}$ given by $D(M) = \Hom_{\mathbf{k}}(M, k)$, on which $R(\beta)$ acts by 
\[
\text{$(a f)(m) = f(\varphi(a)m)$ for $a \in R(\beta), f \in D(M), m \in M$. } 
\]
The $d$-th homogeneous component of $D(M)$ is $D(M)_d = \Hom_{\mathbf{k}}(M_{-d}, \mathbf{k})$. 
A module $M \in \gmod{R(\beta)}$ is said to be self-dual if $DM \simeq M$. 

In the rest of this subsection, assume that $\mathbf{k}$ is a field. 
It is known that every simple module in $\gMod{R}$ is finite-dimensional, absolutely-irreducible, and is a grading shift of some self-dual simple module. 
For $M \in \gMod{R(\beta)}, N \in \gMod{R(\gamma)}$, there is an isomorphism
\[
\Coind_{\beta,\gamma}(M \otimes N) \simeq q^{(\beta,\gamma)} \Ind_{\gamma,\beta} (N \otimes M) = q^{(\beta,\gamma)} N\circ M, 
\]
see \cite[Theorem 2.2]{MR2822211}. 
Since $\Res_{\beta,\gamma} \circ D \simeq (D \otimes D)\circ \Res_{\beta,\gamma}$, this implies
\[
D(M \circ N) \simeq q^{(\beta,\gamma)} DN \circ DM
\]
for $M \in \gmod{R(\beta)}$ and $N \in \gmod{R(\gamma)}$. 

We recall the categorification of $U_q^-(\mathfrak{g})$ proved in \cite{MR2525917,MR2763732}. 
We refer the reader to \cite[Theorem 2.5.2]{murata2025affinehighestweightstructures} for the following version. 

\begin{theorem}[{\cite{MR2525917,MR2763732}}] \label{thm:categorification}
Assume that $\mathbf{k}$ is a field. 
There is an isomorphism of $\mathbb{Q}(q)$-algebras
\[
\chi \colon K(\gMod{R})_{\mathbb{Q}(q)} \to U_q^-(\mathfrak{g}), 
\] 
such that $\chi(R(\alpha_i)) = f_i \ (i \in I)$,
where the multiplication on $K(\gMod{R})_{\mathbb{Q}(q)}$ is induced by the convolution product. 
Moreover, it satisfies the following properties: 
\begin{enumerate}
  \item for $\alpha, \beta \in \mathsf{Q}_+$ and $X \in \gMod{R(\alpha+\beta)}$, we have $(\chi \otimes \chi)(\Res_{\alpha,\beta}X) = r_{\alpha,\beta}(\chi(X))$;
  \item for a finite dimensional module $X \in \gmod{R}$, we have $\chi(DX) = c(\chi(X))$;
  \item for $X \in \fpd{R}$ and $Y \in \gMod{R}$, we have $\Extform{X}{Y} = (\overline{\chi(X)}, \chi(Y))$.   
\end{enumerate}
\end{theorem}

Let $\beta \in \mathsf{Q}_+$, and put $n = \height \beta$. 
We define a $\mathbb{Q}(q)$-linear map $\theta \colon U_q^-(\mathfrak{g})_{-\beta} \to \Map(I^{\beta},\mathbb{Q}(q))$ by
\[
\theta(u)(\nu) = (f_{\nu_1}\cdots f_{\nu_n}, u) \ (u \in U_q^-(\mathfrak{g}), \nu \in I^{\beta}). 
\]
Since the vectors $f_{\nu_1} \cdots f_{\nu_n} \ (\nu \in I^{\beta})$ span $U_q^-(\mathfrak{g})$ and the bilinear form is nondegenerate, $\theta$ is injective. 
For $M \in \gMod{R(\beta)}$, we define $\ch_q(M) \in \Map (I^{\beta},\mathbb{Z}((q)))$ by 
\[
\ch_q(M) (\nu) = \dim_q e(\nu)M \ (\nu \in I^{\beta}). 
\]
It induces a $\mathbb{Z}[q,q^{-1}]$-linear map
\[
\ch_q \colon K(\gMod{R}) \to \Map(I^{\beta},\mathbb{Z}((q))). 
\]

\begin{corollary} \label{cor:formalcharacter}
Let $\beta \in \mathsf{Q}_+$. 
For any $M \in \gMod{R(\beta)}$, 
we have 
\[
\ch_q(M) = \theta(\chi(M)). 
\]
In particular, the map $\ch_q \colon K(\gMod{R}) \to \Map(I^{\beta},\mathbb{Z}[q,q^{-1}])$ is injective.  
\end{corollary}

\begin{proof}
Put $n = \height \beta$. 
For any $\nu \in I^{\beta}$, 
note that $R(\beta)e(\nu) \simeq R(\alpha_{\nu_1}) \circ \cdots \circ R(\alpha_{\nu_n})$ and 
\[
\overline{\chi(R(\beta)e(\nu))} =  \overline{f_{\nu_1} \cdots f_{\nu_n}} = f_{\nu_1} \cdots f_{\nu_n}. 
\]
Hence, we have 
\begin{align*}
\ch_q(M)(\nu) &= \dim_q e(\nu)M = \dim_q \Hom_R(R(\beta)e(\nu),M) = \dim_q \Hom_R(R(\alpha_{\nu_1})\circ \cdots \circ R(\alpha_{\nu_n}),M) \\
&= (f_{\nu_1} \cdots f_{\nu_n},\chi(M)) =  \theta(\chi(M))(\nu), 
\end{align*}
by Theorem \ref{thm:categorification}. 
\end{proof}

For $i \in I$ and $n \in \mathbb{Z}_{\geq 1}$, the algebra $R(n\alpha_i)$ is the nil-Hecke algebra. 
There is a unique self-dual simple $R(n\alpha_i)$-module $L(i^n)$ up to isomorphism. 
Let $P(i^n)$ denote the projective cover of $L(i^n)$.
It is known $L(i^n) \simeq q_i^{n(n-1)/2}L(i)^{\circ n}$. 

\section{Imaginary strata}
\subsection{Convex orders and cuspidal modules}

In this subsection, assume that $\mathbf{k}$ is a field. 

\begin{definition}[{\cite[Definition 1.8, Remark 1.9]{MR3542489}, \cite[Definition 1.16]{MR3771147}}] \label{def:convexpreorder}
Let $V$ be an $\mathbb{R}$-vector space, and $X$ be a subset of $V$ such that $\mathbb{R} x \cap X = \{x\}$ for any $x \in X$. 
A total order $\preceq$ on $X$ is said to be convex if, 
for any $A, B \subset X$ satisfying $\alpha \prec \beta$ for all $\alpha \in A$ and $\beta \in B$, 
we have $\spn_{\mathbb{Q}_{\geq 0}}A \cap \spn_{\mathbb{Q}_{\geq 0}}B = \{0\}$. 
\end{definition}

Let $\Phi_+^{\mathrm{min}} = \Phi_+^{\mathrm{re}} \cup \{\delta\}$. 
Let $p \colon \mathsf{Q} \to \mathring{\mathsf{Q}}$ denote the projection. 

\begin{lemma}[{\cite[Lemma 3.7]{MR3694676}}]\label{lem:coarsetype}
For a convex order $\preceq$ on $\Phi_+^{\mathrm{min}}$, 
there exists a unique element $w  \in \mathring{W}$ such that $p(\Phi_{+, \succ \delta}^{\mathrm{min}}) \subset w \mathring{\mathsf{Q}}_+, p(\Phi_{+, \prec \delta}^{\mathrm{min}}) \subset w \mathring{\mathsf{Q}}_-$. 
\end{lemma}

\begin{definition} [{\cite[Definition 2.9]{MR3874704}}]
For a convex order $\preceq$ on $\Phi_+^{\mathrm{min}}$, we call the element $w$ in Lemma \ref{lem:coarsetype} the coarse type of $\preceq$. 
\end{definition}

It is known that there exists a convex order of coarse type $w$ for any $w \in \mathring{W}$; see \cite[Lemma 6.5]{murata2025affinehighestweightstructures} for instance. 

\begin{definition}
Let $\preceq$ be a convex order, and let $\alpha \in \Phi_+^{\mathrm{min}}, n \in \mathbb{Z}_{\geq 0}$. 
A simple $R(n\alpha)$-module $L$ is $\preceq$-cuspidal if, for any $\beta, \gamma \in \mathsf{Q}_+$ satisfying $\beta + \gamma = n\alpha$ and $\Res_{\beta,\gamma} L \neq 0$, 
we have 
\[
\beta \in \spn_{\mathbb{Z}_{\geq 0}} \Phi_{+,\preceq \alpha}^{\mathrm{min}}, \ \gamma \in \spn_{\mathbb{Z}_{\geq 0}} \Phi_{+,\succeq \alpha}^{\mathrm{min}}.  
\]
\end{definition}

Note that, when $\alpha = \delta$, the notion of $\preceq$-cuspidal $R(n\delta)$-module depends only on the coarse type $w$ of $\preceq$. 
Hence, we may call them $w$-cuspidal $R(n\delta)$-modules without ambiguity. 

For $w \in W$, let $\Xi_n^w$ denote the complete set of self-dual $w$-cuspidal simple $R(n\delta)$-modules up to isomorphism. 
Put $\Xi^w = \sqcup_{n \geq 0} \Xi_n^w$. 
For a convex order $\preceq$ of coarse type $w$ and $\beta \in \mathsf{Q}_+$, we define a set $\Omega^{\preceq}(\beta)$ as 
\[
\left\{ (\boldsymbol{c}_-, L, \boldsymbol{c}_+) \mathrel{}\middle|\mathrel{} \begin{gathered} 
  \boldsymbol{c}_- \colon \Phi_{+, \prec \delta}^{\text{min}} \to \mathbb{Z}_{\geq 0}, \ L \in \Xi^w, \ \boldsymbol{c}_+ \colon \Phi_{+, \succ \delta}^{\text{min}} \to \mathbb{Z}_{\geq 0}, \\
  \text{$\boldsymbol{c}_{-} (\alpha) = 0$ for all but finitely many $\alpha \in \Phi_{+, \prec \delta}^{\mathrm{min}}$}, \\
  \text{$\boldsymbol{c}_{+} (\alpha) = 0$ for all but finitely many $\alpha \in \Phi_{+, \succ \delta}^{\mathrm{min}}$}, \\
  \left(\sum_{\alpha \in \Phi_{+, \prec \delta}^{\text{min}}} \boldsymbol{c}_-(\alpha)\alpha \right) - \wt L + \left( \sum_{\alpha \in \Phi_{+, \succ \delta}^{\text{min}}} \boldsymbol{c}_+(\alpha)\alpha \right) = \beta.    
\end{gathered} \right\}.
\]
Note that when $L \in \Xi_n^w$, we have $\wt L = -n\delta$. 

For $n \in \mathbb{Z}_{\geq 0}$, a multipartition of $n$ is a collection of partitions $\lambda = (\lambda^{(i)})_{i \in \mathring{I}}$ indexed by $\mathring{I}$
such that 
\[
\lvert \lambda \rvert  \coloneqq \sum_{i \in \mathring{I}, k \geq 1} \lambda_k^{(i)} = n. 
\]

\begin{theorem}[{\cite[Corollary 2.17, Theorem 2.19]{MR3542489}}] \label{thm:cuspdecomp}
Let $\preceq$ be a convex order on $\Phi_+^{\mathrm{min}}$. 

(1) For each $\alpha \in \Phi_+^{\mathrm{re}}$ and $n \in \mathbb{Z}_{\geq 0}$, there exists a self-dual $\preceq$-cuspidal simple $R(n\alpha)$-module, uniquely up to isomorphism. 
Let $L^{\preceq}(n\alpha)$ denote this module. 
Moreover, we have 
\[
L^{\preceq}(n\alpha) \simeq q^{n(n-1)(\alpha,\alpha)/4}L^{\preceq}(\alpha)^{\circ n}. 
\]

(2) For $n \geq 0$, the cardinality of $\Xi_n^w$ is equal to the number of multipartitions of $n$. 

(3) For $\beta \in \mathsf{Q}_+$, we have a bijection
 \[
 \Omega^{\preceq}(\beta) \to \{ \text{self-dual simple $R(\beta)$-module} \} / {\simeq}, \  \omega \mapsto L^{\preceq}(\omega), 
 \]
where $L^{\preceq}(\omega)$ is defined as follows. 
Let $\omega = (\boldsymbol{c}_-, L, \boldsymbol{c}_+)$. 
We write 
\begin{align*}
\{\alpha \in \Phi_{+, \prec \delta}^{\mathrm{min}} \mid \boldsymbol{c}_-(\alpha) \neq 0 \} &= \{ \alpha_1^- \prec \alpha_2^- \prec \cdots \prec \alpha_{n_-}^-\}, \\
\{\alpha \in \Phi_{+, \succ \delta}^{\mathrm{min}} \mid \boldsymbol{c}_+(\alpha) \neq 0 \} &= \{ \alpha_1^+ \prec \alpha_{2}^+ \prec \cdots \prec \alpha_{n_+}^+\}. 
\end{align*}
We define $L^{\preceq}(\omega)$ to be the head of $\overline{\Delta}^{\preceq}(\omega)$, where
\[
\overline{\Delta}^{\preceq}(\omega) = L^{\preceq}(\boldsymbol{c}_+(\alpha_{n_+}^+)\alpha_{n_+}^+)\circ \cdots \circ L^{\preceq}(\boldsymbol{c}_+(\alpha_{1}^+)\alpha_{1}^+) \circ L \circ L^{\preceq}(\boldsymbol{c}_-(\alpha_{n_-}^-)\alpha_{n_-}^-) \circ \cdots \circ L^{\preceq}(\boldsymbol{c}_-(\alpha_{1}^-)\alpha_{1}^-). 
\] 
\end{theorem}

\begin{corollary}[{\cite[Corollary 3.34, Corollary 3.45]{murata2025affinehighestweightstructures}}] \label{cor:cuspidaltest}
Let $n \in \mathbb{Z}_{\geq 1}$ and let $L$ be a simple $R^w(n\delta)$-module. 
Then, the following three conditions are equivalent: 
\begin{enumerate}
\item the module $L$ is $w$-cuspidal; 
\item for any $\alpha \in \Phi_+^{\mathrm{re}}$ such that $p(\alpha) \in w\mathring{\Phi}_+$, we have $\Res_{\alpha,n\delta-\alpha} L = 0$;
\item for any $\alpha \in \Phi_+^{\mathrm{re}}$ such that $p(\alpha) \in w\mathring{\Phi}_-$, we have $\Res_{n\delta-\alpha,\alpha} L = 0$.  
\end{enumerate}
\end{corollary}

\subsection{Stratifications} \label{sub:stratifications}

For $w \in \mathring{W}$ and $n \geq 0$, 
let $R^w(n\delta)$ be the quotient algebra of $R(n\delta)$ by the ideal generated by the elements $e(\beta,\gamma) \ (\beta,\gamma \in \mathsf{Q}_+, \beta + \gamma = n\delta)$ such that 
\[
\text{$\beta \not \in \spn_{\mathbb{Z}_{\geq 0}} (\Phi_+^{\mathrm{min}} \cap p^{-1}(w\mathring{\mathsf{Q}}_-))$ or $\gamma \not \in \spn_{\mathbb{Z}_{\geq 0}} (\Phi_+^{\mathrm{min}} \cap p^{-1}(w\mathring{\mathsf{Q}}_+))$}.  
\]

In the rest of this subsection, we assume that $\mathbf{k}$ is a field. 
By definition, a simple $R(n\delta)$-module $L$ is $w$-cuspidal if and only if it descends to an $R^w(n\delta)$-module. 
For $L \in \Xi_n^w$, let $\Delta^w(L)$ be the projective cover of $L$ in $\gMod{R^w(n\delta)}$. 

Let $\preceq$ be a convex order of coarse type $w$. 
For $\alpha \in \Phi_+^{\mathrm{min}}$ and $n \geq 0$, 
let $\Delta^{\preceq}(n\alpha)$ be the projective cover of $L^{\preceq}(n\alpha)$ in the full subcategory of $\gMod{R(n\alpha)}$ consisting of modules whose composition factors are grading shifts of $L^{\preceq}(n\alpha)$. 
Note that this projective cover exists because this full subcategory is equivalent to the category of finitely generated graded modules over the quotient algebra of $R(n\alpha)$ by the ideal generated by all the idempotents of $R(n\alpha)$ that annihilate $L^{\preceq}(n\alpha)$. 
For $\omega = (\boldsymbol{c}_-,L,\boldsymbol{c}_+) \in \Omega^{\preceq}(\beta)$, using the notation in Theorem \ref{thm:cuspdecomp} we define $\Delta^{\preceq}(\omega)$ to be
\[
\Delta^{\preceq}(\boldsymbol{c}_+(\alpha_{n_+}^+)\alpha_{n_+}^+)\circ \cdots \circ \Delta^{\preceq}(\boldsymbol{c}_+(\alpha_{1}^+)\alpha_1^+) \circ \Delta^w(L) \circ \Delta^{\preceq}(\boldsymbol{c}_-(\alpha_{n_-}^-)\alpha_{n_-}^-) \circ \cdots \circ \Delta^{\preceq}(\boldsymbol{c}_-(\alpha_{1}^-)\alpha_1^-). 
\] 

We define $\mathcal{P}(\beta)$ as the set of root partitions of $\beta$, that is, the set of all maps $f \colon \Phi_+^{\mathrm{min}} \to \mathbb{Z}_{\geq 0}$ with finite supports satisfying
\[
\sum_{\alpha \in \Phi_+^{\mathrm{min}}} f(\alpha) \alpha = \beta. 
\]
Let $\leq$ be the bilexicographic order on $\mathcal{P}(\beta)$ with respect to $\preceq$;
that is, for $f, g \in \mathcal{P}(\beta)$, we write $f < g$ if the following two conditions hold: 
\begin{itemize}
  \item there exists $\alpha \in \Phi_+^{\mathrm{min}}$ such that $f(\alpha) < g(\alpha)$ and $f(\alpha') = g(\alpha')$ for all $\alpha' \in \Phi_{+,\prec \alpha}^{\mathrm{min}}$;
  \item there exists $\alpha \in \Phi_+^{\mathrm{min}}$ such that $f(\alpha) < g(\alpha)$ and $f(\alpha') = g(\alpha')$ for all $\alpha' \in \Phi_{+,\succ \alpha}^{\mathrm{min}}$. 
\end{itemize}
We define a map $\rho \colon \Omega^{\preceq}(\beta) \to \mathcal{P}(\beta)$ by sending $\omega = (\boldsymbol{c}_-,L,\boldsymbol{c}_+)$ to $f$, where 
\[
f(\alpha) = \begin{cases}
\boldsymbol{c}_-(\alpha) & \text{if $\alpha \prec \delta$}, \\
\boldsymbol{c}_+(\alpha) & \text{if $\alpha \succ \delta$}, \\
n & \text{if $\alpha = \delta$ and $L \in \Xi_n^w$}. 
\end{cases}
\]

\begin{theorem}[{\cite{MR3694676}, \cite[Theorem 6.10]{murata2025affinehighestweightstructures}}] \label{thm:stratification}
Let $\preceq$ be a convex order on $\Phi_+^{\mathrm{min}}$. 
With respect to the map $\rho \colon \Omega^{\preceq}(\beta) \to \mathcal{P}(\beta)$ and the partial order $\leq$ on $\mathcal{P}(\beta)$, 
the category $\gMod{R(\beta)}$ is stratified in the sense of Kleshchev \cite{MR3335289}. 
For $\omega = (\boldsymbol{c}_-,L,\boldsymbol{c}_+) \in \Omega^{\preceq}(\beta)$, 
the standard module is $\Delta^{\preceq}(\omega)$, and its endomorphism algebra $\End_R(\Delta^{\preceq}(\omega))$ is isomorphic to 
\[
\bigotimes_{\alpha \succ \delta} \End_R (\Delta^{\preceq}(\boldsymbol{c}_+(\alpha)\alpha)) \otimes \End_R (\Delta^w(L)) \otimes \bigotimes_{\alpha \prec \delta} \End_R (\Delta^{\preceq}(\boldsymbol{c}_-(\alpha)\alpha)). 
\] 
For $\alpha \in \Phi_+^{\mathrm{re}}$ and $n \geq 1$, 
the algebra $\End_R(\Delta^{\preceq}(n\alpha))$ is isomorphic to the algebra of symmetric polynomials in $n$-variables $\mathbf{k}[z_1, \ldots, z_n]^{\mathfrak{S}_n}$, 
where $\deg z_k = (\alpha,\alpha) \ (1 \leq k \leq n)$. 
Furthermore, we have an isomorphism
\[
\Delta^{\preceq}(\alpha)^{\circ n} \simeq \Delta^{\preceq}(n\alpha)^{\oplus [n]_{\alpha}!}.  
\]
\end{theorem}

This theorem suggests that a detailed investigation into the algebras $R^w(n\delta)$ is essential. 

\begin{corollary}\label{cor:maximalext}
For $\alpha \in \Phi_+^{\mathrm{re}}$, we have
\[
\chi(\Delta^{\preceq}(\alpha)) = \frac{1}{1-q_{\alpha}^2} \chi(L^{\preceq}(\alpha)). 
\]
\end{corollary}

\begin{proof}
It follows from 
\[
\dim_q \End_R(\Delta^{\preceq}(\alpha)) = \frac{1}{1-q_{\alpha}^2}. 
\]
\end{proof}

\begin{corollary}
For $\omega \in \Omega^{\preceq}(\beta)$, the projective dimension of $\Delta^{\preceq}(\omega)$ is finite. 
Hence, $\Extform{\Delta^{\preceq}(\omega)}{M}$ is well-defined for any $M \in \gMod{R(\beta)}$. 
\end{corollary}

\begin{corollary}
For $\beta \in \mathsf{Q}_+$, the set $\{\chi(\Delta^{\preceq}(\omega)) \mid \omega \in \Omega^{\preceq}(\beta) \}$ is a basis of $U_q^-(\mathfrak{g})_{-\beta}$. 
\end{corollary}

\subsection{Reduction to coarse type $e$}

In this subsection, assume that $\mathbf{k}$ is a field. 
We prove that, for any $w \in \mathring{W}$ and $n \geq 0$, the algebra $R^w(n\delta)$ is graded Morita-equivalent to $R^e(n\delta)$. 
To this end, we use reflection functors constructed in \cite{murata2025diagrammaticapproachreflectionfunctors}, which we recall now. 

\begin{definition}
Let $w,v\in W$, and $\beta \in \mathsf{Q}_+$. 
We define quotient algebras
\[
{}_wR(\beta) = R(\beta)/{}_wJ, \ R_v(\beta) = R(\beta)/J_v, \ {}_wR_v(\beta) = R(\beta)/({}_wJ + J_v), 
\]
where
\begin{align*}
{}_wJ &= \langle e(\nu) \mid \text{$\nu \in I^{\beta}, \alpha_{\nu_1}+ \cdots + \alpha_{\nu_k} \not \in w\mathsf{Q}_+$ for some $1 \leq k \leq \height \beta$} \rangle, \\
J_v &= \langle e(\nu) \mid \text{$\nu \in I^{\beta}, \alpha_{\nu_{\height \beta}}+ \cdots + \alpha_{\nu_k} \not \in v\mathsf{Q}_+$ for some $1 \leq k \leq \height \beta$} \rangle. \\
\end{align*}
\end{definition}

\begin{theorem}[{\cite[Theorem 3.2.10, 4.1.1, 4.3.2]{murata2025diagrammaticapproachreflectionfunctors}}] \label{thm:reflectionfunctor}
Let $i \in I$. 

(1) There exists an equivalence of graded monoidal categories 
\[
\mathcal{S}_i \colon \gMod{{}_{s_i}R} \to \gMod{R_{s_i}}.    
\]
Moreover, for $M \in \gMod{{}_{s_i}R}$, we have 
\[
\chi(\mathcal{S}_i(M)) = S_i(\chi(M)), 
\]
where $S_i$ is the $\mathbb{Q}(q)$-algebra automorphism of $U_q(\mathfrak{g})$ defined by 
\begin{align*}
  S_i(q^h) = q^{s_i h},\ S_i (e_i) = - t_i^{-1}f_i,\ S_i(f_i) = - e_i t_i, \\
  S_i(e_j) = \sum_{r+s = -a_{i,j}}(-1)^r q_i^{-r} e_i^{(r)}e_je_i^{(s)} \ (j \neq i), \\
  S_i(f_j) = \sum_{r+s = -a_{i,j}}(-1)^r q_i^{r}f_i^{(s)}f_jf_i^{(r)} \ (j \neq i). 
\end{align*}

(2) Let $w, v \in W$ and assume that $s_iw >w, s_iv < v$. 
Then, 
\[
\gMod{{}_{s_iw}R_{s_iv}} \subset \gMod{{}_{s_i}R}, \ \gMod{{}_wR_v} \subset \gMod{R_{s_i}}, 
\]
and the equivalence $\mathcal{S}_i \colon \gMod{{}_{s_i}R} \to \gMod{R_{s_i}}$ restricts to a graded monoidal equivalence 
\[
\gMod{{}_{s_iw}R_{s_iv}} \simeq \gMod{{}_wR_v}. 
\]
\end{theorem}

Let $i,j \in I$ and assume that $i\neq j$ and $a_{i,j}a_{j,i} < 4$.
Put $h(i,j) = 2, 3, 4$ or $6$ according to whether $a_{i,j}a_{j,i} = 0,1,2$ or $3$. 
By \cite[Theorem 39.4.3]{MR2759715}, $S_i$ and $S_j$ satisfy the braid relation, namely, we have 
\[
S_i S_j S_i \cdots = S_j S_i S_j \cdots, 
\]
where both products have $h(i,j)$ factors. 
Hence, the automorphisms $S_i \ (i \in I)$ give a braid group action on $U_q(\mathfrak{g})$ and we have a well-defined automorphism $S_w$ for any $w \in W$. 

We regard $\mathcal{S}_i$ as a categorification of $S_i$ and refer to it as a reflection functor. 
The braid relations are lifted to natural isomorphisms as follows.

\begin{theorem}[{\cite[Theorem 5.2.1]{murata2025diagrammaticapproachreflectionfunctors}}]
Let $i,j \in I$ and assume that $i\neq j$ and $a_{i,j}a_{j,i} < 4$.
Put $w = s_is_js_i \cdots = s_js_is_j \cdots \in W$, where both products have $h(i,j)$ factors. 
Then, we have a natural isomorphism of functors
\[
\mathcal{S}_i \mathcal{S}_j \mathcal{S}_i \cdots \simeq \mathcal{S}_j \mathcal{S}_i \mathcal{S}_j \cdots \colon \gMod{{}_wR} \to \gMod{R_w}, 
\]
where both functors are compositions of $h(i,j)$ reflection functors. 
\end{theorem}

Hence, for any $w \in W$, we have a well-defined equivalence
\[
\mathcal{S}_w \colon \gMod{{}_wR} \to \gMod{R_w}. 
\]

For $w \in \mathring{W}$, we write 
\[
\gMod{R^w} = \bigoplus_{n \in \mathbb{Z}_{\geq 0}} \gMod{R^w(n\delta)}. 
\]
This is a monoidal category under the convolution product.

\begin{proposition} \label{prop:reduction}
Let $i \in \mathring{I}, w \in \mathring{W}$, and assume $s_i w > w$. 
Then, the reflection functor $\mathcal{S}_i$ induces a monoidal equivalence 
\[
\gMod{R^w} \to \gMod{R^{s_iw}}, 
\]
sending $\gMod{R^w(n\delta)}$ to $\gMod{R^{s_iw}(n\delta)}$ for any $n \in \mathbb{Z}_{\geq 0}$. 
\end{proposition}

In order to prove this proposition, we use the following lemma. 

\begin{lemma} \label{lem:oneroworder}
Let $i \in \mathring{I}, w \in \mathring{W}$, and assume $s_i w > w$.  
Then, there exists an infinite word $i_k \ (k \in \mathbb{Z})$ in $I$ that satisfies the following conditions: 
\begin{enumerate}
\item $i_0 = i$; 
\item every finite successive subword is reduced; 
\item by setting \[
 \beta_k = \begin{cases}
  s_{i_1} s_{i_2} \cdots s_{i_{k-1}} \alpha_{i_k} & k \geq 1, \\
  s_{i_0} s_{i_{-1}} \cdots s_{i_{k+1}} \alpha_{i_k} & k \leq 0,
 \end{cases}
 \]
 for each $k \in \mathbb{Z}$, we have $\Phi_+^{\mathrm{re}} = \{\beta_k\}_{k \in \mathbb{Z}}$; 
\item the total order $\preceq$ on $\Phi_+^{\mathrm{min}}$ defined by 
\[
\beta_1 \prec \beta_2 \prec \cdots \prec \delta \prec \cdots \prec \beta_{-1} \prec \beta_0 
\]
is convex; 
\item $\preceq$ is of coarse type $w$.
\end{enumerate}
\end{lemma}

\begin{proof}
By \cite[Lemma 6.5]{murata2025affinehighestweightstructures}, there exists an infinite word $i_k \ (k \in \mathbb{Z})$ that satisfies (2) - (5). 
Since $s_iw > w$, we have $\alpha_i \in w\mathring{\Phi}_+$. 
Since $\preceq$ is of coarse type $w$, there exists $k \leq 0$ such that 
\[
\alpha_i = \beta_k = s_{i_0} s_{i_{-1}} \cdots s_{i_{k+1}}\alpha_{i_k}. 
\] 
Hence $s_i s_{i_0} \cdots s_{i_{k+1}} = s_{i_0} \cdots s_{i_{k+1}} s_{i_k}$. 
Replacing $(i_0, \ldots, i_{k+1}, i_k)$ with $(i,i_0, \ldots, i_{k+1})$, we obtain an infinite word satisfying all the required conditions. 
\end{proof}

\begin{proof}[Proof of Proposition \ref{prop:reduction}]
Fix arbitrary $n \in \mathbb{Z}_{\geq 0}$. 
Let $i_k \ (k \in \mathbb{Z})$ be an infinite word as in Lemma \ref{lem:oneroworder}, and let $\preceq$ be the associated convex order. 
Since $\preceq$ is of coarse type $w$, we have $\gMod{R^w(n\delta)} = \gMod{R^{\preceq}(n\delta)}$.
By the same argument as in \cite[Proposition 6.11]{murata2025affinehighestweightstructures}, 
taking $m \geq 0$ large enough for $n$ we have $R^{\preceq}(n\delta) = {}_x R_y(n\delta)$, 
where $x = s_{i_0} s_{i_{-1}} \cdots s_{i_{-m}}, y = s_{i_1} \cdots s_{i_m}$. 

Let $j_k = i_{k-1} \ (k \in \mathbb{Z})$ and define a convex order $\preceq'$ for the infinite word $(j_k)$ as in Lemma \ref{lem:oneroworder}.
Then it is of coarse type $s_iw$, hence $R^{s_iw}(n\delta) = R^{\preceq'}(n\delta)$. 
By taking a larger $m$ if necessary, we have $R^{\preceq'}(n\delta) = {}_{s_ix}R_{s_iy}(n\delta)$. 
Since $i_0 = i$, we have $s_i x < x$ and $s_iy > y$. 
Hence, the assertion follows from Theorem \ref{thm:reflectionfunctor} (2). 

\end{proof}

\begin{corollary} \label{cor:stdcoarsetype}
Let $w \in \mathring{W}$. 
We have an equivalence of monoidal graded categories
\[
\gMod{R^w} \simeq \gMod{R^e}. 
\]
\end{corollary}

Hence, the analysis of imaginary strata is reduced to the case where the convex order is of coarse type $e$. 

\subsection{Minuscule imaginary modules} \label{sub:minuscule}

In this subsection, assume that $\mathbf{k}$ is a field. 
We introduce two families of $R^e(\delta)$-modules. 

\begin{lemma} \label{lem:realrootmodule}
Let $i \in \mathring{I}$ and $n \in \mathbb{Z}_{\geq 1}$. 
Choose a convex order $\preceq$ of coarse type $e$. 
The modules $L^{\preceq}(n\delta -\alpha_i)$ and $\Delta^{\preceq}(n\delta-\alpha_i)$ do not depend on the choice of $\preceq$. 
\end{lemma}

\begin{proof}
For $L^{\preceq}(n\delta-\alpha_i)$, it is proved in \cite[Proposition 6.12]{murata2025affinehighestweightstructures}. 
By definition in Subsection \ref{sub:stratifications}, $\Delta^{\preceq}(n\delta-\alpha_i)$ is uniquely determined by $L^{\preceq}(n\delta-\alpha_i)$ and does not depend on $\preceq$.
\end{proof}

Hence, we may write $L^e(n\delta-\alpha_i)$ (resp.\ $\Delta^e(n\delta-\alpha_i)$) for $L^{\preceq}(n\delta-\alpha_i)$ (resp.\ $\Delta^{\preceq}(n\delta-\alpha_i)$). 
Note also that for any convex order $\preceq$, the module $L^{\preceq}(\alpha_i)$ is the unique self-dual simple $R(\alpha_i)$-module $L(i)$ and $\Delta^{\preceq}(\alpha_i) = R(\alpha_i)$. 

\begin{definition} \label{def:minuscule}
Let $i \in \mathring{I}$. 
We define $L_i^e(\delta) = \hd (L^e(\delta-\alpha_i) \circ L(i))$.
\end{definition}

When we wish to emphasize the coefficient field $\mathbf{k}$, we write $L_{i,\mathbf{k}}^e(\delta)$. 

\begin{proposition}[{\cite[Lemma 13.3]{MR3694676}, \cite[Proposition 6.16]{murata2025affinehighestweightstructures}}] \label{prop:minuscule}
For $i \in \mathring{I}$, the module $L_i^e(\delta)$ is simple and belongs to the subcategory $\gMod{R^e(\delta)}$. 
\end{proposition}

\begin{definition} \label{def:minusculeproj}
For $i \in \mathring{I}$, we define $\Delta_i^e(\delta)$ as the projective cover of $L_i^e(\delta)$ in $\gMod{R^e(\delta)}$. 
\end{definition}

\begin{proposition}[{\cite[Theorem 17.1]{MR3694676}, \cite[Theorem 6.22]{murata2025affinehighestweightstructures}}] \label{prop:imaginarystd}
For $i \in \mathring{I}$, there exists a short exact sequence
 \[
 0 \to q_i^2 R(\alpha_i) \circ \Delta^e(\delta-\alpha_i) \xrightarrow{\mathsf{R}} \Delta^e(\delta-\alpha_i) \circ R(\alpha_i) \to \Delta_i^e(\delta) \to 0, 
 \]
 where $\mathsf{R}$ is given by 
 \[
 u \boxtimes v \mapsto \tau_1 \tau_2 \cdots \tau_{\height (\delta-\alpha_i)}(v \boxtimes u) \ (u \in R(\alpha_i),v \in \Delta^e(\delta-\alpha_i)). 
 \]
\end{proposition}

Note that the injectivity of $\mathsf{R}$ follows from \cite[Proposition 3.20]{murata2025affinehighestweightstructures}. 
The following proposition is proved in the same way. 

\begin{proposition}[{cf.\ \cite[Lemma 20.3]{MR3694676}}] \label{prop:psimodule}
For $i \in \mathring{I}$ and $n \in \mathbb{Z}_{\geq 1}$, there exists an injective homomorphism 
\[
q_i^2 R(\alpha_i) \circ \Delta^e(n\delta-\alpha_i) \to \Delta^e(n\delta-\alpha_i) \circ R(\alpha_i), \ u\boxtimes v \mapsto \tau_1 \tau_2 \cdots \tau_{\height (n\delta-\alpha_i)} (v\boxtimes u), 
\]
where $u \in R(\alpha_i), v \in \Delta^e(n\delta-\alpha_i)$. 
Moreover, its cokernel is a projective $R^e(n\delta)$-module. 
\end{proposition}

\subsection{Affine PBW bases} \label{sub:affinePBW}

We recall the construction of imaginary root vectors and affine PBW bases of $U_q(\mathfrak{g})$ \cite{MR2066942,MR3874704}.
For differences in conventions, see Appendix A. 

First, we introduce real root vectors using the quiver Hecke algebra. 
Let $\preceq$ be a convex order on $\minroot$ of coarse type $w$. 

\begin{definition}
  For $\alpha \in \prroot$, we define the real root vector
  \[
  f^{\preceq}_{\alpha} = \chi (\Delta^{\preceq}(\alpha)) \in U_q^-(\mathfrak{g}), 
\] 
and its divided power
\[
 (f_{\alpha}^{\preceq})^{(n)} = \frac{(f_{\alpha}^{\preceq})^n}{[n]_{\alpha}!} \ (n \geq 1). 
\]
\end{definition}

Note that $(f_{\alpha}^{\preceq})^{(n)} = \chi(\Delta^{\preceq}(n\alpha))$ by Theorem \ref{thm:stratification}. 
By \cite[Theorem 9.1]{MR3694676} and \cite[Proposition 6.8]{murata2025affinehighestweightstructures}, 
our $f_{\alpha}^{\preceq}$ coincides with $E_{\alpha}^{\preceq}$ of \cite[Definition 4.5]{MR3874704}. 
More precisely, if $\alpha \prec \delta$ (resp.\ $\alpha \succ \delta$), there exist $i \in \mathring{I}$ and $x \in W$ such that $x\alpha_i = \alpha$ (resp.\ $x^{-1}\alpha_i = \alpha$) and 
$f_{\alpha}^{\preceq} = S_x(f_i)$ (resp.\ $f_{\alpha}^{\preceq} = S_x^{-1}(f_i)$). 
Since $(1-q_i^2)f_i$ is an element of the dual canonical basis and dual canonical basis is compatible with the braid group action (see \cite[Theorem 4.23]{MR2914878} with a different normalization), 
we obtain the former assertion of the following lemma: 

\begin{lemma} \label{lem:realrootvector}
For any positive real root $\alpha$, 
the element $\chi(L^{\preceq}(\alpha)) = (1-q_{\alpha}^2) f_{\alpha}^{\preceq}$ (Corollary \ref{cor:maximalext}) belongs to the dual canonical basis. 
Moreover, we have 
\[
(f_{\alpha}^{\preceq}, f_{\alpha}^{\preceq}) = \frac{1}{1-q_{\alpha}^2}. 
\]
\end{lemma}

\begin{proof}
The former assertion is proved in the discussion above. 
The latter assertion follows from the fact that $(\cdot,\cdot)$ is invariant by the braid group action \cite[Proposition 38.2.1]{MR2759715}. 
\end{proof}



Next, we recall imaginary root vectors. 
Let $i \in \mathring{I}$. 
Put
\[
\widetilde{w\alpha_i} = \begin{cases}
w\alpha_i & \text{if $w\alpha_i \in \mathring{\Phi}_+$}, \\
w\alpha_i + \delta & \text{otherwise}. 
\end{cases}
\]
For each $k \geq 1$, \cite[Corollary 4.6]{MR3874704} shows that the real root vectors $f_{\widetilde{w\alpha_i}}^{\preceq}$ and $f_{k\delta-\widetilde{w\alpha_i}}^{\preceq}$ depend only on the coarse type $w$ (see also Lemma \ref{lem:realrootmodule}). 
Accordingly, we shall write $f_{\widetilde{w\alpha_i}}^w$ and $f_{k\delta-\widetilde{w\alpha_i}}^w$ instead. 
We define
\[
\psi_{i,k}^w = f_{k \delta - \widetilde{w\alpha_i}}^w f_{\widetilde{w\alpha_i}}^w - q_i^2 f_{\widetilde{w\alpha_i}}^w f_{k \delta - \widetilde{w\alpha_i}}^w \ (k \in \mathbb{Z}_{\geq 0}). 
\]
The vectors $\{\psi_{i,k}^w \mid i \in \mathring{I}, k \in \mathbb{Z}_{\geq 0} \}$ mutually commute \cite[Section 4.2]{MR3874704}. 
We introduce $P_{i,k}^w \ (k \geq 0)$ defined by the following recursive identities:  
\begin{align*}
  P_{i,0}^w &= 1, \\
  P_{i,k}^w &= \displaystyle \dfrac{1}{[k]_i} \sum_{s=1}^k q_i^{k-s} \psi_{i,s}^w P_{i,k-s}^w \ (k \geq 1). 
\end{align*}
For $k <0$, we put $P_{i,k}^w = 0$. 
For a partition $\lambda = (\lambda_1, \ldots, \lambda_l)$, we define
\[
S_{i,\lambda}^w = \det (P_{i,\lambda_a -a+b}^w)_{1 \leq a,b \leq l}.
\]
This is the imaginary root vector defined in \cite[Definition 4.8]{MR3874704}. 

Finally, we introduce the PBW basis. 

\begin{definition} 
Let $\preceq$ be a convex order, and let $\beta \in \mathsf{Q}_+$. 
We define $\Upsilon^{\preceq}(\beta)$ as the set \index{$\Upsilon^{\preceq}(\beta)$}
\[
  \left\{ \boldsymbol{c} = (\boldsymbol{c}_-, \boldsymbol{c}_{\delta}, \boldsymbol{c}_+) \mathrel{}\middle|\mathrel{} \begin{gathered} 
    \boldsymbol{c}_- \colon \Phi_{+, \prec \delta}^{\text{min}} \to \mathbb{Z}_{\geq 0},\ \boldsymbol{c}_+ \colon \Phi_{+, \succ \delta}^{\text{min}} \to \mathbb{Z}_{\geq 0}, \\
    \text{$\boldsymbol{c}_{\delta} = (\lambda^{(i)})_{i \in \mathring{I}}$ is a multipartition},  \\
    \text{$\boldsymbol{c}_- (\alpha) = 0$ for all but finitely many $\alpha \in \Phi_{+,\prec \delta}^{\mathrm{min}}$}, \\
    \text{$\boldsymbol{c}_+ (\alpha) = 0$ for all but finitely many $\alpha \in \Phi_{+,\succ \delta}^{\mathrm{min}}$}, \\
    \left(\sum_{\alpha \in \Phi_{+, \prec \delta}^{\text{min}}} \boldsymbol{c}_-(\alpha)\alpha \right) + \left( \sum_{\alpha \in \Phi_{+, \succ \delta}^{\text{min}}} \boldsymbol{c}_+(\alpha)\alpha \right) + \sum_{i \in \mathring{I}} \lvert \lambda^{(i)} \rvert \delta = \beta.    
  \end{gathered} \right\}. 
\]
An element $\boldsymbol{c}$ of $\Upsilon^{\preceq} (\beta)$ is called a Lusztig datum. 
\end{definition}

\begin{definition} 
  Let $\preceq$ be a convex order, $\beta \in \mathsf{Q}_+$ and $\boldsymbol{c} \in \Upsilon^{\preceq}(\beta)$ be a Lusztig datum. 
  Using the notation in Theorem \ref{thm:cuspdecomp},
  we define 
  \[
  f_{\boldsymbol{c}}^{\preceq} = (f_{\alpha_{n_+}^+}^{\preceq})^{(\boldsymbol{c}_+ (\alpha_{n_+}^+))} \cdots (f_{\alpha_{1}^+}^{\preceq})^{(\boldsymbol{c}_+(\alpha_{1}^+))} \left( \prod_{i \in \mathring{I}} S_{i, \lambda^{(i)}}^w \right) (f_{\alpha_{n_-}^-}^{\preceq})^{(\boldsymbol{c}_-(\alpha_{n_-}^-))} \cdots (f_{\alpha_{1}^-}^{\preceq})^{(\boldsymbol{c}_-(\alpha_{1}^-))}. 
  \]
\end{definition}

\begin{theorem} [{\cite[Theorem 3.13]{MR2066942}, \cite[Proposition 4.21]{MR3874704}}] \label{thm:affinePBW}
  Let $\preceq$ be a convex order and let $\beta \in \mathsf{Q}_+$.
  Then $\{ f_{\boldsymbol{c}}^{\preceq} \mid \boldsymbol{c} \in \Upsilon^{\preceq}(\beta) \}$ is a basis of $\quantum{-}_{-\beta}$. 
\end{theorem}

In the rest of this subsection, we prove the following proposition. 

\begin{proposition} \label{prop:almostorth}
Let $i, j \in \mathring{I}$. 
Then, we have
\[
(\overline{P_{i,1}^w},P_{j,1}^w) \in \delta_{i,j} + q \mathbb{Q}[q]_{(q)}.
\]
\end{proposition}

\begin{lemma} \label{lem:almostorth1}
 $(\overline{P_{i,1}^w},P_{j,1}^w) = - \overline{(P_{i,1}^w,P_{j,1}^w)}$. 
\end{lemma}

\begin{proof}
By definition, we need to prove $(\overline{\psi_{i,1}^w},\psi_{j,1}^w) = -\overline{(\psi_{i,1}^w,\psi_{j,1}^w)}$. 
Using the involution $c$ in Subsection \ref{sub:quantumgroups}, we have 
\[
(\overline{\psi_{i,1}^w},\psi_{j,1}^w) = \overline{(\psi_{i,1}^w, c(\psi_{j,1}^w))}, 
\]
where
\[
c(\psi_{j,1}^w) = q^{(\delta-\widetilde{w\alpha_j}, \widetilde{w\alpha_j})}(c(f_{\widetilde{w\alpha_j}}^w)c(f_{\delta-\widetilde{w\alpha_j}}^w) - q_j^{-2} c(f_{\delta-\widetilde{w\alpha_j}}^w)c(f_{\widetilde{w\alpha_j}}^w)).
\]
Since $(1-q_j^2)f_{\widetilde{w\alpha_j}}^w$ and $(1-q_j^2)f_{\delta-\widetilde{w\alpha_j}}^w$ belong to the dual canonical basis, 
they are invariant by $c$.
It follows that 
\begin{align*}
c(f_{\widetilde{w\alpha_j}}^w) &= \frac{(1-q_j^2)}{(1-q_j^{-2})} f_{\widetilde{w\alpha_j}}^w = -q_j^2 f_{\widetilde{w\alpha_j}}^w, \\
c(f_{\delta-\widetilde{w\alpha_j}}^w) &= \frac{(1-q_j^2)}{(1-q_j^{-2})} f_{\delta-\widetilde{w\alpha_j}}^w = -q_j^2  f_{\delta-\widetilde{w\alpha_j}}^w. 
\end{align*}
Hence, we obtain
\[
c(\psi_{j,1}^w) = q_j^{-2} q_j^4 (f_{\widetilde{w\alpha_j}}^wf_{\delta-\widetilde{w\alpha_j}}^w - q_j^{-2}f_{\delta-\widetilde{w\alpha_j}}^wf_{\widetilde{w\alpha_j}}^w) = - \psi_{j,1}^w. 
\]
The lemma is proved. 
\end{proof}

\begin{lemma} \label{lem:almostorth2}
\[
\Lusform{P_{i,1}^w}{P_{j,1}^w}  = \begin{cases}
  \displaystyle \dfrac{1 + q_i^2}{1-q_i^2}  & \text{if $i=j$}, \\
  \displaystyle \dfrac{q}{1-q^2}  & \text{if $a_{i,j} <0$}, \\
  \displaystyle 0 & \text{if $a_{i,j} = 0$}.
\end{cases}
\]
\end{lemma}

\begin{proof}
By \cite[Theorem 4.13]{MR3874704}, we have $P_{i,1}^w = S_w(P_{i,1}^{e})$.
Since $\Lusform{\cdot}{\cdot}$ is invariant by the braid group action \cite[Proposition 38.2.1]{MR2759715}, we may assume $w = e$. 
Now, we perform a calculation following \cite[Proposition 3.19]{MR2066942}. 
We may assume $a_{i,j} \geq -1$ since the bilinear form is symmetric. 

By \cite[Theorem 5.3.2. (3)]{MR1802170}, we have 
\[
[\psi_{i,1}^{e},e_j] = \begin{cases}
\displaystyle [2]_i t_i f_{\delta-\alpha_i}^{e} & \text{if $i=j$}, \\
\displaystyle t_j f_{\delta-\alpha_j}^{e} & \text{if $a_{i,j} = -1$}, \\
\displaystyle 0 & \text{if $a_{i,j} = 0$}.
\end{cases}
\]
Hence, using (\ref{eq:bracket}) we obtain
\begin{align*}
{}_jr(\psi_{i,1}^e) &= 0, \\
r_j(\psi_{i,1}^e) &= \begin{cases}
  \displaystyle -q_i^2 (q_i-q_i^{-1})[2]_i f_{\delta-\alpha_i}^e & \text{if $i=j$}, \\
  \displaystyle -q^2 (q-q^{-1}) f_{\delta-\alpha_j}^e & \text{if $a_{i,j} = -1$}, \\
  \displaystyle 0 & \text{if $a_{i,j} = 0$}.
  \end{cases}
\end{align*}
Since $\psi_{j,1}^e = f_{\delta-\alpha_j}^e f_j - q_j^2 f_j f_{\delta-\alpha_j}^e$, we deduce using Lemma \ref{lem:realrootvector}
\[
\Lusform{\psi_{i,1}^e}{\psi_{j,1}^e} = \frac{1}{1-q_j^2}\Lusform{r_j(\psi_{i,1}^e)}{f_{\delta-\alpha_j}^e} = \begin{cases}
  \displaystyle \frac{q_i}{1-q_i^2}[2]_i & \text{if $i=j$}, \\
  \displaystyle \frac{q}{1-q^2} & \text{if $a_{i,j} = -1$}, \\
  \displaystyle 0 & \text{if $a_{i,j} = 0$}.
  \end{cases}
\]
Hence, 
\[
\Lusform{P_{i,1}^e}{P_{j,1}^e} = \Lusform{\psi_{i,1}^e}{\psi_{j,1}^e} = \begin{cases}
  \displaystyle \dfrac{1 + q_i^2}{1-q_i^2}  & \text{if $i=j$}, \\
  \displaystyle \dfrac{q}{1-q^2}  & \text{if $a_{i,j} = -1$}, \\
  \displaystyle 0 & \text{if $a_{i,j} = 0$}.
\end{cases}
\]

\end{proof}

\begin{proof}[Proof of Proposition \ref{prop:almostorth}]
By Lemma \ref{lem:almostorth1} and Lemma \ref{lem:almostorth2}, we have
\[
\Lusform{\overline{P_{i,1}^w}}{P_{j,1}^w} = -\overline{\Lusform{P_{i,1}^w}{P_{j,1}^w}} = \begin{cases}
  \displaystyle \dfrac{1 + q_i^2}{1-q_i^2}  & \text{if $i=j$}, \\
  \displaystyle \dfrac{q}{1-q^2}  & \text{if $a_{i,j} = -1$}, \\
  \displaystyle 0 & \text{if $a_{i,j} = 0$}.
\end{cases}
\]
The proposition is proved. 
\end{proof}

\subsection{Positivity and Ext-vanishing}

In this subsection, assume that $\mathbf{k}$ is a field. 
By Proposition \ref{prop:imaginarystd}, we have 
\begin{equation} \label{eq:character}
\chi(\Delta_i^e(\delta)) = \psi_{i,1}^e = P_{i,1}^e. 
\end{equation}
Note that this holds regardless of the characteristic of $\mathbf{k}$ or the parameter $t$. 

\begin{lemma} \label{lem:qdimimaginary}
 Let $i,j \in \mathring{I}$. 
 Then, 
 \[
 \dim_q \Hom_{R(\delta)} (\Delta_i^e(\delta), \Delta_j^e(\delta)) \in \delta_{i,j} + q\mathbb{Q}[q]_{(q)}. 
 \]
 In particular, if $i \neq j$, then $L_i(\delta) \not \simeq L_j(\delta)$. 
\end{lemma}

\begin{proof}
By Theorem \ref{thm:categorification} and Theorem \ref{thm:stratification}, we have 
\[
  \dim_q \Hom_{R(\delta)} (\Delta_i^e(\delta), \Delta_j^e(\delta)) = \Lusform{\overline{P_{i,1}^e}}{P_{j,1}^e}. 
\] 
Hence, the lemma follows from Proposition \ref{prop:almostorth}. 
\end{proof}

\begin{corollary} \label{cor:minusclesimple}
The set $\{L_i^e(\delta) \mid i \in \mathring{I} \}$ is a complete set of representatives of self-dual simple $R^e(\delta)$-modules up to isomorphism. 
\end{corollary}

\begin{proof}
By Theorem \ref{thm:cuspdecomp}, there are exactly $\lvert \mathring{I} \rvert$ non-isomorphic self-dual simple $R^e(\delta)$-modules. 
Hence, the assertion follows from Lemma \ref{lem:qdimimaginary}.  
\end{proof}

\begin{lemma} \label{lem:semisimplicity}
For $i, j \in \mathring{I}$, we have 
\[
\ext_{R(\delta)}^1(L_i^e(\delta), L_j^e(\delta)) = 0.
\] 
\end{lemma}

\begin{proof}
Let $K$ be the kernel of the surjection $\Delta_i^e(\delta) \to L_i^e(\delta)$. 
We have an exact sequence
\[ 
\hom_{R(\delta)} (K, L_j^e(\delta)) \to \ext_{R(\delta)}^1 (L_i^e(\delta), L_j^e(\delta)) \to \ext_{R(\delta)}^1 (\Delta_i^e(\delta), L_j^e(\delta)). 
\]
The right term vanishes by Theorem \ref{thm:stratification}. 
Since $\Delta_j^e(\delta)$ is the projective cover of $L_j^e(\delta)$ in $\gMod{R^e(\delta)}$, we have 
\begin{align*}
[K:L_j^e(\delta)]_q &= \dim_q \Hom_{R(\delta)} (\Delta_j^e(\delta), K) \\
&= \dim_q \Hom_{R(\delta)} (\Delta_j^e(\delta), \Delta_i^e(\delta)) - \dim_q \Hom_{R(\delta)} (\Delta_j^e(\delta), L_i^e(\delta)), 
\end{align*}
which belongs to $q \mathbb{Q}[q]_{(q)}$ by Lemma \ref{lem:qdimimaginary}. 
Hence, the left term in the exact sequence vanishes. 
Therefore, the middle term also vanishes, which proves the lemma. 
\end{proof}

\subsection{Imaginary induction and restriction}

\begin{proposition} \label{prop:imaginaryind}
Let $m,n \in \mathbb{Z}_{\geq 1}$. 
\begin{enumerate}
\item For $X \in \gMod{R^e(m\delta)}$ and $Y \in \gMod{R^e(n\delta)}$, we have $X \circ Y \in \gMod{R^e((m+n)\delta)}$. 
\end{enumerate}
Assume that $\mathbf{k}$ is a field. Then, 
\begin{enumerate}
\item[(2)] For $X \in \gproj{R^e(m\delta)}$ and $Y \in \gproj{R^e(n\delta)}$, we have $X \circ Y \in \gproj{R^e((m+n)\delta)}$. 
\item[(3)] For $Z \in \gMod{R^e((m+n)\delta)}$, we have $\Res_{m\delta,n\delta} Z \in \gMod{(R^e(m\delta)\otimes R^e(n\delta))}$. 
\item[(4)] For $Z \in \gproj{R^e((m+n)\delta)}$, we have $\Res_{m\delta,n\delta} Z \in \gproj{(R^e(m\delta)\otimes R^e(n\delta))}$. 
\end{enumerate}
\end{proposition}

\begin{proof}
In this proof, we ignore grading shifts.
(1) follows from the definition of the algebra $R^e(k\delta) \ (k \in \mathbb{Z}_{\geq 0})$. 

Now, assume that $\mathbf{k}$ is a field. 
(3) is a consequence of the cuspidal decomposition \cite{MR3542489};  
To see this, take a simple subquotient $L_1 \otimes L_2$ of $\Res_{m\delta,n\delta}Z$. 
If $\alpha \in \Phi_+^{\mathrm{re}}$ satisfies $\Res_{\alpha,m\delta-\alpha} L_1 \neq 0$, then $\Res_{\alpha,(m+n)\delta - \alpha} Z \neq 0$. 
Since $Z$ is cuspidal, we have $p(\alpha) \in \mathring{\Phi}_-$. 
By Corollary \ref{cor:cuspidaltest}, we see that $L_1$ is an $R^e(m\delta)$-module. 
Similarly, considering $\Res_{n\alpha-\alpha,\alpha}L_2$ we see that $L_2$ is an $R^e(n\delta)$-module, which proves (3). 

Next, we prove (2). 
It is enough to prove 
\[
\Ext_{R^e}^1 (X \circ Y, Z) = 0
\]
for any $Z \in \gMod{R^e((m+n)\delta)}$. 
The left hand side is 
\begin{align*}
&\Ext_R^1 (X \circ Y, Z) \quad \text{since $\gMod{R^e}$ is a Serre subcategory of $\gMod{R}$} \\
&\simeq \Ext_{R(m\delta) \otimes R(n\delta)}^1 (X \otimes Y, \Res_{m\delta,n\delta}Z) \quad \text{by induction-restriction adjunction},
\end{align*}
which is zero by (3) and the assumption. 

(4) is proved in an analogous way using (1) and coinduction-restriction adjunction. 
\end{proof}

Let $n \in \mathbb{Z}_{\geq 1}$.
For $1 \leq k \leq n-1$, set 
\[
\tilde{\tau}_k = e((k-1)\delta) \boxtimes \tau_{w[\height \delta,\height \delta]}e(2\delta) \boxtimes e((n-k-1)\delta) \in R(n\delta). 
\]
For $w \in \mathfrak{S}_n$, 
we choose a reduced expression $w = s_{k_1} \cdots s_{k_l}$ and define $\tilde{\tau}_w \in R(n\delta)$ by 
\[
\tilde{\tau}_w = \tilde{\tau}_{k_1} \cdots \tilde{\tau}_{k_l}. 
\]
We say that a subset $A$ of $\mathfrak{S}_n$ is closed if, for any $w, v \in \mathfrak{S}_n$ satisfying $v \in A$ and $w \leq v$ in the Bruhat order, we have $w \in A$. 

The following lemma will be used repeatedly to study $\Res_{\delta^n} (M_1 \circ \cdots \circ M_n)$,
where $\Res_{\delta^n}$ stands for the restriction to $R(\delta)^{\otimes n}$.   

\begin{lemma} \label{lem:Mackey}
For $M_k \in \gMod{R^e(\delta)} \ (1 \leq k \leq n)$, 
we have 
\[
\Res_{\delta^n} (M_1 \circ \cdots \circ M_n) = \bigoplus_{w \in \mathfrak{S}_n} \tilde{\tau}_w (M_1 \boxtimes \cdots \boxtimes M_n),
\]
as $\mathbf{k}$-subspaces of $M_1 \circ \cdots \circ M_n$. 
Moreover, for any closed subset $A \subset \mathfrak{S}_n$, the subspace
\[
\bigoplus_{w \in A} \tilde{\tau}_w (M_1 \boxtimes \cdots \boxtimes M_n)
\]
is an $R(\delta)^{\otimes n}$ submodule. 
For any maximal element $w \in A$, we have an isomorphism of $R(\delta)^{\otimes n}$-modules
\begin{align*}
M_{w^{-1}(1)} \otimes \cdots \otimes M_{w^{-1}(n)} &\to \bigoplus_{v \in A} \tilde{\tau}_v (M_1 \boxtimes \cdots \boxtimes M_n)\bigg/\bigoplus_{v \in A\setminus \{w\}} \tilde{\tau}_v (M_1 \boxtimes \cdots \boxtimes M_n), \\
u_1 \otimes \cdots \otimes u_n &\mapsto \tilde{\tau}_w (u_{w(1)} \boxtimes \cdots \boxtimes u_{w(n)}) + \bigoplus_{v \in A \setminus \{w\}} \tilde{\tau}_v (M_1 \boxtimes \cdots \boxtimes M_n). 
\end{align*}
\end{lemma}

\begin{proof}
In general, a restriction of an induced module has a Mackey filtration \cite[Proposition 2.18]{MR2525917}. 
In our case, it follows from the definition of $R^e(\delta)$ that most of the subquotients in the Mackey filtration vanish, leaving only those described in the lemma. 
Note also that we do not need any grading shift in the last isomorphism since $\deg (\tilde{\tau}_1 e(\delta,\delta)) = -(\delta,\delta) =0$. 
\end{proof}

For a closed subset $A \subset \mathfrak{S}_n$, let $F_A(M_1 \circ \cdots \circ M_n)$ denote the submodule 
\[
\bigoplus_{w \in A} \tilde{\tau}_w (M_1 \boxtimes \cdots \boxtimes M_n). 
\]
When $A$ is of the form $\{v \in \mathfrak{S}_n \mid v \leq w\}$ (resp.\ $\{v \in \mathfrak{S}_n \mid v < w\}$) for some $w$, 
we write $F_{\leq w} (M_1 \circ \cdots \circ M_n)$ (resp.\ $F_{< w}(M_1 \circ \cdots \circ M_n)$ instead of $F_A(M_1 \circ \cdots \circ M_n)$. 
By the defining relations of the quiver Hecke algebra, these submodules do not depend on the choices of reduced expressions involved in defining $\tilde{\tau}_w$.   

In the rest of this subsection, we assume that $\mathbf{k}$ is a field. 
For $n \in \mathbb{Z}_{\geq 0}$ and $\underline{i} = (i_1, \ldots, i_n) \in \mathring{I}^n$,   
we define 
\[
L^e(\underline{i}) = L_{i_1}^e(\delta) \circ \cdots \circ L_{i_n}^e(\delta), \ \Delta^e(\underline{i}) = \Delta_{i_1}^e(\delta) \circ \cdots \circ \Delta_{i_n}^e(\delta).  
\]
When we wish to emphasize the coefficient field, we write $L_{\mathbf{k}}^e(\underline{i}) = L^e(\underline{i})$. 

\begin{proposition} \label{prop:imaginaryMackey}
Let $n \in \mathbb{Z}_{\geq 0}$ and $\underline{i} \in \mathring{I}^n$. 
Then, we have isomorphisms of graded $R^e(\delta)^{\otimes n}$-modules
\begin{align*}
\Res_{\delta^n} (\Delta^e(\underline{i})) &\simeq \bigoplus_{w \in \mathfrak{S}_n} \Delta_{i_{w(1)}}^e(\delta) \otimes \cdots \otimes \Delta_{i_{w(n)}}^e(\delta), \\
\Res_{\delta^n} (L^e(\underline{i})) &\simeq \bigoplus_{w \in \mathfrak{S}_n} L_{i_{w(1)}}^e(\delta) \otimes \cdots  \otimes L_{i_{w(n)}}^e(\delta). 
\end{align*}
\end{proposition}

\begin{proof}
By Lemma \ref{lem:Mackey}, the Mackey filtration of $\Res_{\delta^n}(\Delta^e(\underline{i}))$ consists of $n!$ subquotients, 
each isomorphic to $\Delta_{i_{w(1)}}^e(\delta) \otimes \cdots \otimes \Delta_{i_{w(n)}}^e(\delta)$ for $w \in \mathfrak{S}_n$. 
Since all of these subquotients are projective $R^e(\delta)^{\otimes n}$-modules, 
the filtration splits into a direct sum and we obtain the first isomorphism. 

Similarly, the Mackey filtration of the restriction of $L^e(\underline{i})$ consists of $n!$ subquotients, 
each isomorphic to $L_{i_{w(1)}}^e(\delta) \otimes \cdots \otimes L_{i_{w(n)}}^e(\delta)$ for $w  \in \mathfrak{S}_n$. 
By Lemma \ref{lem:semisimplicity}, the filtration splits into a direct sum and we obtain the second isomorphism. 
\end{proof}

\begin{lemma} \label{lem:tensorspacehom}
Let $n \in \mathbb{Z}_{\geq 0}$ and $\underline{i},\underline{j} \in \mathring{I}^n$. 
For $i \in \mathring{I}$, set $l_i = \lvert \{ 1 \leq k \leq n \mid i_k = i \} \rvert$ and $m_i = \lvert \{1 \leq k \leq n \mid j_k = i \} \rvert$. 

(1) We have 
\[
\dim_q \Hom_R(L^e(\underline{i}),L^e(\underline{j})) = \begin{cases}
\prod_{i \in \mathring{I}} (l_i!) & \text{if $l_i = m_i$ for all $i \in \mathring{I}$}, \\
0 & \text{otherwise}. 
\end{cases}
\]
In particular, the graded vector space $\Hom_R(L^e(\underline{i}),L^e(\underline{j}))$ is concentrated in degree zero. 

(2) We have 
\[
\dim_q \Hom_R(\Delta^e(\underline{i}),\Delta^e(\underline{j})) \in \begin{cases}
\prod_{i \in \mathring{I}} (l_i!) + q\mathbb{Q}[q]_{(q)} & \text{if $l_i = m_i$ for all $i \in \mathring{I}$}, \\
q \mathbb{Q}[q]_{(q)} & \text{otherwise}.
\end{cases}
\]

(3) The canonical surjective homomorphisms $\Delta^e(\underline{i}) \to L^e(\underline{i})$ and $\Delta^e(\underline{j}) \to L^e(\underline{j})$ induce isomorphisms
\[
\hom_R(\Delta^e(\underline{i}),\Delta^e(\underline{j})) \xrightarrow{\sim} \hom_R(\Delta^e(\underline{i}),L^e(\underline{j})) \xleftarrow{\sim} \hom_R(L^e(\underline{i}),L^e(\underline{j})). 
\]

(4) The isomorphism of (3) is compatible with compositions; 
for another $\underline{h} \in \mathring{I}^n$, the following diagram commutes: 
\[
\begin{tikzcd}
\hom_R(\Delta^e(\underline{h}),\Delta^e(\underline{i})) \otimes \hom_R(\Delta^e(\underline{i}),\Delta^e(\underline{j})) \arrow[r]\arrow[d,-,"\sim" sloped] & \hom_R(\Delta^e(\underline{h}),\Delta^e(\underline{j})) \arrow[d,-,"\sim"sloped] \\
\hom_R(L^e(\underline{h}),L^e(\underline{i})) \otimes \hom_R(L^e(\underline{i}),L^e(\underline{j})) \arrow[r] & \hom_R(L^e(\underline{h}),L^e(\underline{j})) 
\end{tikzcd}
\]
where the horizontal homomorphisms are given by composition, and the vertical isomorphisms are given by (3). 
\end{lemma}

\begin{proof}
(1) and (2) follows from Proposition \ref{prop:imaginaryMackey} and Lemma \ref{lem:qdimimaginary}. 
To prove (3), first note that the second homomorphism is injective. 
Since $\Delta^e(\underline{i})$ is a projective $R^e(n\delta)$-module, the first homomorphism of (3) is surjective. 
Since the left and right spaces have the same dimension by (1) and (2), both homomorphisms are isomorphisms. 
(4) is straightforward. 
\end{proof}

\subsection{Imaginary cuspidal modules of one color} \label{sub:onecolor}

In this subsection, we assume that $\mathbf{k}$ is a field. 
Let $i \in \mathring{I}$. 
Let $\mathring{\varpi}_i^{\lor}$ be the fundamental coweight for $\mathring{\mathfrak{g}}$, 
namely, $\langle \mathring{\varpi}_i^{\lor}, \alpha_j \rangle = \delta_{i,j} \ (j \in \mathring{I})$. 
Following \cite{MR3670026}, we define 
\[
F_i^{\epsilon} = \{ \alpha \in \Phi_+^{\mathrm{min}} \mid \langle \mathring{\varpi}_i^{\lor}, p(\alpha) \rangle \in \epsilon \mathbb{Z}_{>0} \} \ (\epsilon \in \{-,0,+\}). 
\]
Let $\mathsf{\Delta}_i$ be the set of real roots in $F_i^0$ that cannot be written as sums of other roots in $F_i^0$. 
This set can be described as follows. 
Consider the diagram $\Gamma_i$ obtained from the Dynkin diagram of $\mathring{\mathfrak{g}}$ by removing the vertex $i$. 
Its connected components are Dynkin diagrams of some simple Lie algebras.
The set $\mathsf{\Delta}_i$ consists of $\alpha_j \ (j \in \mathring{I}\setminus \{i\})$ and $\delta - \alpha$ where $\alpha$ is the longest element of each connected component of $\Gamma_i$. 

\begin{lemma} \label{lem:cuspidalcriterion}
Let $i \in \mathring{I}$, and let $L$ be a self-dual simple $R(n\delta)$-module for some $n \geq 0$. 
The following are equivalent: 
\begin{enumerate}
\item for $\beta,\gamma \in \mathsf{Q}_+$ satisfying $\beta + \gamma = n\delta$ and $\Res_{\beta,\gamma} L \neq 0$, we have 
\[
\beta \in \spn_{\mathbb{Z}_{\geq 0}} (F_i^- \cup F_i^0); 
\]
\item for $\beta,\gamma \in \mathsf{Q}_+$ satisfying $\beta + \gamma = n\delta$ and $\Res_{\beta,\gamma} L \neq 0$, we have 
\[
\gamma \in \spn_{\mathbb{Z}_{\geq 0}}(F_i^+ \cup F_i^0).
\]
\end{enumerate}
\end{lemma}

\begin{proof}
It follows as in \cite[Corollary 3.45]{murata2025affinehighestweightstructures} from the cuspidal decomposition of $L$ with respect to a convex order $\preceq$ satisfying 
\[
\alpha^- \prec \alpha^0 \prec \alpha^+ \ (\alpha^- \in F_i^-, \alpha^0 \in F_i^0, \alpha^+ \in F_i^+). 
\]
\end{proof}

\begin{lemma}\label{lem:onecolor}
Let $i \in \mathring{I}$, and let $L$ be a self-dual simple $R(n\delta)$-module for some $n \geq 0$. 
The following are equivalent: 
\begin{enumerate}
\item $L$ is $e$-cuspidal and $e(n\delta-\alpha_j,\alpha_j)L = 0$ for any $j \in \mathring{I}\setminus \{i\}$; 
\item (resp.\ (3)) \ For $\beta,\gamma \in \mathsf{Q}_+$ satisfying $\beta + \gamma = n\delta$ and $\Res_{\beta,\gamma} L \neq 0$, we have 
\[
\text{$\beta \in \spn_{\mathbb{Z}_{\geq 0}} (F_i^- \cup F_i^0)$ \ (resp.\ $\gamma \in \spn_{\mathbb{Z}_{\geq 0}}(F_i^+ \cup F_i^0)$)}, 
\]
and $e(n\delta-\alpha,\alpha)L$ (resp.\ $e(\alpha,n\delta-\alpha)L$) is zero for any $\alpha \in \mathsf{\Delta}_i$.
\end{enumerate}
\end{lemma}

\begin{proof}
In view of Lemma \ref{lem:cuspidalcriterion}, the equivalence between (2) and (3) follows from \cite[Corollary 3.29]{MR3542489}. 

We prove that (1) implies (2). 
Assume that $L$ satisfies (1).  
The former condition of (2) follows from the assumption that $L$ is $e$-cuspidal.
The $e$-cuspidality also implies that $e(n\delta-\alpha,\alpha) L = 0$ for all $\alpha \in \mathsf{\Delta}_i \setminus \{\alpha_j \mid j \in \mathring{I} \setminus \{i\} \}$.
For $j \in \mathring{I} \setminus \{i\}$, we have $e(n\delta-\alpha_j,\alpha_j)L= 0$ by the assumption. 

It remains prove that (2) and (3) imply (1). 
Assume that $L$ satisfies (2) and (3). 
\cite[Proposition 3.41]{MR3542489} shows that $L$ is $e$-cuspidal. 
For $j \in \mathring{I} \setminus \{i\}$, we have $e(n\delta-\alpha_j,\alpha_j)L= 0$ by the assumption in (2). 
\end{proof}

For $i \in \mathring{I}$ and $n \in \mathbb{Z}_{\geq 0}$, 
let $\Xi_n^i$ be the subset of $\Xi_n^e$ consisting of all the self-dual simple $R^e(n\delta)$-modules up to isomorphism satisfying the three equivalent conditions in Lemma \ref{lem:onecolor}. 
Let $\Xi^i = \sqcup_{n \geq 0} \Xi_n^i$.

\begin{lemma} \label{lem:colorres}
Let $i \in \mathring{I}$, $m,n \in \mathbb{Z}_{\geq 0}$ and $L \in \Xi_{m+n}^i$. 
Then, every composition factor of $\Res_{m\delta,n\delta}L$ is of the form $L_1 \otimes L_2$ up to grading shift for some $L_1 \in \Xi_m^i$ and $L_2 \in \Xi_n^i$.
\end{lemma}

\begin{proof}
We ignore grading shifts in this proof. 
Assume that $L_1 \otimes L_2$ is a composition factor of $\Res_{m\delta,n\delta}$, 
where $L_1$ is a simple $R(m\delta)$-module and $L_2$ is a simple $R(n\delta)$-module.
Since $L$ satisfies Condition (2) of Lemma \ref{lem:onecolor}, so does $L_1$, and hence $L_1 \in \Xi^i$. 
Similarly, Condition (3) passes to $L_2$, which means $L_2 \in \Xi^i$. 
\end{proof}

\begin{theorem} \label{thm:onecolor}
For $L_i \in \Xi^i \ (i \in \mathring{I})$, 
the convolution product $\circ_{i \in \mathring{I}} L_i$ of all these modules does not depend on the order. 
It yields a bijection
\[
\prod_{i \in \mathring{I}} \Xi^i \to \Xi^e, \ (L_i) \mapsto \circ_{i \in \mathring{I}}L_i. 
\]
\end{theorem}

\begin{proof}
Ignoring grading shifts, this is \cite[Corollary 3.45]{MR3542489}. 
We verify that the bijection respects grading. 

Let $i,j \in \mathring{I}$ satisfying $i \neq j$, and let $m,n \in \mathbb{Z}_{\geq 1}$. 
We prove that given $L_i \in \Xi_m^i$ and $L_j \in \Xi_n^j$, 
$L_i \circ L_j \simeq L_j \circ L_i$ as graded modules. 
Since we already know that both sides are simple and mutually isomorphic as ungraded modules, 
it suffices to show that
\[
\hom_R (q^d L_i\circ L_j, L_j \circ L_i) = 0 
\]
for $d \neq 0$. 
By induction-restriction adjunction, the left hand side is isomorphic to 
\[
\hom_{R(m\delta)\otimes R(n\delta)}(q^dL_i \otimes L_j, \Res_{m\delta,n\delta}(L_j\circ L_i)). 
\]
Hence, it is enough to prove that $[\Res_{m\delta,n\delta}(L_j \circ L_i):L_i \otimes L_j]_q = 1$.    
We consider the Mackey filtration \cite[Proposition 2.18]{MR2525917} of $\Res_{m\delta,n\delta}(L_j \circ L_i)$. 

Since both $L_i$ and $L_j$ are $e$-cuspidal, 
any nonzero subquotient of the Mackey filtration  is of the form 
\begin{equation} \label{eq:subquotient}
(\Ind_{m_1\delta,m_2\delta}\otimes \Ind_{n_1\delta,n_2\delta}) (\Res_{m_1\delta,n_1\delta}L_j \otimes \Res_{m_2\delta,n_2\delta}L_i), 
\end{equation}
where $m_1, m_2, n_1, n_2 \in \mathbb{Z}_{\geq 0}$ satisfy 
\[
m_1 + m_2 = m,\ n_1 + n_2 = n, \ m_1 + n_1 = n,\ m_2 + n_2 = m. 
\]
Here, no grading shift is necessary since $(\delta,\delta) = 0$. 

Note that $m_1=n_2, m_2 = m-m_1$ and $n_1 = n-m_1$. 
If $m_1=n_2 = 0, m_2 = m$ and $n_1 = n$, this subquotient is isomorphic to $L_j \otimes L_i$.
Next, assume $m_1 = n_2 \neq 0$. 
By Lemma \ref{lem:colorres}, 
any composition factor of the restriction $\Res_{m_1\delta,n_1\delta}L_j$ is of the form $M_1 \otimes N_1$ up to grading shifts for some $M_1, N_1 \in \Xi^j$.
Similarly, any composition factor of $\Res_{m_2\delta,n_2\delta}L_i$ is of the form $M_2 \otimes N_2$ up to grading shift for some $M_2, N_2 \in \Xi^i$.
Hence, the module (\ref{eq:subquotient}) has a filtration whose successive quotients are of the form
\[
(M_1 \circ M_2) \otimes (N_1 \circ N_2)
\]
up to grading shifts for some $M_1 \in \Xi_{m_1}^j, N_1 \in \Xi_{n_1}^j, M_2 \in \Xi_{m_2}^i$ and $N_2 \in \Xi_{n_2}^i$. 
We already know that $M_1 \circ M_2$ and $N_1 \circ N_2$ are simple modules.
Since $m_1 = n_2 \neq 0$, $M_1 \in \Xi_{m_1}^j$ and $N_2 \in \Xi_{n_2}^i$ are nontrivial. 
Hence, no grading shifts of $M_1 \circ M_2$ belong to $\Xi^i$, while no grading shifts of $N_1 \circ N_2$ belong to $\Xi^j$ by the first paragraph of the proof. 
It implies that the subquotient (\ref{eq:subquotient}) does not contain any grading shift of $L_i \otimes L_j$ as a composition factors if $m_1 = n_2 \neq 0$. 

Therefore, we obtain $[\Res_{m\delta,n\delta}(L_j \circ L_i):L_i \otimes L_j]_q = 1$. 
\end{proof}

\begin{lemma} \label{lem:minusculeonecolor}
For $i \in \mathring{I}$, we have $\Xi_1^i = \{L_i^e(\delta)\}$. 
\end{lemma}

\begin{proof}
By Theorem \ref{thm:onecolor}, we have 
\[
\Xi_1^e = \sqcup_{i \in \mathring{I}} \Xi_1^i. 
\]
By Definition \ref{def:minuscule} and induction-restriction adjunction, we have $e(\delta-\alpha_i,i)L_i^e(\delta) \neq 0$. 
Hence, $L_i^e(\delta) \not \in \Xi_1^j$ for any $j \in \mathring{I} \setminus \{i\}$. 
On the other hand, we have $\Xi_1^e = \{L_i^e(\delta) \mid i \in \mathring{I}\}$ by Corollary \ref{cor:minusclesimple}. 
Hence, the lemma follows. 
\end{proof}

\section{Quantum imaginary Schur-Weyl duality}

\subsection{Action of the Hecke algebra} \label{sub:heckeaction}

In this subsection, we assume that $\mathbf{k}$ is a field. 
Recall the parameter $t \in \mathbf{k}^{\times}$ from Subsection \ref{sub:quiverHecke}; 
it appears in the coefficients of the polynomials $Q$ in type $A_N^{(1)}$, 
whereas it is set to be $1$ in all the other types. 

For $n \geq 0$, let $H_n = H_n(t)$ be the Hecke algebra of $\mathfrak{S}_n$, 
a $\mathbf{k}$-algebra generated by $T_k \ (1 \leq k \leq n-1)$ with relations
\[
T_k T_l = T_l T_k \ (k-l \geq 2), \ (T_k+1)(T_k-t) = 0, \ T_kT_{k+1}T_k = T_{k+1}T_kT_{k+1}.
\]
Note that if $X_N^{(1)}$ is not $A_N^{(1)}$, $H_n$ is just the group algebra of $\mathfrak{S}_n$. 
For $m, n \in \mathbb{Z}_{\geq 0}$, there is a canonical inclusion
\[
\iota_{m,n} \colon H_m \otimes H_n \to H_{m+n}, 
\]
given by 
\[
\iota_{m,n} (T_k \otimes 1) = T_k, \ \iota_{m,n} (1 \otimes T_l) = T_{m+l} \ (1 \leq k \leq m-1, 1 \leq l \leq n-1)
\]
For $\underline{n} = (n_i)_{i \in \mathring{I}} \in \mathbb{Z}_{\geq 0}^{\mathring{I}}$, 
we define 
\[
H_{\underline{n}} = H_{n_1} \otimes \cdots  \otimes H_{n_N}. 
\]
For $\underline{m}, \underline{n} \in \mathbb{Z}_{\geq 0}^{\mathring{I}}$, 
there is a canonical inclusion
\[
\iota_{\underline{m},\underline{n}} \colon H_{\underline{m}} \otimes H_{\underline{n}} \to H_{\underline{m} + \underline{n}}, 
\]
given by
\[
\iota_{\underline{m},\underline{n}}(x_1 \otimes \cdots \otimes x_N) \otimes (y_1 \otimes \cdots \otimes y_N) = \iota_{m_1,n_1}(x_1 \otimes y_1) \otimes \cdots \otimes \iota_{m_N,n_N}(x_N \otimes y_N).
\]
When we wish to emphasize the coefficient field, we write $H_{\underline{n},\mathbf{k}} = H_{\underline{n}}$. 
  
Set $\lvert \underline{n} \rvert = \sum_{i \in \mathring{I}} n_i$. 
We define $R^e(\lvert \underline{n} \rvert)$-modules
\[
L^e(\underline{n}) = L_1^e(\delta)^{\circ n_1} \circ \cdots \circ L_N^e(\delta)^{\circ n_N}, \ \Delta^e(\underline{n}) = \Delta_1^e(\delta)^{\circ n_1} \circ \cdots \circ \Delta_N^e(\delta)^{\circ n_N}. 
\]
When we wish to emphasize the coefficient field, we write $L_{\mathbf{k}}^e(\underline{n}) = L^e(\underline{\lambda})$. 

\begin{lemma} \label{lem:singlecolor}
Let $i,j \in \mathring{I}$ and assume $i \neq j$. Then, 
\begin{enumerate} 
  \item the modules $L_i^e(\delta) \circ L_j^e(\delta)$ and $L_j^e(\delta) \circ L_i^e(\delta)$ are self-dual simple and isomorphic to each other;
  \item $\Delta_i^e(\delta) \circ \Delta_j^e(\delta)$ is a projective cover of $L_i^e(\delta) \circ L_j^e(\delta)$ in $\gMod{R^e}$. 
\end{enumerate}
\end{lemma}

\begin{proof} 
(1) follows from Theorem \ref{thm:onecolor} and Lemma \ref{lem:minusculeonecolor}. 

(2) Let $P = \Delta_i^e(\delta) \circ \Delta_j^e(\delta)$, 
which is a projective $R^e(2\delta)$-module by Proposition \ref{prop:imaginaryind} (2). 
For a simple $R^e(2\delta)$-module $L$, we have by induction-restriction adjunction: 
\begin{equation} \label{eq:singlecolor1}
\Hom_{R^e}(P,L) \simeq \Hom_{R^e(\delta)\otimes R^e(\delta)} (\Delta_i^e(\delta) \otimes \Delta_j^e(\delta), \Res_{\delta,\delta}L). 
\end{equation}
If $L = L_i^e(\delta)\circ L_j^e(\delta)$, Proposition \ref{prop:imaginaryMackey} shows 
\[
\Res_{\delta,\delta} L \simeq (L_i^e(\delta) \otimes L_j^e(\delta)) \oplus (L_j^e(\delta)\otimes L_i^e(\delta)).
\]
Since $i \neq j$, it follows that the graded dimension of (\ref{eq:singlecolor1}) is $1$ in this case. 

It remains to prove that, if any grading shifts of $L$ is not isomorphic to $L_i^e(\delta)\circ L_j^e(\delta)$, 
the space (\ref{eq:singlecolor1}) is zero. 
It suffices to show that $[\Res_{\delta,\delta}L \colon L_i^e(\delta)\otimes L_j^e(\delta)]_q = 0$. 
We ignore grading below. 
By Theorem \ref{thm:onecolor} and Lemma \ref{lem:minusculeonecolor}, 
either $L$ belongs to $\Xi_2^k$ for some $k \in \mathring{I}$, or $L \simeq L_1 \circ L_2$ for some $L_1 \in \Xi_1^{k_1}, L_2 \in \Xi_1^{k_2}$ with $k_1, k_2 \in \mathring{I}$ satisfying $k_1 \neq k_2$ and $\{k_1,k_2\} \neq \{i,j\}$. 
In the former case, the assertion follows from Lemma \ref{lem:colorres}. 
In the latter case, the assertion follows from Proposition \ref{prop:imaginaryMackey}. 

\end{proof}
Let $1 \leq i < j \leq N$.
Using Lemma \ref{lem:singlecolor}, we choose an isomorphism 
\begin{equation} \label{eq:isom1}
\kappa_{j,i} \colon L_j^e(\delta) \circ L_i^e(\delta) \to L_i^e(\delta) \circ L_j^e(\delta), 
\end{equation}
which is unique up to a scalar multiple. 
We define $\kappa_{i,j} = \kappa_{j,i}^{-1}$. 
By Lemma \ref{lem:tensorspacehom}, $\kappa_{j,i}$ uniquely lifts to a homogeneous homomorphism 
\begin{equation} \label{eq:isom2}
\Delta_j^e(\delta) \circ \Delta_i^e(\delta) \to \Delta_i^e(\delta) \circ \Delta_j^e(\delta).   
\end{equation}
By Lemma \ref{lem:tensorspacehom}, the inverse of (\ref{eq:isom1}) also has a lift, which implies that (\ref{eq:isom2}) is an isomorphism. 

Using these isomorphisms, we define isomorphisms 
\begin{equation} \label{eq:isom3}
L^e(\underline{m}) \circ L^e(\underline{n}) \simeq L^e(\underline{m}+\underline{n}), \ \Delta^e(\underline{m}) \circ \Delta^e(\underline{n}) \simeq \Delta^e(\underline{m}+\underline{n}) \ (\underline{m},\underline{n} \in \mathbb{Z}_{\geq 0}^{\mathring{I}})
\end{equation}
as follows.
Consider an element $w = w_{\underline{m},\underline{n}}$ of $\mathfrak{S}_{\lvert\underline{m}+\underline{n}\rvert}$ defined by 
\begin{align*}
w\left(\sum_{1 \leq p \leq k-1} m_p +r\right) &= \sum_{1\leq p \leq k-1} (m_p+n_p) + r \ (1 \leq k \leq N, 1 \leq r \leq m_k),  \\ 
w\left(\lvert \underline{m}\rvert + \sum_{1 \leq p \leq k-1}n_p + r\right) &= \sum_{1 \leq p \leq k-1} (m_p + n_p)+ m_k + r \ (1 \leq k \leq N, 1 \leq r \leq n_k). 
\end{align*}
Take a reduced expression of $w$ and compose isomorphisms (\ref{eq:isom1}) or (\ref{eq:isom2}) accordingly to obtain isomorphisms (\ref{eq:isom3}). 
We claim that this isomorphism is independent of the choice of the reduced expression. 
In fact, there exist no $1 \leq a < b < c \leq \lvert \underline{m} + \underline{n}\rvert$ such that $w(a) > w(b) > w(c)$,
hence any reduced expression of $w$ can be obtained from another by only using the relations $s_k s_l = s_ls_k \ (l-k \geq 2)$. 

The isomorphisms (\ref{eq:isom3}) induce homomorphisms
\begin{align} \label{eq:hom}
&\grend_R(\Delta^e(\underline{m})) \otimes \grend_R(\Delta^e(\underline{n})) \xrightarrow{can} \grend_R(\Delta^e(\underline{m})\circ \Delta^e(\underline{n})) \simeq \grend_R(\Delta^e(\underline{m}+\underline{n})), \\ 
&\grend_R(L^e(\underline{m})) \otimes \grend_R(L^e(\underline{n})) \xrightarrow{can} \grend_R(L^e(\underline{m})\circ L^e(\underline{n})) \simeq \grend_R(L^e(\underline{m}+\underline{n})). \notag
\end{align}
The following theorem is the main result of this paper. 

\begin{theorem} \label{thm:heckeaction}
  The following assertions hold: 
  \begin{enumerate}
  \item for any $\underline{n} \in \mathbb{Z}_{\geq 0}^{\mathring{I}}$ and $d < 0$, we have $\End_R(\Delta^e(\underline{n}))_d = 0$;
  \item for any $\underline{n} \in \mathbb{Z}_{\geq 0}^{\mathring{I}}$, $\End_R(L^e(\underline{n}))$ is concentrated in degree zero; 
  \item for each $\underline{n} \in \mathbb{Z}_{\geq 0}^{\mathring{I}}$, there exist isomorphisms of algebras
  \[
 \grend_{R(\lvert \underline{n} \rvert \delta)} (\Delta^e(\underline{n})) \simeq \grend_{R(\lvert \underline{n}\rvert\delta)} (L^e(\underline{n})) \simeq H_{\underline{n}},
  \]
  through which the homomorphisms (\ref{eq:hom}) coincides with the canonical inclusion
  \[
  \iota_{\underline{m},\underline{n}} \colon H_{\underline{m}}\otimes H_{\underline{n}} \to H_{\underline{m}+\underline{n}}. 
  \]
\end{enumerate}
\end{theorem}

(1) and (2) are special cases of Lemma \ref{lem:tensorspacehom}.  
We also see using Proposition \ref{prop:imaginaryMackey} that there are canonical isomorphisms
\[
\grend_R(\Delta^e(\underline{n})) \simeq \grend_R(L^e(\underline{n})) \simeq \grend_R(L_1^e(\delta)^{\circ n_1}) \otimes \cdots \otimes \grend_R(L_N^e(\delta)^{\circ n_N}). 
\]
Hence, Theorem \ref{thm:heckeaction} is reduced to the following proposition. 

\begin{proposition} \label{prop:check}
For each $i \in \mathring{I}$ and $n \in \mathbb{Z}_{\geq 0}$, there exists an isomorphism of algebras
\begin{equation} \label{eq:isom4}
\grend_R(L_i^e(\delta)^{\circ n}) \simeq H_n, 
\end{equation}
that satisfies the following conditions: 
\begin{enumerate}
\item under the identification (\ref{eq:isom4}), the following two canonical homomorphisms coincide for $i \in \mathring{I}$ and $m,n \in \mathbb{Z}_{\geq 0}$:  
\[
\grend_R(L_i^e(\delta)^{\circ m}) \otimes \grend_R(L_i^e(\delta)^{\circ n}) \to \grend_R(L_i^e(\delta)^{\circ m+n}), \ \iota_{m,n} \colon H_m \otimes H_n \to H_{m+n};
\]
\item for $i, j \in \mathring{I}$ with $i \neq j$ and $m,n \in \mathbb{Z}_{\geq 0}$, the induced isomorphism 
\begin{align*}
&H_m \otimes H_n \simeq \grend_R(L_i^e(\delta)^{\circ m}) \otimes \grend_R(L_j^e(\delta)^{\circ n}) \simeq \grend_R(L_i^e(\delta)^{\circ m}\circ L_j^e(\delta)^{\circ n}) \\
&\overset{(\ref{eq:isom3})}{\simeq} \grend_R(L_j^e(\delta)^{\circ n}\circ L_i^e(\delta)^{\circ m}) \simeq \grend_R(L_j^e(\delta)^{\circ n}) \otimes \grend_R(L_i^e(\delta)^{\circ m}) \simeq H_n \otimes H_m
\end{align*}
coincides with $x \otimes y \mapsto y \otimes x \ (x \in H_m, y \in H_n)$. 
\end{enumerate}
\end{proposition}

By Lemma \ref{lem:tensorspacehom}, 
we have $\dim \grend_R(L_i^e(\delta)^{\circ n}) = n!$. 
Hence, to prove the proposition it suffices to give a nonconstant endomorphism $\mathsf{T}_i$ of $L_i^e(\delta)^{\circ 2}$ for each $i \in \mathring{I}$, 
and verify that the following assertions hold: 
\begin{enumerate}[label = (\alph*)]
\item $(\mathsf{T}_i+1)(\mathsf{T}_i-t) = 0$ in $\grend_R(L_i^e(\delta)^{\circ 2})$; 
\item $\mathsf{T}_i^{1,2} \mathsf{T}_i^{2,3}\mathsf{T}_i^{1,2} = \mathsf{T}_i^{2,3}\mathsf{T}_i^{1,2}\mathsf{T}_i^{2,3}$ in $\grend_R(L_i^e(\delta)^{\circ 3})$;
\item the induced homomorphism $H_n \to \grend_R(L_i^e(\delta)^{\circ n})$ is injective for any $n \geq 0$;
\item Proposition \ref{prop:check} (2) holds for $j \neq i$ and $(m,n) = (2,1)$. 
\end{enumerate}
Here, for $X_k, Y_k \in \gMod{R} \ (1 \leq k \leq r)$ and a homomorphism $f \colon X_k \circ X_{k+1} \to Y_k \circ Y_{k+1}$ for a specific $1 \leq k \leq r-1$, 
$f^{k,k+1}$ denotes the induced homomorphism from $X_1 \circ \cdots \circ X_r$ to $Y_1 \circ \cdots \circ Y_r$. 

In the next subsection, we will see that (c) and (d) automatically hold for nonconstant $T_i$ satisfying (a) and (b). 
Hence, it remains to construct nonconstant $\mathsf{T}_i$ satisfying (a) and (b). 
When $t =1$, this is already known \cite[Section 14]{MR3694676}, \cite[Section 4.2]{MR3589160}.
Note that, outside of type $A_N^{(1)}$, the parameter $t$ reduces to $1$ (Subsection \ref{sub:quiverHecke}); 
hence the  result follows from the existing works. 

To find $\mathsf{T}_i$ in the type $A_N^{(1)}$ case with an arbitrary parameter $t$, 
we first complete the type $A_1^{(1)}$ case by explicitly constructing an endomorphism $\mathsf{T}_1$ of $L_1^e(\delta)^{\circ 2}$ in Subsection \ref{sub:A1case}. 
Then, we explain how to reduce the general case of type $A_N^{(1)}$ to type $A_1^{(1)}$ in Subsection \ref{sub:ANcase}. 

\subsection{Proof of (c) and (d)}

In this subsection, assume that $\mathbf{k}$ is a field. 
We prove that given nonconstant endomorphisms $\mathsf{T}_i$ of $L_i^e(\delta)^{\circ 2}$ for $i \in \mathring{I}$ satisfying (a) and (b), 
they also satisfy (c) and (d). 
Note that we have an isomorphism 
\begin{align*}
\End_R(L_i^e(\delta)^{\circ 2}) &\simeq \Hom_{R(\delta)\otimes R(\delta)} (L_i^e(\delta)^{\otimes 2},\Res_{\delta,\delta}L_i^e(\delta)^{\circ 2}), \\
f &\mapsto (u\otimes v \mapsto f(u\boxtimes v)), 
\end{align*}
by induction-restriction adjunction. 
By Lemma \ref{lem:Mackey}, the Mackey-filtration of $\Res_{\delta,\delta}L_i^e(\delta)^{\circ 2}$ yields a short exact sequence of $R^e(\delta)^{\otimes 2}$-modules
\[
0 \to F_{<s_1}L_i^e(\delta)^{\circ 2} = L_i^e(\delta) \otimes L_i^e(\delta) \to \Res_{\delta,\delta}L_i^e(\delta)^{\circ 2} \to L_i^e(\delta)\otimes L_i^e(\delta) \to 0. 
\]
Hence, the assumption that $\mathsf{T}_i$ is nonconstant implies the following lemma. 

\begin{lemma} \label{lem:leadingterm}
Let $i \in \mathring{I}$. 
There exists $b_i \in \mathbf{k}^{\times}$ such that 
\[
\mathsf{T}_i(u\boxtimes v) \equiv b_i \tilde{\tau}_1 (v \boxtimes u) \mod F_{<s_1} L_i^e(\delta)^{\circ 2} \ (u,v \in L_i^e(\delta)). 
\]
\end{lemma}

Now, we prove (c) for $i \in \mathring{I}$. 

\begin{proof}[Proof of (c)]
Let $n \geq 1$ and let $\pi_i \colon H_n \to \End_R(L_i^e(\delta)^{\circ n})$ be the homomorphism induced from $\mathsf{T}_i$. 
By Lemma \ref{lem:leadingterm}, for any $w \in \mathfrak{S}_n$ and $u_1, \ldots, u_n \in L_i^e(\delta)$, we have 
\[
\pi_i(T_w) (u_1 \boxtimes \cdots \boxtimes u_n) \equiv b_i^{\ell(w)}\tilde{\tau}_w(u_{w^{-1}(1)} \boxtimes \cdots \boxtimes u_{w^{-1}(n)}) \mod F_{<w} L_i^e(\delta)^{\circ n}
\]
Hence, $\{\pi_i(T_w) \mid w \in \mathfrak{S}_n\}$ is linearly-independent, which shows that $\pi_i$ is injective. 
\end{proof}

Next, we prove (d) for $i,j \in \mathring{I}$ such that $i \neq j$. 

\begin{proof}[Proof of (d)]
Since $(m,n) = (2,1)$, the composition described in Proposition \ref{prop:check} (2) yields a $\mathbf{k}$-algebra automorphism
\[
\psi \colon H_2 \simeq  H_2 \otimes H_1 \simeq H_1 \otimes H_2 \simeq H_2, 
\]
which must be shown to be the identity. 
Suppose for contradiction that it is not the identity. 
Note that $H_2$ has only at most one nontrivial $\mathbf{k}$-algebra automorphism given by 
\[
T_1 \mapsto -T_1 + t-1. 
\] 
When the characteristic of $\mathbf{k}$ is $2$ and $t=1$, this coincides with the identity.  
Therefore, we may assume that $\psi(T_1) = -T_1 +t -1$ and either the characteristic of $\mathbf{k}$ is not $2$ or $t \neq 1$. 

First, we assume that the characteristic of $\mathbf{k}$ is not $2$. 
Recall that the definition of automorphism $\psi$ involves an isomorphism $\kappa_{i,j} \colon L_i^e(\delta)\circ L_j^e(\delta) \to L_j^e(\delta)\circ L_i^e(\delta)$ of (\ref{eq:isom1}). 
Similarly to Lemma \ref{lem:leadingterm}, there exists $c_{i,j} \in \mathbf{k}^{\times}$ such that 
\[
\kappa_{i,j} (u \boxtimes v) \equiv c_{i,j} \tilde{\tau}_1(v \boxtimes u) \mod F_{<s_1} (L_j^e(\delta)\circ L_i^e(\delta)), 
\]
for $u \in L_i^e(\delta), v \in L_j^e(\delta)$.
Now, the equation $\psi(T_1) = -T_1 + t-1$ implies 
\[
b_ic_{i,j}^2\tilde{\tau}_{w_3} (u_1 \boxtimes u_2 \boxtimes u_3) \equiv -b_ic_{i,j}^2 \tilde{\tau}_{w_3}(u_1 \boxtimes u_2 \boxtimes u_3) \mod F_{< w_3} (L_j^e(\delta) \circ L_i^e(\delta)^{\circ 2}), 
\]
for any $u_1 \in L_j^e(\delta), u_2,u_3 \in L_i^e(\delta)$. 
Since the characteristic is not $2$, it yields a contradiction. 

Next, we assume that the characteristic of $\mathbf{k}$ is $2$ and $t \neq 1$. 
Then, $t \neq -1$. 
Hence, the module $L_i^e(\delta)^{\circ 2}$ has an eigenspace decomposition
\[
L_i^e(\delta)^{\circ 2} = \Ker (T_1+1) \oplus \Ker (T_1-t). 
\]
Let $K = \Ker (T_1+1), K' = \Ker (T_1-t)$. 
Since $\psi(T_1) = -T_1 + t-1$, we have 
\begin{equation} \label{eq:contradiction}
K\circ L_j^e(\delta) \simeq L_j^e(\delta) \circ K'. 
\end{equation}
We claim that $[K] = [K']$ in the Grothendieck group $K(\gmod{R})$. 
Since $K$ is a submodule of $L_i^e(\delta)^{\circ 2}$ and $L_i^e(\delta) \in \Xi^i$ by Lemma \ref{lem:minusculeonecolor}, 
we have $e(2\delta-\alpha_j,j)K = 0$ for any $j \in \mathring{I} \setminus \{i\}$. 
Hence, Lemma \ref{lem:onecolor} shows that all the composition factors of $K$ belong to $\Xi^i$ up to grading shifts. 
This is also true for $K'$. 
Therefore, Theorem \ref{thm:onecolor} and (\ref{eq:contradiction}) imply $[K] = [K']$. 
Hence, $[L_i^e(\delta)^{\circ 2}] = [K] + [K'] = 2[K]$. 
However, there exists a simple module that appears in the composition series of $L_i^e(\delta)^{\circ 2}$ exactly once as shown in Lemma \ref{lem:string} below.  
It yields a contradiction.
\end{proof}

In the proof above, we used the following lemma:

\begin{lemma}\label{lem:string}
For a simple module $L \in \gmod{R}$, 
there exists a simple module $L'$ such that 
\[
[L\circ L \colon L']_{q =1} = 1. 
\]
\end{lemma}

To prove this lemma, we use the following result: 

\begin{theorem}[{\cite{MR2822211}}]\label{thm:crystal}
The set $\mathcal{B}$ of isomorphism classes of simple modules in $\gmod{R}$ up to grading shift has a $\mathfrak{g}$-crystal structure given by 
\begin{align*}
&\wt (L) = -\beta \quad \text{if $L \in \gmod{R(\beta)}$}, \\
&\tilde{f}_i(L) = \hd (L(i) \circ L), \\
&\tilde{e}_i(L) = \soc E_iL, \\
&\varepsilon_i(L) = \max \{n \geq 0 \mid E_i^nL \neq 0 \},  
\end{align*}
where $E_i$ is the functor $\gmod{R(\beta)} \to \gmod{R(\beta-\alpha_i)}$ defined by $E_iM = \Res_{\alpha_i,\beta-\alpha_i}M$. 
This crystal $\mathcal{B}$ is isomorphic to the crystal $B(\infty)$. 
Moreover, for any simple module $L \in \gmod{R(\beta)}$ with $\varepsilon_i(L) = n$, we have 
\[
\tilde{e}_i^{n} (L) = E_i^{(n)}L, \ L(i^n)\otimes \tilde{e}_i^n(L)\simeq E_i^nL
\]
where $E_i^{(n)}L = \Hom_{R(n\alpha_i)} (P(i^n), e(n\alpha_i,\beta-n\alpha_i)L)$. 
\end{theorem}

\begin{proof}[Proof of Lemma \ref{lem:string}]
We may ignore grading.
We proceed by induction on $\wt L$. 
If $\wt L =0$, the assertion is trivial. 
We assume $\wt L \neq 0$. 
We expand $[L\circ L]$ in $K(\gmod{R})_{q = 1}$ as 
\[
[L \circ L] = \sum_{L' \in \mathcal{B}} a_{L'} [L'] \ (a_{L'} \in \mathbb{Z}_{\geq 0}). 
\]
By Theorem \ref{thm:crystal}, there exists $i \in I$ such that $\varepsilon_i(L) \neq 0$. 
Put $n = \varepsilon_i(L)$ and $L_0 = \tilde{e}_i^n(L) = E_i^{(n)}L$. 
By considering the Mackey-filtration \cite[Proposition 2.18]{MR2525917}, 
we have 
\[
E_i^{2n+1}(L \circ L) = 0, \ E_i^{2n}(L\circ L) = (L(i^n) \circ L(i^n)) \otimes (L_0 \circ L_0) \simeq L(i^{2n})\otimes (L_0 \circ L_0). 
\]
Hence, 
\[
E_i^{(2n)}(L\circ L) \simeq L_0 \circ L_0. 
\]
It follows that $a_{L'} = 0$ if $\varepsilon_i(L') \geq 2n+1$ and 
\[
[L_0 \circ L_0] = \sum_{L' \in \mathcal{B}, \varepsilon_i(L') = 2n} a_{L'} [\tilde{e}_i^{2n}L']. 
\] 
By the induction hypothesis, there exists $L' \in \mathcal{B}$ such that $\varepsilon_i(L') = 2n$ and $a_{L'} = 1$. 
The lemma is proved. 
\end{proof}

\subsection{The $A_1^{(1)}$ case} \label{sub:A1case}

In this subsection, assume that $\mathbf{k}$ is a field. 
We address the case of $X_N^{(1)} = A_1^{(1)}$ and construct an endomorphism $\mathsf{T}=\mathsf{T}_1$ of $L_1^e(\delta)^{\circ 2}$ satisfying (a) and (b). 
Since $\delta-\alpha_1 = \alpha_0$, we have 
\[
L^e(\delta-\alpha_1) = L(0), \ \Delta^e(\delta-\alpha_1) = R(\alpha_0). 
\]

\begin{lemma} \label{lem:minusculestr}
The module $L_1^e(\delta)$ is spanned by a single vector $v$ as a $\mathbf{k}$-vector space, 
and the action of $R(\delta)$ is given by 
\begin{align*}
e(0,1)v = v, \ e(1,0) v = 0, \ \tau_1 v = 0, \ x_1 v = 0, \ x_2 v = 0. 
\end{align*}
\end{lemma}

\begin{proof}
It is straightforward to see that there exists a one-dimensional module $L = \mathbf{k}v$ given by the formulas of the lemma. 
By induction-restriction adjunction, we have a surjective homomorphism $L(0) \circ L(1) \to L$. 
Hence, $L \simeq \hd (L(0) \circ L(1)) = L_1^e(\delta)$ (see Proposition \ref{prop:minuscule}). 
\end{proof}

By Proposition \ref{prop:imaginarystd}, we have a short exact sequence 
\[
0 \to q^2 R(\alpha_1) \circ R(\alpha_0) \to R(\alpha_0) \circ R(\alpha_1) \to \Delta_1^e(\delta) \to 0.
\]
Note that $R(\alpha_j) \circ R(\alpha_k) \ (j, k \in I)$ is canonically identified with $R(\alpha_j+\alpha_k)e(j,k)$, 
and the injective homomorphism above coincides with the right multiplication by $\tau_1$. 

\begin{lemma}
There uniquely exists an endomorphism $\tau$ of $\Delta_1^e(\delta)^{\circ 2}$ (resp.\ $L_1^e(\delta)^{\circ 2}$) such that 
\[
\tau(u \boxtimes u') = \tau_{w[2,2]}(u' \boxtimes u) \ (u, u' \in \Delta_1^e(\delta) (\text{resp. $L_1^e(\delta)$})). 
\]
\end{lemma}

\begin{proof}
By Lemma \ref{lem:tensorspacehom}, 
it suffices to prove the lemma for $\Delta_1^e(\delta)$. 
By the discussion above, we have 
\[
\Delta_1^e(\delta)^{\circ 2} \simeq R(2\delta)e(0,1,0,1)/K, \ K = R(2\delta)\tau_1e(0,1,0,1) + R(2\delta)\tau_3 e(0,1,0,1). 
\]
Hence, it suffices to prove that $K \tau_{w[2,2]} \subset K$. 
Note that 
\[
\tau_{w[2,2]} = \tau_2 \tau_1 \tau_3 \tau_2 = \tau_2 \tau_3 \tau_1 \tau_2. 
\]
For $u \in R(2\delta)$, we have 
\begin{align*}
&u\tau_1 \tau_{w[2,2]}e(0,1,0,1) \\
&= u \tau_1 \tau_2 \tau_1 \tau_3 \tau_2 e(0,1,0,1) \\
&= u\tau_2 \tau_1 \tau_2 \tau_3 \tau_2 e(0,1,0,1) \\
&= u\tau_2 \tau_1 \tau_3 \tau_2 \tau_3 e(0,1,0,1) - u\tau_2 \tau_1 \overline{Q}_{1,0,1}(x_2,x_3,x_4) e(0,1,0,1) \\
&= (u\tau_2 \tau_1 \tau_3 \tau_2) \tau_3 e(0,1,0,1) - (u\tau_2 \overline{Q}_{1,0,1} (x_1,x_3,x_4))\tau_1 e(0,1,0,1), 
\end{align*}
which belongs to $K$. 
Similarly, $u \tau_3 \tau_{w[2,2]}e(0,1,0,1)$ belongs to $K$. 
\end{proof}

We define $\mathsf{T} = \tau +t$. 
By definition, $\mathsf{T}$ is nonconstant. 
In the rest of this subsection, we prove that $\mathsf{T}$ satisfies (a) and (b) by direct calculations. 
It suffices to show these relations on the generating vector $v^{\boxtimes n} \in L_i^e(\delta)^{\circ n} \ (n = 2,3)$. 
We use the following diagrammatic presentation.
An element of $R(n\delta)$ is represented by a linear combination of dotted planar diagrams decorated by $I$, following \cite[Section 2.1]{MR2525917}, 
where each element $e(\nu), x_k^ae(\nu)$ and $\tau_ke(\nu)$ corresponds to 
\[
\xy 0;/r.17pc/: 
(0,0)*{\sline};
(0,-4)*{\scriptstyle \nu_1}; 
(4,0)*{\sline};
(4,-4)*{\scriptstyle \nu_2};
\endxy \cdots \xy 0;/r.17pc/: 
(0,0)*{\sline};
(0,-4)*{\scriptstyle \nu_{2n}};
\endxy, \ \xy 0;/r.17pc/: 
(0,0)*{\sline};
(0,-4)*{\scriptstyle \nu_1};
\endxy \cdots \xy 0;/r.17pc/: 
(0,0)*{\sdot};
(2,0)*{\scriptstyle a};
(0,-4)*{\scriptstyle \nu_k};
\endxy \cdots \xy 0;/r.17pc/: 
(0,0)*{\sline};
(0,-4)*{\scriptstyle \nu_{2n}};
\endxy, \ \xy 0;/r.17pc/: 
(0,0)*{\sline};
(0,-4)*{\scriptstyle \nu_1};
\endxy \cdots \xy 0;/r.17pc/: 
(0,0)*{\cross};
(-3,-4)*{\scriptstyle \nu_k};
(3,-4)*{\scriptstyle \nu_{k+1}};
\endxy \cdots \xy 0;/r.17pc/: 
(0,0)*{\sline};
(0,-4)*{\scriptstyle \nu_{2n}};
\endxy, 
\]
and the multiplication $u_1 u_2$ in $R(n\delta)$ is depicted by placing the diagram of $u_1$ on top of that of $u_2$. 
For $u \in R(n\delta)$, the element 
\[
u (v^{\boxtimes n}) = ue(0,1,0,1,\ldots,0,1) (v^{\boxtimes n})
\]
in $L_i^e(\delta)^{\circ n}$ is represented by placing the diagram of $ue(0,1,0,1, \ldots, 0,1)$ on top of 
\[
\xy 0;/r.17pc/:
(-4,0)*{\fcolorbox{black}{white}{$v$}};
(4,0)*{\fcolorbox{black}{white}{$v$}}; 
\endxy \cdots 
\xy 0;/r.17pc/:
(0,0)*{\fcolorbox{black}{white}{$v$}};
\endxy \quad \text{($n$ boxes).} 
\]
Since the lower endpoints of the diagram of $ue(0,1,\ldots,0,1)$ are always labelled with $(0,1,\ldots,0,1)$, 
we shall omit this decoration. 
For instance, the relations in Lemma \ref{lem:minusculestr} are represented by
\[
\xy 0;/r.17pc/:
(0,0)*{\xybox{
(0,2)*{\cross};
(0,-2.5)*{\fcolorbox{black}{white}{$v$}};  
}} \endxy \ = \ \xy 0;/r.17pc/:
(0,0)*{\xybox{
(-2,2)*{\sdot};
(2,2)*{\sline};
(0,-2.5)*{\fcolorbox{black}{white}{$v$}};  
}} \endxy \ = \ \xy 0;/r.17pc/:
(0,0)*{\xybox{
(-2,2)*{\sline};
(2,2)*{\sdot};
(0,-2.5)*{\fcolorbox{black}{white}{$v$}};  
}} \endxy \ = 0. 
\]

The quadratic relation of (a), 
\[
(\mathsf{T}+1)(\mathsf{T}-t) = 0, 
\]
immediately follows from the lemma below:

\begin{lemma}
We have $\tau^2 = - (1+t) \tau$. 
\end{lemma}

\begin{proof}
\begin{align*}
\xy 0;/r.17pc/:
(0,0)*{\xybox{(0,6)*{\stau};
(0,18)*{\stau};
(-4,-2.5)*{\fcolorbox{black}{white}{$v$}};
(4,-2.5)*{\fcolorbox{black}{white}{$v$}};  
}}
\endxy =& -t \ \xy 0;/r.17pc/:
(0,0)*{\xybox{
(0,2)*{\cross};
(-6,2)*{\sline};
(6,2)*{\sline};
(-4,6)*{\cross};
(4,6)*{\cross};
(-6,10)*{\sline};
(-2,10)*{\sdot};
(0,10)*{\scriptstyle 2};
(2,10)*{\sline};
(6,10)*{\sline};
(-4,14)*{\cross};
(4,14)*{\cross};
(0,18)*{\cross};
(-6,18)*{\sline};
(6,18)*{\sline};
(-4,-2.5)*{\fcolorbox{black}{white}{$v$}};
(4,-2.5)*{\fcolorbox{black}{white}{$v$}}; 
}}
\endxy \ + (1+t) \ \xy 0;/r.17pc/:
(0,0)*{\xybox{
(0,2)*{\cross};
(-6,2)*{\sline};
(6,2)*{\sline};
(-4,6)*{\cross};
(4,6)*{\cross};
(-6,10)*{\sline};
(-2,10)*{\sdot};
(2,10)*{\sdot};
(6,10)*{\sline};
(-4,14)*{\cross};
(4,14)*{\cross};
(-6,18)*{\sline};
(0,18)*{\cross};
(6,18)*{\sline};
(-4,-2.5)*{\fcolorbox{black}{white}{$v$}};
(4,-2.5)*{\fcolorbox{black}{white}{$v$}};
}}
\endxy \ - \ \xy 0;/r.17pc/:
(0,0)*{\xybox{
(0,2)*{\cross};
(-6,2)*{\sline};
(6,2)*{\sline};
(-4,6)*{\cross};
(4,6)*{\cross};
(-6,10)*{\sline};
(-2,10)*{\sline};
(2,10)*{\sdot};
(4,10)*{\scriptstyle 2};
(6,10)*{\sline};
(-4,14)*{\cross};
(4,14)*{\cross};
(-6,18)*{\sline};
(0,18)*{\cross};
(6,18)*{\sline};
(-4,-2.5)*{\fcolorbox{black}{white}{$v$}};
(4,-2.5)*{\fcolorbox{black}{white}{$v$}};
}}
\endxy \\
=& -t \left( \
\xy 0;/r.17pc/:
(0,0)*{\xybox{
(0,2)*{\cross};
(-6,2)*{\sdot};
(-4,2)*{\scriptstyle 2};
(6,2)*{\sline};
(-4,6)*{\cross};
(4,6)*{\cross};
(-6,10)*{\sline};
(-2,10)*{\sline};
(2,10)*{\sline};
(6,10)*{\sline};
(-4,14)*{\cross};
(4,14)*{\cross};
(-6,18)*{\sline};
(0,18)*{\cross};
(6,18)*{\sline};
(-4,-2.5)*{\fcolorbox{black}{white}{$v$}};
(4,-2.5)*{\fcolorbox{black}{white}{$v$}};
}}
\endxy \ + \ \xy 0;/r.17pc/:
(0,0)*{\xybox{
(0,2)*{\cross};
(-6,2)*{\sline};
(6,2)*{\sline};
(-6,6)*{\sdot};
(-2,6)*{\sline};
(4,6)*{\cross};
(-6,10)*{\sline};
(-2,10)*{\sline};
(2,10)*{\sline};
(6,10)*{\sline};
(-4,14)*{\cross};
(4,14)*{\cross};
(-6,18)*{\sline};
(0,18)*{\cross};
(6,18)*{\sline};
(-4,-2.5)*{\fcolorbox{black}{white}{$v$}};
(4,-2.5)*{\fcolorbox{black}{white}{$v$}};
}}
\endxy \ + \ \xy 0;/r.17pc/:
(0,0)*{\xybox{
(0,2)*{\cross};
(-6,2)*{\sline};
(6,2)*{\sline};
(-6,6)*{\sline};
(-2,6)*{\sdot};
(4,6)*{\cross};
(-6,10)*{\sline};
(-2,10)*{\sline};
(2,10)*{\sline};
(6,10)*{\sline};
(-4,14)*{\cross};
(4,14)*{\cross};
(-6,18)*{\sline};
(0,18)*{\cross};
(6,18)*{\sline};
(-4,-2.5)*{\fcolorbox{black}{white}{$v$}};
(4,-2.5)*{\fcolorbox{black}{white}{$v$}};
}}
\endxy
\ \right) + (1+t) \left( \ 
\xy 0;/r.17pc/:
(0,0)*{\xybox{
(0,2)*{\cross};
(-6,2)*{\sdot};
(6,2)*{\sline};
(-4,6)*{\cross};
(4,6)*{\cross};
(-6,10)*{\sline};
(-2,10)*{\sline};
(2,10)*{\sdot};
(6,10)*{\sline};
(-4,14)*{\cross};
(4,14)*{\cross};
(-6,18)*{\sline};
(0,18)*{\cross};
(6,18)*{\sline};
(-4,-2.5)*{\fcolorbox{black}{white}{$v$}};
(4,-2.5)*{\fcolorbox{black}{white}{$v$}};
}}
\endxy \ + \
\xy 0;/r.17pc/:
(0,0)*{\xybox{
(0,2)*{\cross};
(-6,2)*{\sline};
(6,2)*{\sline};
(-6,6)*{\sline};
(-2,6)*{\sline};
(4,6)*{\cross};
(-6,10)*{\sline};
(-2,10)*{\sline};
(2,10)*{\sdot};
(6,10)*{\sline};
(-4,14)*{\cross};
(4,14)*{\cross};
(-6,18)*{\sline};
(0,18)*{\cross};
(6,18)*{\sline};
(-4,-2.5)*{\fcolorbox{black}{white}{$v$}};
(4,-2.5)*{\fcolorbox{black}{white}{$v$}};
}}
\endxy
\ \right) \\
&- \left( \
\xy 0;/r.17pc/:
(0,0)*{\xybox{
(0,2)*{\cross};
(-6,2)*{\sline};
(8,2)*{\scriptstyle 2};
(6,2)*{\sdot};
(-4,6)*{\cross};
(4,6)*{\cross};
(-6,10)*{\sline};
(-2,10)*{\sline};
(2,10)*{\sline};
(6,10)*{\sline};
(-4,14)*{\cross};
(4,14)*{\cross};
(-6,18)*{\sline};
(0,18)*{\cross};
(6,18)*{\sline};
(-4,-2.5)*{\fcolorbox{black}{white}{$v$}};
(4,-2.5)*{\fcolorbox{black}{white}{$v$}};
}}
\endxy \ - \ \xy 0;/r.17pc/:
(0,0)*{\xybox{
(0,2)*{\cross};
(-6,2)*{\sline};
(6,2)*{\sline};
(6,6)*{\sdot};
(2,6)*{\sline};
(-4,6)*{\cross};
(-6,10)*{\sline};
(-2,10)*{\sline};
(2,10)*{\sline};
(6,10)*{\sline};
(-4,14)*{\cross};
(4,14)*{\cross};
(-6,18)*{\sline};
(0,18)*{\cross};
(6,18)*{\sline};
(-4,-2.5)*{\fcolorbox{black}{white}{$v$}};
(4,-2.5)*{\fcolorbox{black}{white}{$v$}};
}}
\endxy \ - \ \xy 0;/r.17pc/:
(0,0)*{\xybox{
(0,2)*{\cross};
(-6,2)*{\sline};
(6,2)*{\sline};
(6,6)*{\sline};
(2,6)*{\sdot};
(-4,6)*{\cross};
(-6,10)*{\sline};
(-2,10)*{\sline};
(2,10)*{\sline};
(6,10)*{\sline};
(-4,14)*{\cross};
(4,14)*{\cross};
(-6,18)*{\sline};
(0,18)*{\cross};
(6,18)*{\sline};
(-4,-2.5)*{\fcolorbox{black}{white}{$v$}};
(4,-2.5)*{\fcolorbox{black}{white}{$v$}};
}}
\endxy
\ \right) \\
=& -t \left( 0+0+ \ 
\xy 0;/r.17pc/:
(0,0)*{\xybox{
(0,2)*{\cross};
(-6,2)*{\sline};
(6,2)*{\sline};
(1,1)*{\scriptstyle \bullet};
(-6,6)*{\sline};
(-2,6)*{\sline};
(4,6)*{\cross};
(-6,10)*{\sline};
(-2,10)*{\sline};
(2,10)*{\sline};
(6,10)*{\sline};
(-4,14)*{\cross};
(4,14)*{\cross};
(-6,18)*{\sline};
(0,18)*{\cross};
(6,18)*{\sline};
(-4,-2.5)*{\fcolorbox{black}{white}{$v$}};
(4,-2.5)*{\fcolorbox{black}{white}{$v$}};
}}
\endxy \  
\right) + (1+t) \left( 0 + \left( \
\xy 0;/r.17pc/:
(0,0)*{\xybox{
(0,2)*{\cross};
(-6,2)*{\sline};
(6,2)*{\sdot};
(-6,6)*{\sline};
(-2,6)*{\sline};
(4,6)*{\cross};
(-6,10)*{\sline};
(-2,10)*{\sline};
(2,10)*{\sline};
(6,10)*{\sline};
(-4,14)*{\cross};
(4,14)*{\cross};
(-6,18)*{\sline};
(0,18)*{\cross};
(6,18)*{\sline};
(-4,-2.5)*{\fcolorbox{black}{white}{$v$}};
(4,-2.5)*{\fcolorbox{black}{white}{$v$}};
}}
\endxy \ - \ \xy 0;/r.17pc/:
(0,0)*{\xybox{
(0,2)*{\cross};
(-6,2)*{\sline};
(6,2)*{\sline};
(-6,6)*{\sline};
(-2,6)*{\sline};
(2,6)*{\sline};
(6,6)*{\sline};
(-6,10)*{\sline};
(-2,10)*{\sline};
(2,10)*{\sline};
(6,10)*{\sline};
(-4,14)*{\cross};
(4,14)*{\cross};
(-6,18)*{\sline};
(0,18)*{\cross};
(6,18)*{\sline};
(-4,-2.5)*{\fcolorbox{black}{white}{$v$}};
(4,-2.5)*{\fcolorbox{black}{white}{$v$}};
}}
\endxy
\ \right)
\right) \\
&- \left(0-0 - 
\xy 0;/r.17pc/:
(0,0)*{\xybox{
(0,2)*{\cross};
(-6,2)*{\sline};
(6,2)*{\sline};
(-1,1)*{\scriptstyle \bullet};
(6,6)*{\sline};
(2,6)*{\sline};
(-4,6)*{\cross};
(-6,10)*{\sline};
(-2,10)*{\sline};
(2,10)*{\sline};
(6,10)*{\sline};
(-4,14)*{\cross};
(4,14)*{\cross};
(-6,18)*{\sline};
(0,18)*{\cross};
(6,18)*{\sline};
(-4,-2.5)*{\fcolorbox{black}{white}{$v$}};
(4,-2.5)*{\fcolorbox{black}{white}{$v$}};
}}
\endxy 
\ \right) \\
=& -(1+t) \ \xy 0;/r.17pc/:
(0,0)*{\xybox{(0,6)*{\stau};
(-4,-2.5)*{\fcolorbox{black}{white}{$v$}};
(4,-2.5)*{\fcolorbox{black}{white}{$v$}};  
}}
\endxy. 
\end{align*}
\end{proof}

In order to prove the braid relation of (b), 
\[
\mathsf{T}^{2,3}\mathsf{T}^{1,2}\mathsf{T}^{2,3} = \mathsf{T}^{1,2}\mathsf{T}^{2,3}\mathsf{T}^{1,2}, 
\]
we use the following lemma. 

\begin{lemma} \label{lem:taubraid}
We have 
\[
\tau^{2,3}\tau^{1,2}\tau^{2,3} - \tau^{1,2}\tau^{2,3}\tau^{1,2} = t(\tau^{2,3}-\tau^{1,2}). 
\]
\end{lemma}

Note that the left hand side is 
\begin{align*}
&(\mathsf{T}^{2,3}-t)(\mathsf{T}^{1,2}-t)(\mathsf{T}^{2,3}-t) - (\mathsf{T}^{1,2}-t)(\mathsf{T}^{2,3}-t)(\mathsf{T}^{1,2}-t) \\
&= \mathsf{T}^{2,3}\mathsf{T}^{1,2}\mathsf{T}^{2,3} - \mathsf{T}^{1,2}\mathsf{T}^{2,3}\mathsf{T}^{1,2} -t (\mathsf{T}^{2,3})^2 + t(\mathsf{T}^{1,2})^2 + t^2\mathsf{T}^{2,3}-t^2\mathsf{T}^{1,2} \\
&= \mathsf{T}^{2,3}\mathsf{T}^{1,2}\mathsf{T}^{2,3} - \mathsf{T}^{1,2}\mathsf{T}^{2,3}\mathsf{T}^{1,2} + t(\mathsf{T}^{2,3}- \mathsf{T}^{1,2}) \quad \text{by (a)}. 
\end{align*}
Hence, (b) follows from this lemma. 

\begin{proof}[Proof of Lemma \ref{lem:taubraid}]
The left hand side is 
\begin{align*}
\xy 0;/r.17pc/:
(0,0)*{\xybox{
(4,6)*{\stau};
(-10,2)*{\sline};
(-10,6)*{\sline};
(-10,10)*{\sline};
(-6,2)*{\sline};
(-6,6)*{\sline};
(-6,10)*{\sline};
(10,14)*{\sline};
(10,18)*{\sline};
(10,22)*{\sline};
(6,14)*{\sline};
(6,18)*{\sline};
(6,22)*{\sline};
(-10,26)*{\sline};
(-10,30)*{\sline};
(-10,34)*{\sline};
(-6,26)*{\sline};
(-6,30)*{\sline};
(-6,34)*{\sline};
(-4,18)*{\stau};
(4,30)*{\stau};
(-8,-2.5)*{\fcolorbox{black}{white}{$v$}};
(0,-2.5)*{\fcolorbox{black}{white}{$v$}};
(8,-2.5)*{\fcolorbox{black}{white}{$v$}};  
}}
\endxy \ - \ \xy 0;/r.17pc/:
(0,0)*{\xybox{
(-4,6)*{\stau};
(10,2)*{\sline};
(10,6)*{\sline};
(10,10)*{\sline};
(6,2)*{\sline};
(6,6)*{\sline};
(6,10)*{\sline};
(-10,14)*{\sline};
(-10,18)*{\sline};
(-10,22)*{\sline};
(-6,14)*{\sline};
(-6,18)*{\sline};
(-6,22)*{\sline};
(10,26)*{\sline};
(10,30)*{\sline};
(10,34)*{\sline};
(6,26)*{\sline};
(6,30)*{\sline};
(6,34)*{\sline};
(4,18)*{\stau};
(-4,30)*{\stau};
(-8,-2.5)*{\fcolorbox{black}{white}{$v$}};
(0,-2.5)*{\fcolorbox{black}{white}{$v$}};
(8,-2.5)*{\fcolorbox{black}{white}{$v$}};  
}}
\endxy \ = X + Y + Z + W, 
\end{align*} 
where 
\begin{align*}
X = \ \xy 0;/r.17pc/:
(0,0)*{\xybox{
(4,6)*{\stau};
(-10,2)*{\sline};
(-10,6)*{\sline};
(-10,10)*{\sline};
(-6,2)*{\sline};
(-6,6)*{\sline};
(-6,10)*{\sline};
(10,14)*{\sline};
(10,18)*{\sline};
(10,22)*{\sline};
(6,14)*{\sline};
(6,18)*{\sline};
(6,22)*{\sline};
(-10,26)*{\sline};
(-10,30)*{\sline};
(-10,34)*{\sline};
(-6,26)*{\sline};
(-6,30)*{\sline};
(-6,34)*{\sline};
(-4,18)*{\stau};
(4,30)*{\stau};
(-8,-2.5)*{\fcolorbox{black}{white}{$v$}};
(0,-2.5)*{\fcolorbox{black}{white}{$v$}};
(8,-2.5)*{\fcolorbox{black}{white}{$v$}};  
}}
\endxy \ - \ \xy 0;/r.17pc/:
(0,0)*{\xybox{
(4,2)*{\cross};
(-2,2)*{\sline};
(10,2)*{\sline};
(0,6)*{\cross};
(8,6)*{\cross};
(2,10)*{\sline};
(6,10)*{\sline};
(10,10)*{\sline};
(-10,2)*{\sline};
(-10,6)*{\sline};
(-10,10)*{\sline};
(-6,2)*{\sline};
(-6,6)*{\sline};
(-4,10)*{\cross};
(10,14)*{\sline};
(10,18)*{\sline};
(10,22)*{\sline};
(-10,26)*{\sline};
(-10,30)*{\sline};
(-10,34)*{\sline};
(-4,26)*{\cross};
(-6,30)*{\sline};
(-6,34)*{\sline};
(2,26)*{\sline};
(6,26)*{\sline};
(10,26)*{\sline};
(0,30)*{\cross};
(8,30)*{\cross};
(4,34)*{\cross};
(10,34)*{\sline};
(-2,34)*{\sline};
(-10,14)*{\sline};
(-6,14)*{\sline};
(-8,18)*{\cross};
(-10,22)*{\sline};
(-6,22)*{\sline};
(0,14)*{\cross};
(6,14)*{\sline};
(-2,18)*{\sline};
(4,18)*{\cross};
(0,22)*{\cross};
(6,22)*{\sline};
(-8,-2.5)*{\fcolorbox{black}{white}{$v$}};
(0,-2.5)*{\fcolorbox{black}{white}{$v$}};
(8,-2.5)*{\fcolorbox{black}{white}{$v$}};  
}}
\endxy \ , &\ Y = \ \xy 0;/r.17pc/:
(0,0)*{\xybox{
(4,2)*{\cross};
(-2,2)*{\sline};
(10,2)*{\sline};
(0,6)*{\cross};
(8,6)*{\cross};
(2,10)*{\sline};
(6,10)*{\sline};
(10,10)*{\sline};
(-10,2)*{\sline};
(-10,6)*{\sline};
(-10,10)*{\sline};
(-6,2)*{\sline};
(-6,6)*{\sline};
(-4,10)*{\cross};
(10,14)*{\sline};
(10,18)*{\sline};
(10,22)*{\sline};
(-10,26)*{\sline};
(-10,30)*{\sline};
(-10,34)*{\sline};
(-4,26)*{\cross};
(-6,30)*{\sline};
(-6,34)*{\sline};
(2,26)*{\sline};
(6,26)*{\sline};
(10,26)*{\sline};
(0,30)*{\cross};
(8,30)*{\cross};
(4,34)*{\cross};
(10,34)*{\sline};
(-2,34)*{\sline};
(-10,14)*{\sline};
(-6,14)*{\sline};
(-8,18)*{\cross};
(-10,22)*{\sline};
(-6,22)*{\sline};
(0,14)*{\cross};
(6,14)*{\sline};
(-2,18)*{\sline};
(4,18)*{\cross};
(0,22)*{\cross};
(6,22)*{\sline};
(-8,-2.5)*{\fcolorbox{black}{white}{$v$}};
(0,-2.5)*{\fcolorbox{black}{white}{$v$}};
(8,-2.5)*{\fcolorbox{black}{white}{$v$}};  
}}
\endxy \ - \ \xy 0;/r.17pc/:
(0,0)*{\xybox{
(4,2)*{\cross};
(-2,10)*{\sline};
(10,2)*{\sline};
(-10,2)*{\sline};
(2,6)*{\sline};
(-2,6)*{\sline};
(8,6)*{\cross};
(-8,6)*{\cross};
(2,10)*{\sline};
(6,10)*{\sline};
(10,10)*{\sline};
(-6,10)*{\sline};
(-10,10)*{\sline};
(-4,2)*{\cross};
(10,14)*{\sline};
(10,18)*{\sline};
(10,22)*{\sline};
(-10,14)*{\sline};
(-10,18)*{\sline};
(-10,22)*{\sline};
(-4,34)*{\cross};
(2,26)*{\sline};
(6,26)*{\sline};
(10,26)*{\sline};
(-6,26)*{\sline};
(-10,26)*{\sline};
(-2,30)*{\sline};
(2,30)*{\sline};
(8,30)*{\cross};
(-8,30)*{\cross};
(4,34)*{\cross};
(10,34)*{\sline};
(-10,34)*{\sline};
(-2,26)*{\sline};
(0,14)*{\cross};
(6,14)*{\sline};
(-6,14)*{\sline};
(4,18)*{\cross};
(0,22)*{\cross};
(6,22)*{\sline};
(-6,22)*{\sline};
(-4,18)*{\cross};
(-8,-2.5)*{\fcolorbox{black}{white}{$v$}};
(0,-2.5)*{\fcolorbox{black}{white}{$v$}};
(8,-2.5)*{\fcolorbox{black}{white}{$v$}};  
}}
\endxy \ , \\
Z = \ \xy 0;/r.17pc/:
(0,0)*{\xybox{
(4,2)*{\cross};
(-2,10)*{\sline};
(10,2)*{\sline};
(-10,2)*{\sline};
(2,6)*{\sline};
(-2,6)*{\sline};
(8,6)*{\cross};
(-8,6)*{\cross};
(2,10)*{\sline};
(6,10)*{\sline};
(10,10)*{\sline};
(-6,10)*{\sline};
(-10,10)*{\sline};
(-4,2)*{\cross};
(10,14)*{\sline};
(10,18)*{\sline};
(10,22)*{\sline};
(-10,14)*{\sline};
(-10,18)*{\sline};
(-10,22)*{\sline};
(-4,34)*{\cross};
(2,26)*{\sline};
(6,26)*{\sline};
(10,26)*{\sline};
(-6,26)*{\sline};
(-10,26)*{\sline};
(-2,30)*{\sline};
(2,30)*{\sline};
(8,30)*{\cross};
(-8,30)*{\cross};
(4,34)*{\cross};
(10,34)*{\sline};
(-10,34)*{\sline};
(-2,26)*{\sline};
(0,14)*{\cross};
(6,14)*{\sline};
(-6,14)*{\sline};
(4,18)*{\cross};
(0,22)*{\cross};
(6,22)*{\sline};
(-6,22)*{\sline};
(-4,18)*{\cross};
(-8,-2.5)*{\fcolorbox{black}{white}{$v$}};
(0,-2.5)*{\fcolorbox{black}{white}{$v$}};
(8,-2.5)*{\fcolorbox{black}{white}{$v$}};  
}}
\endxy \ - \ \xy 0;/r.17pc/:
(0,0)*{\xybox{
(-4,2)*{\cross};
(2,2)*{\sline};
(-10,2)*{\sline};
(-0,6)*{\cross};
(-8,6)*{\cross};
(-2,10)*{\sline};
(-6,10)*{\sline};
(-10,10)*{\sline};
(10,2)*{\sline};
(10,6)*{\sline};
(10,10)*{\sline};
(6,2)*{\sline};
(6,6)*{\sline};
(4,10)*{\cross};
(-10,14)*{\sline};
(-10,18)*{\sline};
(-10,22)*{\sline};
(10,26)*{\sline};
(10,30)*{\sline};
(10,34)*{\sline};
(4,26)*{\cross};
(6,30)*{\sline};
(6,34)*{\sline};
(-2,26)*{\sline};
(-6,26)*{\sline};
(-10,26)*{\sline};
(-0,30)*{\cross};
(-8,30)*{\cross};
(-4,34)*{\cross};
(-10,34)*{\sline};
(2,34)*{\sline};
(10,14)*{\sline};
(6,14)*{\sline};
(8,18)*{\cross};
(10,22)*{\sline};
(6,22)*{\sline};
(-0,14)*{\cross};
(-6,14)*{\sline};
(2,18)*{\sline};
(-4,18)*{\cross};
(-0,22)*{\cross};
(-6,22)*{\sline};
(-8,-2.5)*{\fcolorbox{black}{white}{$v$}};
(0,-2.5)*{\fcolorbox{black}{white}{$v$}};
(8,-2.5)*{\fcolorbox{black}{white}{$v$}};  
}}
\endxy \ &, \ W = \xy 0;/r.17pc/:
(0,0)*{\xybox{
(-4,2)*{\cross};
(2,2)*{\sline};
(-10,2)*{\sline};
(-0,6)*{\cross};
(-8,6)*{\cross};
(-2,10)*{\sline};
(-6,10)*{\sline};
(-10,10)*{\sline};
(10,2)*{\sline};
(10,6)*{\sline};
(10,10)*{\sline};
(6,2)*{\sline};
(6,6)*{\sline};
(4,10)*{\cross};
(-10,14)*{\sline};
(-10,18)*{\sline};
(-10,22)*{\sline};
(10,26)*{\sline};
(10,30)*{\sline};
(10,34)*{\sline};
(4,26)*{\cross};
(6,30)*{\sline};
(6,34)*{\sline};
(-2,26)*{\sline};
(-6,26)*{\sline};
(-10,26)*{\sline};
(-0,30)*{\cross};
(-8,30)*{\cross};
(-4,34)*{\cross};
(-10,34)*{\sline};
(2,34)*{\sline};
(10,14)*{\sline};
(6,14)*{\sline};
(8,18)*{\cross};
(10,22)*{\sline};
(6,22)*{\sline};
(-0,14)*{\cross};
(-6,14)*{\sline};
(2,18)*{\sline};
(-4,18)*{\cross};
(-0,22)*{\cross};
(-6,22)*{\sline};
(-8,-2.5)*{\fcolorbox{black}{white}{$v$}};
(0,-2.5)*{\fcolorbox{black}{white}{$v$}};
(8,-2.5)*{\fcolorbox{black}{white}{$v$}};  
}}
\endxy \ -  \ \xy 0;/r.17pc/:
(0,0)*{\xybox{
(-4,6)*{\stau};
(10,2)*{\sline};
(10,6)*{\sline};
(10,10)*{\sline};
(6,2)*{\sline};
(6,6)*{\sline};
(6,10)*{\sline};
(-10,14)*{\sline};
(-10,18)*{\sline};
(-10,22)*{\sline};
(-6,14)*{\sline};
(-6,18)*{\sline};
(-6,22)*{\sline};
(10,26)*{\sline};
(10,30)*{\sline};
(10,34)*{\sline};
(6,26)*{\sline};
(6,30)*{\sline};
(6,34)*{\sline};
(4,18)*{\stau};
(-4,30)*{\stau};
(-8,-2.5)*{\fcolorbox{black}{white}{$v$}};
(0,-2.5)*{\fcolorbox{black}{white}{$v$}};
(8,-2.5)*{\fcolorbox{black}{white}{$v$}};  
}}
\endxy \ . 
\end{align*}
Note that the two diagrams in $X$ or $Y$ differ by the position of the third strand,
whereas those in $Z$ or $W$ differ by the position of the fourth strand. 
By the defining relations of the quiver Hecke algebra, both $Y$ and $Z$ are zero. 

To compute $X$ and $W$, we temporarily consider a general choice of polynomial $Q_{0,1}(u_1,u_2) = au_1^2 + bu_1u_2+ cu_2^2 \ (a,c \in \mathbf{k}^{\times},  b \in \mathbf{k})$, 
and write $R_{a,b,c}(3\delta) = R_Q(3\delta)$. 
Since $\overline{Q}_{0,1,0}(u_1,u_2,u_3) = au_1 + au_3 + bu_2$ and $\overline{Q}_{1,0,1}(u_1,u_2,u_3) = cu_1 + cu_3+bu_2$, we have 
\begin{align*}
X =& c \ \xy 0;/r.17pc/:
(0,0)*{\xybox{
(4,2)*{\cross};
(-2,2)*{\sline};
(10,2)*{\sline};
(0,6)*{\cross};
(8,6)*{\cross};
(2,10)*{\sline};
(6,10)*{\sline};
(10,10)*{\sline};
(-10,2)*{\sline};
(-10,6)*{\sline};
(-10,10)*{\sline};
(-6,2)*{\sline};
(-6,6)*{\sline};
(-4,10)*{\cross};
(-10,18)*{\sline};
(-10,22)*{\sline};
(-10,26)*{\sline};
(-4,18)*{\cross};
(-6,22)*{\sline};
(-6,26)*{\sline};
(2,18)*{\sline};
(6,18)*{\sline};
(10,18)*{\sline};
(0,22)*{\cross};
(8,22)*{\cross};
(4,26)*{\cross};
(10,26)*{\sline};
(-2,26)*{\sline};
(-8,14)*{\cross};
(-2,14)*{\sdot};
(2,14)*{\sline};
(6,14)*{\sline};
(10,14)*{\sline};
(-8,-2.5)*{\fcolorbox{black}{white}{$v$}};
(0,-2.5)*{\fcolorbox{black}{white}{$v$}};
(8,-2.5)*{\fcolorbox{black}{white}{$v$}};  
}}
\endxy \ +c \ \xy 0;/r.17pc/:
(0,0)*{\xybox{
(4,2)*{\cross};
(-2,2)*{\sline};
(10,2)*{\sline};
(0,6)*{\cross};
(8,6)*{\cross};
(2,10)*{\sline};
(6,10)*{\sline};
(10,10)*{\sline};
(-10,2)*{\sline};
(-10,6)*{\sline};
(-10,10)*{\sline};
(-6,2)*{\sline};
(-6,6)*{\sline};
(-4,10)*{\cross};
(-10,18)*{\sline};
(-10,22)*{\sline};
(-10,26)*{\sline};
(-4,18)*{\cross};
(-6,22)*{\sline};
(-6,26)*{\sline};
(2,18)*{\sline};
(6,18)*{\sline};
(10,18)*{\sline};
(0,22)*{\cross};
(8,22)*{\cross};
(4,26)*{\cross};
(10,26)*{\sline};
(-2,26)*{\sline};
(-8,14)*{\cross};
(-2,14)*{\sline};
(2,14)*{\sline};
(6,14)*{\sdot};
(10,14)*{\sline};
(-8,-2.5)*{\fcolorbox{black}{white}{$v$}};
(0,-2.5)*{\fcolorbox{black}{white}{$v$}};
(8,-2.5)*{\fcolorbox{black}{white}{$v$}};  
}}
\endxy \ + b \ \xy 0;/r.17pc/:
(0,0)*{\xybox{
(4,2)*{\cross};
(-2,2)*{\sline};
(10,2)*{\sline};
(0,6)*{\cross};
(8,6)*{\cross};
(2,10)*{\sline};
(6,10)*{\sline};
(10,10)*{\sline};
(-10,2)*{\sline};
(-10,6)*{\sline};
(-10,10)*{\sline};
(-6,2)*{\sline};
(-6,6)*{\sline};
(-4,10)*{\cross};
(-10,18)*{\sline};
(-10,22)*{\sline};
(-10,26)*{\sline};
(-4,18)*{\cross};
(-6,22)*{\sline};
(-6,26)*{\sline};
(2,18)*{\sline};
(6,18)*{\sline};
(10,18)*{\sline};
(0,22)*{\cross};
(8,22)*{\cross};
(4,26)*{\cross};
(10,26)*{\sline};
(-2,26)*{\sline};
(-8,14)*{\cross};
(-2,14)*{\sline};
(2,14)*{\sdot};
(6,14)*{\sline};
(10,14)*{\sline};
(-8,-2.5)*{\fcolorbox{black}{white}{$v$}};
(0,-2.5)*{\fcolorbox{black}{white}{$v$}};
(8,-2.5)*{\fcolorbox{black}{white}{$v$}};  
}}
\endxy \\
=& c \ \xy 0;/r.17pc/:
(0,0)*{\xybox{
(4,2)*{\cross};
(-2,2)*{\sline};
(10,2)*{\sline};
(0,6)*{\cross};
(8,6)*{\cross};
(2,10)*{\sline};
(6,10)*{\sline};
(10,10)*{\sline};
(-10,2)*{\sline};
(-10,6)*{\sline};
(-10,10)*{\sline};
(-6,2)*{\sdot};
(-6,6)*{\sline};
(-4,10)*{\cross};
(-10,18)*{\sline};
(-10,22)*{\sline};
(-10,26)*{\sline};
(-4,18)*{\cross};
(-6,22)*{\sline};
(-6,26)*{\sline};
(2,18)*{\sline};
(6,18)*{\sline};
(10,18)*{\sline};
(0,22)*{\cross};
(8,22)*{\cross};
(4,26)*{\cross};
(10,26)*{\sline};
(-2,26)*{\sline};
(-8,14)*{\cross};
(-2,14)*{\sline};
(2,14)*{\sline};
(6,14)*{\sline};
(10,14)*{\sline};
(-8,-2.5)*{\fcolorbox{black}{white}{$v$}};
(0,-2.5)*{\fcolorbox{black}{white}{$v$}};
(8,-2.5)*{\fcolorbox{black}{white}{$v$}};  
}}
\endxy \ +c \ \left(\ \xy 0;/r.17pc/:
(0,0)*{\xybox{
(4,2)*{\cross};
(-2,2)*{\sline};
(10,2)*{\sdot};
(0,6)*{\cross};
(8,6)*{\cross};
(2,10)*{\sline};
(6,10)*{\sline};
(10,10)*{\sline};
(-10,2)*{\sline};
(-10,6)*{\sline};
(-10,10)*{\sline};
(-6,2)*{\sline};
(-6,6)*{\sline};
(-4,10)*{\cross};
(-10,18)*{\sline};
(-10,22)*{\sline};
(-10,26)*{\sline};
(-4,18)*{\cross};
(-6,22)*{\sline};
(-6,26)*{\sline};
(2,18)*{\sline};
(6,18)*{\sline};
(10,18)*{\sline};
(0,22)*{\cross};
(8,22)*{\cross};
(4,26)*{\cross};
(10,26)*{\sline};
(-2,26)*{\sline};
(-8,14)*{\cross};
(-2,14)*{\sline};
(2,14)*{\sline};
(6,14)*{\sline};
(10,14)*{\sline};
(-8,-2.5)*{\fcolorbox{black}{white}{$v$}};
(0,-2.5)*{\fcolorbox{black}{white}{$v$}};
(8,-2.5)*{\fcolorbox{black}{white}{$v$}};  
}}
\endxy \ - \ \xy 0;/r.17pc/:
(0,0)*{\xybox{
(4,2)*{\cross};
(-2,2)*{\sline};
(10,2)*{\sline};
(0,6)*{\cross};
(6,6)*{\sline};
(10,6)*{\sline};
(2,10)*{\sline};
(6,10)*{\sline};
(10,10)*{\sline};
(-10,2)*{\sline};
(-10,6)*{\sline};
(-10,10)*{\sline};
(-6,2)*{\sline};
(-6,6)*{\sline};
(-4,10)*{\cross};
(-10,18)*{\sline};
(-10,22)*{\sline};
(-10,26)*{\sline};
(-4,18)*{\cross};
(-6,22)*{\sline};
(-6,26)*{\sline};
(2,18)*{\sline};
(6,18)*{\sline};
(10,18)*{\sline};
(0,22)*{\cross};
(8,22)*{\cross};
(4,26)*{\cross};
(10,26)*{\sline};
(-2,26)*{\sline};
(-8,14)*{\cross};
(-2,14)*{\sline};
(2,14)*{\sline};
(6,14)*{\sline};
(10,14)*{\sline};
(-8,-2.5)*{\fcolorbox{black}{white}{$v$}};
(0,-2.5)*{\fcolorbox{black}{white}{$v$}};
(8,-2.5)*{\fcolorbox{black}{white}{$v$}};  
}}
\endxy
\ \right) + b \left(\ \xy 0;/r.17pc/:
(0,0)*{\xybox{
(4,2)*{\cross};
(-2,2)*{\sdot};
(10,2)*{\sline};
(0,6)*{\cross};
(8,6)*{\cross};
(2,10)*{\sline};
(6,10)*{\sline};
(10,10)*{\sline};
(-10,2)*{\sline};
(-10,6)*{\sline};
(-10,10)*{\sline};
(-6,2)*{\sline};
(-6,6)*{\sline};
(-4,10)*{\cross};
(-10,18)*{\sline};
(-10,22)*{\sline};
(-10,26)*{\sline};
(-4,18)*{\cross};
(-6,22)*{\sline};
(-6,26)*{\sline};
(2,18)*{\sline};
(6,18)*{\sline};
(10,18)*{\sline};
(0,22)*{\cross};
(8,22)*{\cross};
(4,26)*{\cross};
(10,26)*{\sline};
(-2,26)*{\sline};
(-8,14)*{\cross};
(-2,14)*{\sline};
(2,14)*{\sline};
(6,14)*{\sline};
(10,14)*{\sline};
(-8,-2.5)*{\fcolorbox{black}{white}{$v$}};
(0,-2.5)*{\fcolorbox{black}{white}{$v$}};
(8,-2.5)*{\fcolorbox{black}{white}{$v$}};  
}}
\endxy \ + \ \xy 0;/r.17pc/:
(0,0)*{\xybox{
(4,2)*{\cross};
(-2,2)*{\sline};
(10,2)*{\sline};
(-2,6)*{\sline};
(2,6)*{\sline};
(8,6)*{\cross};
(2,10)*{\sline};
(6,10)*{\sline};
(10,10)*{\sline};
(-10,2)*{\sline};
(-10,6)*{\sline};
(-10,10)*{\sline};
(-6,2)*{\sline};
(-6,6)*{\sline};
(-4,10)*{\cross};
(-10,18)*{\sline};
(-10,22)*{\sline};
(-10,26)*{\sline};
(-4,18)*{\cross};
(-6,22)*{\sline};
(-6,26)*{\sline};
(2,18)*{\sline};
(6,18)*{\sline};
(10,18)*{\sline};
(0,22)*{\cross};
(8,22)*{\cross};
(4,26)*{\cross};
(10,26)*{\sline};
(-2,26)*{\sline};
(-8,14)*{\cross};
(-2,14)*{\sline};
(2,14)*{\sline};
(6,14)*{\sline};
(10,14)*{\sline};
(-8,-2.5)*{\fcolorbox{black}{white}{$v$}};
(0,-2.5)*{\fcolorbox{black}{white}{$v$}};
(8,-2.5)*{\fcolorbox{black}{white}{$v$}};  
}}
\endxy
\ \right) \\
=& 0 + c\left( 0 - \left(\ 
\xy 0;/r.17pc/:
(0,0)*{\xybox{
(4,2)*{\cross};
(-2,2)*{\sline};
(10,2)*{\sline};
(0,6)*{\cross};
(6,6)*{\sline};
(10,6)*{\sline};
(2,10)*{\sline};
(6,10)*{\sline};
(10,10)*{\sline};
(-10,2)*{\sline};
(-10,6)*{\sline};
(-8,10)*{\cross};
(-6,2)*{\sline};
(-6,6)*{\sline};
(-2,10)*{\sline};
(-2,18)*{\sline};
(-10,22)*{\sline};
(-10,26)*{\sline};
(-8,18)*{\cross};
(-6,22)*{\sline};
(-6,26)*{\sline};
(2,18)*{\sline};
(6,18)*{\sline};
(10,18)*{\sline};
(0,22)*{\cross};
(8,22)*{\cross};
(4,26)*{\cross};
(10,26)*{\sline};
(-2,26)*{\sline};
(-4,14)*{\cross};
(-10,14)*{\sline};
(2,14)*{\sline};
(6,14)*{\sline};
(10,14)*{\sline};
(-8,-2.5)*{\fcolorbox{black}{white}{$v$}};
(0,-2.5)*{\fcolorbox{black}{white}{$v$}};
(8,-2.5)*{\fcolorbox{black}{white}{$v$}};  
}}
\endxy \ + a \ \xy 0;/r.17pc/:
(0,0)*{\xybox{
(4,2)*{\cross};
(-2,2)*{\sline};
(10,2)*{\sline};
(0,6)*{\cross};
(6,6)*{\sline};
(10,6)*{\sline};
(2,10)*{\sline};
(6,10)*{\sline};
(10,10)*{\sline};
(-10,2)*{\sline};
(-10,6)*{\sline};
(-10,10)*{\sline};
(-6,2)*{\sline};
(-6,6)*{\sline};
(-6,10)*{\sline};
(-2,10)*{\sline};
(-10,18)*{\sline};
(-10,22)*{\sline};
(-10,26)*{\sline};
(-6,18)*{\sline};
(-2,18)*{\sline};
(-6,22)*{\sline};
(-6,26)*{\sline};
(2,18)*{\sline};
(6,18)*{\sline};
(10,18)*{\sline};
(0,22)*{\cross};
(8,22)*{\cross};
(4,26)*{\cross};
(10,26)*{\sline};
(-2,26)*{\sline};
(-10,14)*{\sdot};
(-6,14)*{\sline};
(-2,14)*{\sline};
(2,14)*{\sline};
(6,14)*{\sline};
(10,14)*{\sline};
(-8,-2.5)*{\fcolorbox{black}{white}{$v$}};
(0,-2.5)*{\fcolorbox{black}{white}{$v$}};
(8,-2.5)*{\fcolorbox{black}{white}{$v$}};  
}}
\endxy \ +a \ \xy 0;/r.17pc/:
(0,0)*{\xybox{
(4,2)*{\cross};
(-2,2)*{\sline};
(10,2)*{\sline};
(0,6)*{\cross};
(6,6)*{\sline};
(10,6)*{\sline};
(2,10)*{\sline};
(6,10)*{\sline};
(10,10)*{\sline};
(-10,2)*{\sline};
(-10,6)*{\sline};
(-10,10)*{\sline};
(-6,2)*{\sline};
(-6,6)*{\sline};
(-6,10)*{\sline};
(-2,10)*{\sline};
(-10,18)*{\sline};
(-10,22)*{\sline};
(-10,26)*{\sline};
(-6,18)*{\sline};
(-2,18)*{\sline};
(-6,22)*{\sline};
(-6,26)*{\sline};
(2,18)*{\sline};
(6,18)*{\sline};
(10,18)*{\sline};
(0,22)*{\cross};
(8,22)*{\cross};
(4,26)*{\cross};
(10,26)*{\sline};
(-2,26)*{\sline};
(-10,14)*{\sline};
(-6,14)*{\sline};
(-2,14)*{\sdot};
(2,14)*{\sline};
(6,14)*{\sline};
(10,14)*{\sline};
(-8,-2.5)*{\fcolorbox{black}{white}{$v$}};
(0,-2.5)*{\fcolorbox{black}{white}{$v$}};
(8,-2.5)*{\fcolorbox{black}{white}{$v$}};  
}}
\endxy \ + b \ \xy 0;/r.17pc/:
(0,0)*{\xybox{
(4,2)*{\cross};
(-2,2)*{\sline};
(10,2)*{\sline};
(0,6)*{\cross};
(6,6)*{\sline};
(10,6)*{\sline};
(2,10)*{\sline};
(6,10)*{\sline};
(10,10)*{\sline};
(-10,2)*{\sline};
(-10,6)*{\sline};
(-10,10)*{\sline};
(-6,2)*{\sline};
(-6,6)*{\sline};
(-6,10)*{\sline};
(-2,10)*{\sline};
(-10,18)*{\sline};
(-10,22)*{\sline};
(-10,26)*{\sline};
(-6,18)*{\sline};
(-2,18)*{\sline};
(-6,22)*{\sline};
(-6,26)*{\sline};
(2,18)*{\sline};
(6,18)*{\sline};
(10,18)*{\sline};
(0,22)*{\cross};
(8,22)*{\cross};
(4,26)*{\cross};
(10,26)*{\sline};
(-2,26)*{\sline};
(-10,14)*{\sline};
(-6,14)*{\sdot};
(-2,14)*{\sline};
(2,14)*{\sline};
(6,14)*{\sline};
(10,14)*{\sline};
(-8,-2.5)*{\fcolorbox{black}{white}{$v$}};
(0,-2.5)*{\fcolorbox{black}{white}{$v$}};
(8,-2.5)*{\fcolorbox{black}{white}{$v$}};  
}}
\endxy
\ \right) \right) \\
&+ b \left(
0 + \left( \ \xy 0;/r.17pc/:
(0,0)*{\xybox{
(4,2)*{\cross};
(-2,2)*{\sline};
(10,2)*{\sline};
(-2,6)*{\sline};
(2,6)*{\sline};
(8,6)*{\cross};
(2,10)*{\sline};
(6,10)*{\sline};
(10,10)*{\sline};
(-10,2)*{\sline};
(-10,6)*{\sline};
(-2,10)*{\sline};
(-6,2)*{\sline};
(-6,6)*{\sline};
(-8,10)*{\cross};
(-2,18)*{\sline};
(-10,22)*{\sline};
(-10,26)*{\sline};
(-8,18)*{\cross};
(-6,22)*{\sline};
(-6,26)*{\sline};
(2,18)*{\sline};
(6,18)*{\sline};
(10,18)*{\sline};
(0,22)*{\cross};
(8,22)*{\cross};
(4,26)*{\cross};
(10,26)*{\sline};
(-2,26)*{\sline};
(-4,14)*{\cross};
(-10,14)*{\sline};
(2,14)*{\sline};
(6,14)*{\sline};
(10,14)*{\sline};
(-8,-2.5)*{\fcolorbox{black}{white}{$v$}};
(0,-2.5)*{\fcolorbox{black}{white}{$v$}};
(8,-2.5)*{\fcolorbox{black}{white}{$v$}};  
}}
\endxy \ +a \ \xy 0;/r.17pc/:
(0,0)*{\xybox{
(4,2)*{\cross};
(-2,2)*{\sline};
(10,2)*{\sline};
(-2,6)*{\sline};
(2,6)*{\sline};
(8,6)*{\cross};
(2,10)*{\sline};
(6,10)*{\sline};
(10,10)*{\sline};
(-10,2)*{\sline};
(-10,6)*{\sline};
(-10,10)*{\sline};
(-6,2)*{\sline};
(-6,6)*{\sline};
(-6,10)*{\sline};
(-2,10)*{\sline};
(-10,18)*{\sline};
(-10,22)*{\sline};
(-10,26)*{\sline};
(-6,18)*{\sline};
(-2,18)*{\sline};
(-6,22)*{\sline};
(-6,26)*{\sline};
(2,18)*{\sline};
(6,18)*{\sline};
(10,18)*{\sline};
(0,22)*{\cross};
(8,22)*{\cross};
(4,26)*{\cross};
(10,26)*{\sline};
(-2,26)*{\sline};
(-10,14)*{\sdot};
(-6,14)*{\sline};
(-2,14)*{\sline};
(2,14)*{\sline};
(6,14)*{\sline};
(10,14)*{\sline};
(-8,-2.5)*{\fcolorbox{black}{white}{$v$}};
(0,-2.5)*{\fcolorbox{black}{white}{$v$}};
(8,-2.5)*{\fcolorbox{black}{white}{$v$}};  
}}
\endxy \ + a \ \xy 0;/r.17pc/:
(0,0)*{\xybox{
(4,2)*{\cross};
(-2,2)*{\sline};
(10,2)*{\sline};
(-2,6)*{\sline};
(2,6)*{\sline};
(8,6)*{\cross};
(2,10)*{\sline};
(6,10)*{\sline};
(10,10)*{\sline};
(-10,2)*{\sline};
(-10,6)*{\sline};
(-10,10)*{\sline};
(-6,2)*{\sline};
(-6,6)*{\sline};
(-6,10)*{\sline};
(-2,10)*{\sline};
(-10,18)*{\sline};
(-10,22)*{\sline};
(-10,26)*{\sline};
(-6,18)*{\sline};
(-2,18)*{\sline};
(-6,22)*{\sline};
(-6,26)*{\sline};
(2,18)*{\sline};
(6,18)*{\sline};
(10,18)*{\sline};
(0,22)*{\cross};
(8,22)*{\cross};
(4,26)*{\cross};
(10,26)*{\sline};
(-2,26)*{\sline};
(-10,14)*{\sline};
(-6,14)*{\sline};
(-2,14)*{\sdot};
(2,14)*{\sline};
(6,14)*{\sline};
(10,14)*{\sline};
(-8,-2.5)*{\fcolorbox{black}{white}{$v$}};
(0,-2.5)*{\fcolorbox{black}{white}{$v$}};
(8,-2.5)*{\fcolorbox{black}{white}{$v$}};  
}}
\endxy \ + b \ \xy 0;/r.17pc/:
(0,0)*{\xybox{
(4,2)*{\cross};
(-2,2)*{\sline};
(10,2)*{\sline};
(-2,6)*{\sline};
(2,6)*{\sline};
(8,6)*{\cross};
(2,10)*{\sline};
(6,10)*{\sline};
(10,10)*{\sline};
(-10,2)*{\sline};
(-10,6)*{\sline};
(-10,10)*{\sline};
(-6,2)*{\sline};
(-6,6)*{\sline};
(-6,10)*{\sline};
(-2,10)*{\sline};
(-10,18)*{\sline};
(-10,22)*{\sline};
(-10,26)*{\sline};
(-6,18)*{\sline};
(-2,18)*{\sline};
(-6,22)*{\sline};
(-6,26)*{\sline};
(2,18)*{\sline};
(6,18)*{\sline};
(10,18)*{\sline};
(0,22)*{\cross};
(8,22)*{\cross};
(4,26)*{\cross};
(10,26)*{\sline};
(-2,26)*{\sline};
(-10,14)*{\sline};
(-6,14)*{\sdot};
(-2,14)*{\sline};
(2,14)*{\sline};
(6,14)*{\sline};
(10,14)*{\sline};
(-8,-2.5)*{\fcolorbox{black}{white}{$v$}};
(0,-2.5)*{\fcolorbox{black}{white}{$v$}};
(8,-2.5)*{\fcolorbox{black}{white}{$v$}};  
}}
\endxy
\ \right)  
\right). 
\end{align*} 
Note that the last four diagrams are zero since $\tau_5^2 e(*,1,1) = 0$ in $R(3\delta)$. 
Hence, 
\begin{align*}
X =& -ac \ \xy 0;/r.17pc/:
(0,0)*{\xybox{
(4,2)*{\cross};
(-2,2)*{\sline};
(10,2)*{\sline};
(0,6)*{\cross};
(6,6)*{\sline};
(10,6)*{\sline};
(-10,2)*{\sline};
(-10,6)*{\sline};
(-6,2)*{\sline};
(-6,6)*{\sline};
(-10,14)*{\sline};
(-10,18)*{\sline};
(-6,14)*{\sline};
(-6,18)*{\sline};
(0,14)*{\cross};
(8,14)*{\cross};
(4,18)*{\cross};
(10,18)*{\sline};
(-2,18)*{\sline};
(-10,10)*{\sline};
(-6,10)*{\sline};
(-2,10)*{\sdot};
(2,10)*{\sline};
(6,10)*{\sline};
(10,10)*{\sline};
(-8,-2.5)*{\fcolorbox{black}{white}{$v$}};
(0,-2.5)*{\fcolorbox{black}{white}{$v$}};
(8,-2.5)*{\fcolorbox{black}{white}{$v$}};  
}}
\endxy \ = -ac \left( \ \xy 0;/r.17pc/:
(0,0)*{\xybox{
(4,2)*{\cross};
(-2,2)*{\sline};
(10,2)*{\sline};
(0,6)*{\cross};
(6,6)*{\sline};
(10,6)*{\sline};
(-10,2)*{\sline};
(-10,6)*{\sline};
(-6,2)*{\sline};
(-6,6)*{\sline};
(-10,14)*{\sline};
(-10,18)*{\sline};
(-6,14)*{\sline};
(-6,18)*{\sline};
(0,14)*{\cross};
(8,14)*{\cross};
(4,18)*{\cross};
(10,18)*{\sline};
(-2,18)*{\sline};
(-10,10)*{\sline};
(-6,10)*{\sline};
(-2,10)*{\sline};
(2,10)*{\sline};
(6,10)*{\sline};
(10,10)*{\sline};
(5,1)*{\scriptstyle \bullet};
(-8,-2.5)*{\fcolorbox{black}{white}{$v$}};
(0,-2.5)*{\fcolorbox{black}{white}{$v$}};
(8,-2.5)*{\fcolorbox{black}{white}{$v$}};  
}}
\endxy \ - \ \xy 0;/r.17pc/:
(0,0)*{\xybox{
(4,2)*{\cross};
(-2,2)*{\sline};
(10,2)*{\sline};
(-2,6)*{\sline};
(2,6)*{\sline};
(6,6)*{\sline};
(10,6)*{\sline};
(-10,2)*{\sline};
(-10,6)*{\sline};
(-6,2)*{\sline};
(-6,6)*{\sline};
(-10,14)*{\sline};
(-10,18)*{\sline};
(-6,14)*{\sline};
(-6,18)*{\sline};
(0,14)*{\cross};
(8,14)*{\cross};
(4,18)*{\cross};
(10,18)*{\sline};
(-2,18)*{\sline};
(-10,10)*{\sline};
(-6,10)*{\sline};
(-2,10)*{\sline};
(2,10)*{\sline};
(6,10)*{\sline};
(10,10)*{\sline};
(-8,-2.5)*{\fcolorbox{black}{white}{$v$}};
(0,-2.5)*{\fcolorbox{black}{white}{$v$}};
(8,-2.5)*{\fcolorbox{black}{white}{$v$}};  
}}
\endxy
\ \right) = ac\ \xy 0;/r.17pc/:
(0,0)*{\xybox{(0,6)*{\stau};
(-14,2)*{\sline};
(-14,6)*{\sline};
(-14,10)*{\sline};
(-10,2)*{\sline};
(-10,6)*{\sline};
(-10,10)*{\sline};
(-4,-2.5)*{\fcolorbox{black}{white}{$v$}};
(4,-2.5)*{\fcolorbox{black}{white}{$v$}};  
(-12,-2.5)*{\fcolorbox{black}{white}{$v$}};
}}
\endxy \ . 
\end{align*}

Let $\overline{0} = 1, \overline{1} = 0$. 
Consider an isomorphism $R_{a,b,c}(3\delta) \to R_{c,b,a}(3\delta)$ given by 
\[
e(\nu_1,\ldots,\nu_6) \mapsto e(\overline{\nu_6},\ldots,\overline{\nu_1}), \ x_k \mapsto x_{7-k}, \ \tau_k \mapsto -\tau_{6-k}. 
\]
Be careful of the sign. 
Under this isomorphism, $W$ for $R_{a,b,c}(3\delta)$ coincides with $-X$ for $R_{c,b,a}(3\delta)$. 
(The sign is determined by the fact that the number of crossings in every relevant diagram is even.)
Hence,
\[
W  = - ca\ \xy 0;/r.17pc/:
(0,0)*{\xybox{(0,6)*{\stau};
(14,2)*{\sline};
(14,6)*{\sline};
(14,10)*{\sline};
(10,2)*{\sline};
(10,6)*{\sline};
(10,10)*{\sline};
(-4,-2.5)*{\fcolorbox{black}{white}{$v$}};
(4,-2.5)*{\fcolorbox{black}{white}{$v$}};  
(12,-2.5)*{\fcolorbox{black}{white}{$v$}};
}}
\endxy \ . 
\]

Since our choice of polynomials $Q$ is given by $(a,b,c) = (-t,1+t,-1)$,
we deduce 
\[
X + W = t (\tau^{2,3}-\tau^{1,2})(v\boxtimes v \boxtimes v). 
\]
Therefore, the assertion follows. 
\end{proof}

\subsection{The $A_N^{(1)}$ case} \label{sub:ANcase}

In this subsection, assume that $\mathbf{k}$ is a field. 
Let $X_N^{(1)} = A_N^{(1)}$ and fix $i \in \mathring{I}$, namely, $1 \leq i \leq N$. 
We construct an endomorphism $\mathsf{T}_i$ of $L_i^e(\delta)^{\circ 2}$ satisfying (a) and (b). 
To this end, we will introduce, following \cite{MR3670026}, a monoidal functor 
\[
\mathcal{F}_i \colon \gMod{\underline{R}} \to \gMod{R}, 
\]
where $\underline{R}$ is the quiver Hecke algebra of type $A_1^{(1)}$ associated with the same parameter $t$ as $R$. 
To avoid confusion between these two different root systems, 
the following notations for type $A_1^{(1)}$ are used only in this subsection. 
We index the Dynkin nodes with underlined numbers $\underline{0}$ and $\underline{1}$, 
and set $\underline{I} = \{\underline{0}, \underline{1} \}$. 
The simple roots are denoted by $\alpha_{\underline{0}}$ and $\alpha_{\underline{1}}$.  
The positive root lattice is denoted by $\underline{\mathsf{Q}}_+$, and the null root is denoted by $\underline{\delta}$.  
Furthermore, the simple module defined in Definition \ref{def:minuscule} (for $\underline{1}$) is denoted by $\underline{L}_{\underline{1}}^e(\underline{\delta})$, 
and the module in Definition \ref{def:minusculeproj} is denoted by $\underline{\Delta}_{\underline{1}}^e(\underline{\delta})$. 
In this subsection, we do not use the underlined notation for multipartitions or elements of $\mathbb{Z}_{\geq 0}^{\mathring{I}}$. 

We define 
\[
\beta_{\underline{0}} = \delta-\alpha_i, \ \beta_{\underline{1}} = \alpha_i \in \mathsf{Q}_+. 
\]
Recall the modules $L^e(\delta-\alpha_i)$ and $\Delta^e(\delta-\alpha_i)$ from Subsection \ref{sub:minuscule}. 
We write 
\[
L(\underline{0}) = L^e(\delta-\alpha_i), \ \Delta(\underline{0}) = \Delta^e(\delta-\alpha_i), L(\underline{1}) = L(i), \ \Delta(\underline{1}) = R(\alpha_i) \in \gMod{R}.  
\]
For $\underline{\alpha} \in \underline{\mathsf{Q}}_+$, 
we will define a right action of $\underline{R}(\underline{\alpha})$ on 
\[
\Delta(\underline{\alpha}) = \bigoplus_{\underline{\nu} \in \underline{I}^{\underline{\alpha}}} \Delta(\underline{\nu}), \ \Delta(\underline{\nu}) = \Delta(\underline{\nu}_1) \circ \cdots \circ \Delta(\underline{\nu}_n) \in \gMod{R}, 
\]
where $n$ is the height of $\underline{\alpha}$.
If $\underline{\alpha} = a \alpha_{\underline{0}} + b \alpha_{\underline{1}}$, 
then $\Delta(\underline{\alpha})$ is an $(R(a\beta_{\underline{0}}+b\beta_{\underline{1}}), R(\underline{\alpha}))$-bimodule. 
The functor $\mathcal{F}_i$ will be defined as 
\[
\mathcal{F}_i(M) = \Delta(\underline{\alpha}) \otimes_{\underline{R}(\underline{\alpha})} M \ (M \in \gMod{\underline{R}(\underline{\alpha})}). 
\]

We begin with describing the explicit structures of $L(\underline{0})$ and $\Delta(\underline{0})$. 
Let 
\[
y = s_0(s_1 s_2\cdots s_{i-1})(s_N s_{N-1}\cdots s_{i+1}) \in W, 
\]
and let $C$ be the subset of $I^{\delta-\alpha_i}$ consisting of all reduced words of $y$. 
Note that any element of $C$ is obtained from $(0,1,2,\ldots,i-1,N,N-1,\ldots,i+1)$ by repeatedly interchanging consecutive entries $j,k$ satisfying $j-k \not \in \{1,-1\} + (N+1)\mathbb{Z}$.  

\begin{lemma}\label{lem:homogeneous}
The module $L(\underline{0})$ has a homogeneous basis $\{v_{\nu} \mid \nu \in C\}$ with $\deg v_{\nu} = 0$, 
where the action of $R(\delta-\alpha_i)$ is given by 
\begin{align*}
e(\nu')v_{\nu} &= \delta_{\nu',\nu} v_{\nu} \ (\nu' \in I^{\delta-\alpha_i}), \\
x_kv_{\nu} &= 0 \ (1 \leq k \leq N), \\
\tau_kv_{\nu} &= \begin{cases}
v_{s_k\nu} & \text{if $\nu_k - \nu_{k+1} \not \in \{1,-1\} + (N+1)\mathbb{Z}$}, \\
0 & \text{otherwise}, 
\end{cases} \ (1 \leq k \leq N-1). 
\end{align*}
\end{lemma}

\begin{proof}
It is essentially the same as \cite[Corollary 5.18]{MR3175157} since $R(\delta-\alpha_i)$ is of finite type $A_N$ and it does not depend on the choice of polynomials $Q$ up to isomorphism by Lemma \ref{lem:modifyQ}.  
For completeness, we provide a proof.
It is easy to see that the formulas given in the lemma define a self-dual simple module $L(C)$.
We use Theorem \ref{thm:crystal} to prove that $L(C) \simeq L(\underline{0})$.  
Since $L_i^e(\delta)$ is defined as  the head of $L(\underline{0}) \circ L(\underline{1})$, 
we have a nonzero homomorphism $L(\underline{0}) \otimes L(\underline{1}) \to \Res_{\delta-\alpha_i,\alpha_i}L_i^e(\delta)$ by induction-restriction adjunction. 
Since $L(\underline{0})  \otimes L(\underline{1})$ is simple, this homomorphism is injective.
Hence, we have $\varepsilon_j(L(\underline{0})) \leq \varepsilon_j(L_i^e(\delta))$ for any $j \in I$.
Since $L_i^e(\delta)$ is $e$-cuspidal, we deduce that $\varepsilon_j(L(\underline{0})) = 0$ for any $j \in \mathring{I}$.  
On the other hand, we have $\varepsilon_j(L(C)) = 0$ for any $j \in \mathring{I}$ by definition. 
Since $\delta-\alpha_i = \sum_{j \in I, j\neq i}\alpha_j$, we have 
\[
\varepsilon_0(L(\underline{0})) = \varepsilon_0(L(C)) = 1. 
\]
Note that there are only one element $b \in B(\infty)$ such that 
\[
\wt (b) = -(\delta-\alpha_i), \ \varepsilon_j(b) = \delta_{j,0} \ (j \in I).  
\]
Indeed, the set of elements $b$ satisfying these conditions are identified with the weight $\varpi_0-(\delta-\alpha_i)$ set of the highest weight crystal $B(\varpi_0)$, 
which consists of exactly one element since $\varpi_0 - (\delta-\alpha_i) = y\varpi_0$. 
Hence, $L(\underline{0}) \simeq L(C)$. 
\end{proof}

\begin{lemma} \label{lem:Deltazero}
$\Delta(\underline{0})$ has a homogeneous basis $\{z^dv_{\nu} \mid \nu \in C, d \in \mathbb{Z}_{\geq 0}\}$ with $\deg z^d v_{\nu} = 2d$, 
where the action of $R(\delta-\alpha_i)$ is given by 
\begin{align*}
e(\nu') z^dv_{\nu} &= \delta_{\nu',\nu}z^dv_{\nu} \ (\nu' \in I^{\delta-\alpha_i}), \\
x_k z^dv_{\nu} &= \begin{cases}
z^{d+1}v_{\nu} & \text{if $0 \leq \nu_k \leq i-1$}, \\
tz^{d+1}v_{\nu} & \text{if $i+1 \leq \nu_k \leq N$}, 
\end{cases} \ (1 \leq k \leq N)\\
\tau_kz^dv_{\nu} &= \begin{cases}
z^dv_{s_k\nu} & \text{if $\nu_k - \nu_{k+1} \not \in \{1,-1\} + (N+1)\mathbb{Z}$}, \\
0 & \text{otherwise}, 
\end{cases} \ (1\leq k \leq N -1). 
\end{align*}
\end{lemma}

\begin{proof}
It is straightforward to see that the formulas given in the lemma define an $R(\delta-\alpha_i)$-module $M$. 
For instance, one of the necessary relations is 
\[
\tau_k^2z^dv_{\nu} = Q_{\nu_k,\nu_{k+1}}(x_k,x_{k+1})z^dv_{\nu}, 
\]
which can be verified as follows: 
\begin{align*}
&\tau_k^2z^dv_{\nu} = \begin{cases}
z^dv_{\nu} & \text{if $\nu_k - \nu_{k+1} \not \in \{1,-1\} + (N+1)\mathbb{Z}$}, \\
0 & \text{otherwise},
\end{cases} \\
&Q_{\nu_k,\nu_{k+1}}(x_k,x_{k+1})z^dv_{\nu} \\
&= \begin{cases}
z^dv_{\nu} & \text{if $\nu_k - \nu_{k+1} \not \in \{1,-1\} + (N+1)\mathbb{Z}$}, \\
Q_{\nu_k,\nu_{k+1}}(z,z)z^dv_{\nu} = (z-z)z^dv_{\nu}= 0 & \text{if $\nu_k - \nu_{k+1} \in \{1,-1\}, 0\leq \nu_k,\nu_{k+1} \leq i-1$}, \\
Q_{N,0} (tz,z) z^dv_{\nu} = (tz-tz) z^dv_{\nu} = 0 & \text{if $\{\nu_k,\nu_{k+1}\} = \{N,0\}$}, \\
Q_{\nu_k,\nu_{k+1}}(tz,tz)z^dv_{\nu} = (tz-tz)z^dv_{\nu} = 0 & \text{if $\nu_k - \nu_{k+1} \in \{1,-1\}, i+1 \leq \nu_k,\nu_{k+1} \leq N$}. 
\end{cases}
\end{align*} 
We have a surjective homomorphism $M \to L(\underline{0})$ given by 
\[
z^dv_{\nu} \mapsto \delta_{d,0} v_{\nu} \ (\nu \in C, d \in \mathbb{Z}_{\geq 0}). 
\]
This homomorphism restricts to an isomorphism of vector spaces $M_0 \simeq L(\underline{0})$.  
Since $L(\underline{0})$ is simple, we see that any nonzero element of $M_0$ generate $M$ as an $R(\delta-\alpha_i)$-module. 
Hence, the head of $M$ as a graded module is $L(\underline{0})$. 
By definition, we have
\begin{equation} \label{eq:character}
\chi(M) = \frac{1}{1-q^2} \chi(L(\underline{0})). 
\end{equation}
In particular, all the composition factors of $M$ are grading shifts of $L(\underline{0})$. 
Hence, we have a surjective homomorphism $\Delta(\underline{0}) \to M$. 
Since $\chi(\Delta(\underline{0})) = \chi(M)$ by Corollary \ref{cor:maximalext} and (\ref{eq:character}), 
this must be an isomorphism.   
\end{proof}

Note that there is an endomorphism $z$ of $\Delta(\underline{0})$ given by $z^dv_{\nu}\mapsto z^{d+1}v_{\nu}$.  
We also have an endomorphism $w$ of $\Delta(\underline{1}) = R(\alpha_i)$ given by the multiplication by $x_1$. 

\begin{lemma}
We have 
\[
\End_R (\Delta(\underline{0})) = \mathbf{k}[z], \ \End_R(\Delta(\underline{1})) = \mathbf{k}[w], 
\]
both of which are isomorphic to the polynomial algebra in one variable. 
\end{lemma}

\begin{proof}
It follows from Theorem \ref{thm:stratification}.
\end{proof}

\begin{lemma}[{\cite[Lemma 3.10]{MR3670026}}] \label{lem:tau}
For $\underline{j}, \underline{k} \in \underline{I}$, there exists a unique homomorphism
\[
\tau_{\underline{j},\underline{k}} \colon q^{-(\alpha_{\underline{j}},\alpha_{\underline{k}})} \Delta(\underline{j}) \circ \Delta(\underline{k}) \to \Delta(\underline{k}) \circ \Delta(\underline{j}), 
\]
such that 
\[
\tau_{\underline{j},\underline{k}}(u \boxtimes v) = \tau_{w[\height \beta_{\underline{k}}, \height \beta_{\underline{j}}]}(v \boxtimes u) \ (u \in \Delta(\underline{j}), v \in \Delta(\underline{k})).
\]
\end{lemma}

\begin{remark} \label{rem:unmixingr}
If $\underline{j} = \underline{k} = \underline{1}$, the module $\Delta(\underline{1}) \circ \Delta(\underline{1})$ is canonically isomorphic to $R(2\alpha_i)$ and $\tau_{\underline{1},\underline{1}}$ is the right multiplication by $\tau_1$. 
If $\underline{j} \neq \underline{k}$, the existence of the homomorphism $\tau_{\underline{j},\underline{k}}$ is proved as follows. 
Since $\delta-\alpha_i = \sum_{j \in I, j\neq i} \alpha_j$ and $\alpha_i$ does not appear in this expansion, 
the Mackey filtration of $\Res_{\alpha_{\underline{j}},\alpha_{\underline{k}}} (\Delta(\underline{k})\circ \Delta(\underline{j}))$ is one-step, namely, we have 
\begin{align*}
\Res_{\alpha_i,\delta-\alpha_i} (\Delta(\underline{0})\circ \Delta(\underline{1})) &\simeq q^2 \Delta(\underline{1}) \otimes \Delta(\underline{0}), \\
\Res_{\delta-\alpha_i,\alpha_i} (\Delta(\underline{1})\circ \Delta(\underline{0})) &\simeq q^2 \Delta(\underline{0}) \otimes \Delta(\underline{1}). 
\end{align*}
By induction-restriction adjunction, it shows the existence of $\tau_{\underline{j},\underline{k}}$. 
Note that $\tau_{\underline{1},\underline{0}}$ coincides with the injective homomorphism $\mathsf{R}$ in Proposition \ref{prop:imaginarystd}, and we have 
\[
\Cok (\tau_{\underline{1},\underline{0}}) \simeq \Delta_i^e(\delta). 
\]
\end{remark}

By \cite[Section 3.4]{MR3670026}, there exist homogeneous quadratic polynomials $Q_{\underline{0},\underline{1}}(u,v) = Q_{\underline{1},\underline{0}}(v,u)$ satisfying
\[
\tau_{\underline{1},\underline{0}}\tau_{\underline{0},\underline{1}} = Q_{\underline{0},\underline{1}}(z,w), \ \tau_{\underline{0},\underline{1}}\tau_{\underline{1},\underline{0}} = Q_{\underline{1},\underline{0}}(w,z), 
\]
as endomorphisms of $\Delta(\underline{0}) \circ \Delta(\underline{1})$ and of $\Delta(\underline{1}) \circ \Delta(\underline{0})$ respectively. 
The following is the key new result in this subsection: 

\begin{proposition} \label{prop:computeQ}
We have $Q_{\underline{0},\underline{1}}(u,v) = (u-v)(v-tu)$. 
\end{proposition}

\begin{proof}
Let $\nu = (0,1,\ldots,i-1,N,N-1,\ldots,i+1) \in C$. 
We have 
\begin{align*}
Q_{\underline{0},\underline{1}}(z,w)(v_{\nu}\boxtimes e(i)) &=\tau_{\underline{1},\underline{0}}\tau_{\underline{0},\underline{1}}(v_{\nu} \boxtimes e(i)) \\
&= \tau_N \tau_{N-1}\cdots \tau_2 \tau_1 \tau_1 \tau_2 \cdots \tau_{N-1}\tau_N (v_{\nu} \boxtimes e(i)).  \\
\end{align*}
When $i \in \{1, \ldots, N-1\}$, this is equal to 
\begin{align*}
&Q_{i-1,i}(x_i,x_{N+1})Q_{i+1,i}(x_N,x_{N+1})(v_{\nu} \boxtimes e(i)) \\
&= (x_i-x_{N+1})(-x_N +  x_{N+1})(v_{\nu} \boxtimes e(i)) \\ 
&= (z-w)(-tz+w)(v_{\nu}\boxtimes e(i)) \quad \text{by Lemma \ref{lem:Deltazero}}. 
\end{align*}
On the other hand, when $i = N$ and hence $\nu = (0,1,\ldots, N-1)$, it is 
\begin{align*}
&Q_{0,N}(x_1,x_{N+1})Q_{N-1,N}(x_N,x_{N+1})(v_{\nu}\boxtimes e(N)) \\
&= (-tx_1+x_{N+1})(x_N-x_{N+1})(v_{\nu}\boxtimes e(N)) \\
&= (-tz+w)(z-w)(v_{\nu}\boxtimes e(N)) \quad \text{by Lemma \ref{lem:Deltazero}}. 
\end{align*}
The proposition is proved.
\end{proof}

By \cite[Theorem 3.12]{MR3670026}, there exists $c \in \mathbf{k}^{\times}$ such that 
\[
\tau_{\underline{0},\underline{0}}(\id \otimes z) - (z \otimes \id)\tau_{\underline{0},\underline{0}} = (\id \otimes z)\tau_{\underline{0},\underline{0}} - \tau_{\underline{0},\underline{0}}(z \otimes \id) = c \id, 
\]
as endomorphisms of $\Delta(\underline{0})^{\circ 2}$. 
Recall that $\underline{R}$ is the quiver Hecke algebra of type $A_1^{(1)}$ associated with the same parameter $t$ as $R$. 
For $\underline{\alpha} \in \underline{\mathsf{Q}}_+$ of height $n$, 
the module $\Delta(\underline{\alpha})\in \gMod{R}$ is a right $\underline{R}(\underline{\alpha})$-module as follows: 
\begin{itemize}
\item for $\underline{\nu} \in \underline{I}^{\underline{\alpha}}$, $e(\underline{\nu})$ acts as the projection to $\Delta(\underline{\nu}_1) \circ \cdots \circ \Delta(\underline{\nu}_n)$; 
\item if $\underline{\nu}_k = \underline{0}$, $e(\underline{\nu})x_k$ acts as $\id_{\Delta(\underline{\nu}_1)} \otimes \cdots \otimes z \otimes \cdots \otimes \id_{\Delta(\underline{\nu}_n)}$ on $\Delta(\underline{\nu}_k)$, where $z$ appears in the $k$-th position;
\item similarly, if $\underline{\nu}_k = \underline{1}$, $e(\underline{\nu})x_k$ acts as an endomorphism induced from $w$ on $\Delta(\underline{\nu}_k)$; 
\item $e(\underline{\nu})\tau_k$ acts as $c^{-\delta_{\underline{\nu}_k, \underline{0}}\delta_{\underline{\nu}_{k+1}, \underline{0}}}\tau_{\underline{\nu}_k,\underline{\nu}_{k+1}}$ from $\Delta(\underline{\nu})$ to $\Delta(s_k\underline{\nu})$. 
\end{itemize}
We define $\mathcal{F}_i \colon \gMod{\underline{R}} \to \gMod{R}$ by
\[
\mathcal{F}_i(M) = \Delta(\underline{\alpha})\otimes_{\underline{R}(\underline{\alpha})} M \ (M \in \gMod{\underline{R}(\underline{\alpha})}). 
\]
By definition, $\mathcal{F}_i$ is right exact and monoidal. 

\begin{remark}
We adopt a slightly different normalization for the right action of $\underline{R}(\underline{\alpha})$ on $\Delta(\underline{\alpha})$ from the one in \cite{MR3670026}.
In fact, these two actions only differ by a certain automorphism of $\underline{R}(\underline{\alpha})$. 
\end{remark}

\begin{remark}
The functor $\mathcal{F}_i$ is the same as the one obtained by the procedure \cite[Section 4.2]{MR3790066}. 
In fact, the homomorphism $\tau_{\underline{j}, \underline{k}}$ is characterized up to a scalar multiple as the homogeneous homomorphism of degree $-(\alpha_{\underline{j}}, \alpha_{\underline{k}})$, 
and hence coincides with the homomorphisms $\mathsf{R}_{j,k}$ (if $j \neq k$) or $\mathsf{r}_j$ (if $j =k$) used in \cite{MR3790066}. 
\end{remark}

\begin{lemma} \label{lem:GDelta}
We have 
\[
\mathcal{F}_i(\underline{\Delta}_{\underline{1}}^e(\underline{\delta})) \simeq \Delta_i^e(\delta). 
\]
\end{lemma}

\begin{proof}
By Proposition \ref{prop:imaginarystd}, we have a short exact sequence 
\[
0 \to q^2 \underline{R}(\alpha_{\underline{1}}) \circ \underline{R}(\alpha_{\underline{0}}) \to \underline{R}(\alpha_{\underline{0}}) \circ \underline{R}(\alpha_{\underline{1}}) \to \underline{\Delta}_{\underline{1}}^e(\underline{\delta}) \to 0. 
\]
Note that, we have canonical isomorphisms
\[
\underline{R}(\alpha_{\underline{0}}) \circ \underline{R}(\alpha_{\underline{1}}) \simeq \underline{R}(\underline{\delta})e(\underline{0},\underline{1}), \ \underline{R}(\alpha_{\underline{1}}) \circ \underline{R}(\alpha_{\underline{0}}) \simeq \underline{R}(\underline{\delta})e(\underline{1},\underline{0}),
\]
and the injective homomorphism in this short exact sequence is the right multiplication by $\tau_1$. 
Hence, applying $\mathcal{F}_i$ we obtain an exact sequence 
\[
q^2 \Delta(\underline{1}) \circ \Delta(\underline{0}) \to \Delta(\underline{0}) \circ \Delta(\underline{1}) \to \mathcal{F}_i(\underline{\Delta}_{\underline{1}}^e(\underline{\delta})) \to 0, 
\]
where the first homomorphism is $\tau_{\underline{1},\underline{0}}$. 
By Remark \ref{rem:unmixingr}, we deduce the lemma. 
\end{proof}

In Subsection \ref{sub:A1case}, we have constructed an endomorphism $\mathsf{T}_{\underline{1}}$ of $\underline{\Delta}_{\underline{1}}^e(\underline{\delta})^{\circ 2}$ that satisfies (a) and (b).  
Identifying $\Delta_i^e(\delta)$ with $\mathcal{F}_i(\underline{\Delta}_{\underline{1}}^e(\underline{\delta}))$ using Lemma \ref{lem:GDelta}, 
we define the following endomorphism of $\Delta_i^e(\delta)^{\circ 2}$:
\[
\mathsf{T}_i = \mathcal{F}_i(\mathsf{T}_{\underline{1}}). 
\]
Since the functor $\mathcal{F}_i$ is monoidal, the endomorphism $\mathsf{T}_i$ also satisfies (a) and (b). 
Hence, it remains to prove that $\mathsf{T}_i$ is nonconstant, 
which can be seen from the lemma below: 

\begin{lemma} \label{lem:Tiform}
For $u, u' \in \Delta_i^e(\delta)$, we have 
\[
\mathsf{T}_i (u \boxtimes u') = \tau_{w[N+1,N+1]}(u' \boxtimes u) +t  u\boxtimes u'. 
\]
\end{lemma}

Note that, when $N =1$, this action is the same as the one given in Subsection \ref{sub:A1case}. 

\begin{proof}
Recall that $\mathsf{T}_{\underline{1}} =  \tau +t$, where 
\[
\tau (u \boxtimes u') = \tau_{w[2,2]}(u' \boxtimes u) \ (u,u' \in \underline{\Delta}_{\underline{1}}^e(\underline{\delta})). 
\] 
Hence, we have a commutative diagram 
\[
\begin{tikzcd}
\underline{R}(2\underline{\delta})e(\underline{0},\underline{1},\underline{0},\underline{1}) \arrow[r,twoheadrightarrow]\arrow[d,"\times (\tau_{w[2,2]}+t)"'] & \underline{\Delta}_{\underline{1}}^e(\underline{\delta})^{\circ 2} \arrow[d,"\mathsf{T}_{\underline{1}}"] \\
\underline{R}(2\underline{\delta})e(\underline{0},\underline{1},\underline{0},\underline{1}) \arrow[r,twoheadrightarrow] & \underline{\Delta}_{\underline{1}}^e(\underline{\delta})^{\circ 2} \\
\end{tikzcd}
\]
where both horizontal homomorphisms are induced from the surjection $\underline{R}(\alpha_{\underline{0}}) \circ \underline{R}(\alpha_{\underline{1}}) \to \underline{\Delta}_{\underline{1}}^e(\underline{\delta})$. 
Applying the right exact functor $\mathcal{F}_i$, we obtain a commutative diagram
\[
\begin{tikzcd}
\Delta(\underline{0})\circ \Delta(\underline{1}) \circ \Delta(\underline{0}) \circ \Delta(\underline{1}) \arrow[r,twoheadrightarrow]\arrow[d,"\tau_{\underline{0},\underline{0}}^{2,3}\tau_{\underline{0},\underline{0}}^{1,2}\tau_{\underline{1},\underline{1}}^{3,4}\tau_{\underline{1},\underline{0}}^{2,3} +t"'] & \Delta_i^e(\delta)^{\circ 2}\arrow[d,"\mathsf{T}_i"] \\
\Delta(\underline{0})\circ \Delta(\underline{1}) \circ \Delta(\underline{0}) \circ \Delta(\underline{1}) \arrow[r,twoheadrightarrow] & \Delta_i^e(\delta)^{\circ 2} \\
\end{tikzcd}
\]
where both horizontal homomorphisms are surjective. 
Therefore, we obtain 
\[
\mathsf{T}_i (u \boxtimes u') = \tau_{w[N+1,N+1]}(u' \boxtimes u) +t u\boxtimes u' \ (u,u' \in \Delta_i^e(\delta)). 
\]
\end{proof}

Now, it is proved that $\mathsf{T}_i$ is nonconstant and satisfies both (a) and (b), 
which completes the proof of Theorem \ref{thm:heckeaction}. 

\subsection{Lattices} \label{sub:lattices}

In this subsection, we fix $\mathcal{O}$ and $t \in \mathcal{O}^{\times}$ as follows: 
\begin{itemize}
\item For type $A_N^{(1)}$, let $\mathcal{O} = F[z,z^{-1}]$ for some field $F$ and $t = z$. 
\item For the other types, let $\mathcal{O} = \mathbb{Z}$ and $t = 1$. 
\end{itemize}
Over this coefficient ring, we have the following result: 

\begin{lemma} \label{lem:Oform}
Let $i \in \mathring{I}$. 
There exists an $R_{\mathcal{O}}(\delta)$-module $L_{i,\mathcal{O}}^e(\delta)$ that is free of finite rank as an $\mathcal{O}$-module, 
and such that $L_{i,\mathcal{O}}^e(\delta) \otimes_{\mathcal{O}} E$ is the minuscule imaginary cuspidal module $L_{i,E}^e(\delta)$ for any field extension $E \supset F$ and $t \in E^{\times}$. 
\end{lemma}

\begin{proof}
It is entirely analogous to \cite[Remark 5.5]{MR3175157} using Lemma \ref{lem:minusculeonecolor} (indeed, it is identical except for type $A_N^{(1)}$). 

Alternatively, in type $A_N^{(1)}$, one can explicitly construct the $\mathcal{O}$-form since $L_i^e(\delta)$ is a homogeneous module in the sense of \cite{MR2677617}. 
In fact, for $i \in \mathring{I}$, let 
\[
w_i = (s_0s_1 \cdots s_{i-1})(s_N s_{N-1} \cdots s_i) \in W,
\]
and let $C_i$ be the subset of $I^{\delta}$ consisting of all reduced words of $w_i$. 
We can define $L_{i,\mathcal{O}}^e(\delta)$ as an $R_{\mathcal{O}}(\delta)$-module equipped with a free $\mathcal{O}$-basis $\{v_{\nu} \mid \nu \in C_i\}$ concentrated in degree zero, 
where the action is given by the same formula to as in Lemma \ref{lem:homogeneous}.  
It is easy to see that its base change to any field $\mathbf{k}$ is an simple module, 
which is isomorphic to $L_{i,\mathbf{k}}^e(\delta)$ by the characterization given in Lemma \ref{lem:minusculeonecolor}. 
\end{proof}

\begin{remark}
The proof shows that $L_i^e(\delta)$ is defined over $\mathbb{Z}[t,t^{-1}]$. 
\end{remark}

\begin{corollary}
Let $\mathbf{k}$ be a field. 
For $i \in \mathring{I}$. 
the element $\chi(L_{i,\mathbf{k}}^e(\delta)) \in U_q^-(\mathfrak{g})$ does not depend on $\mathbf{k}$.   
\end{corollary}

\begin{proof}
By Corollary \ref{cor:formalcharacter}, we have   
\[
\ch_q(L_{i,\mathbf{k}}^e(\delta)) = \theta(\chi(L_{i,\mathbf{k}}^e(\delta))). 
\]
By Lemma  \ref{lem:Oform}, the left hand side does not depend on $\mathbf{k}$. 
Since $\theta$ is independent of $\mathbf{k}$ and is injective, 
it follows that $\chi(L_{i,\mathbf{k}}^e(\delta))$ is also independent of $\mathbf{k}$.
\end{proof}

We fix such a module $L_{i,\mathcal{O}}^e(\delta)$ for each $i \in \mathring{I}$. 
For $\underline{n} \in \mathbb{Z}_{\geq 0}^{\mathring{I}}$, 
we define 
\[
L_{\mathcal{O}}^e(\underline{n}) = L_{1,\mathcal{O}}^e(\delta)^{\circ n_1} \circ \cdots \circ L_{N,\mathcal{O}}^e(\delta)^{\circ n_N}. 
\]
Let $\mathcal{K}$ be the fraction field of $\mathcal{O}$. 
We may regard $L_{\mathcal{O}}^e(\underline{n})$ as a $\mathcal{O}$ lattice of $L_{\mathcal{K}}^e(\underline{n})$. 
Note that the endomorphism algebra $\End_{R_{\mathcal{O}}(\delta)} (L_{i,\mathcal{O}}^e(\delta))$ is isomorphic to $\mathcal{O}$, 
and there exists no nontrivial homomorphism from $L_{i,\mathcal{O}}^e(\delta)$ to $L_{j,\mathcal{O}}^e(\delta)$ if $i \neq j$. 
Let $H_{\underline{n},\mathcal{O}}$ denote the tensor product of Hecke algebras over $\mathcal{O}$ defined in a similar way to Subsection \ref{sub:heckeaction}. 

Hereafter, we need to fix an isomorphism of Theorem \ref{thm:heckeaction} as follows: 
\begin{itemize}
  \item For type $A_N^{(1)}$, we use the one constructed in Subsection \ref{sub:ANcase}; that is, the one determined by Lemma \ref{lem:Tiform}. 
  \item For the other types, we use the one constructed in \cite[Section 4.2]{MR3589160}. 
\end{itemize}

\begin{proposition}
For $\underline{n} \in \mathbb{Z}_{\geq 0}^{\mathring{I}}$, 
there exists a unique isomorphism 
\[
\End_{R_{\mathcal{O}}}(L_{\mathcal{O}}^e(\underline{n})) \simeq H_{\underline{n}, \mathcal{O}}, 
\]
whose base change to any field $E$ coincides with the isomorphism of Theorem \ref{thm:heckeaction} specified above. 
\end{proposition}

\begin{proof}
Except for type $A_N^{(1)}$, it is proved in \cite{MR3589160}. 
Assume $X_N^{(1)} = A_N^{(1)}$ and consider the inclusion $L_{\mathcal{O}}^e(\underline{n}) \subset L_{\mathcal{K}}^e(\underline{e})$. 
We already have an action of $H_{\underline{n},\mathcal{K}}$ on $L_{\mathcal{K}}^e(\underline{e})$ determined by the formula in Lemma \ref{lem:Tiform}. 
It shows that the subspace $L_{\mathcal{O}}^e(\underline{n})$ is stable under the action of $H_{\underline{n},\mathcal{O}}$. 
Hence, we obtain an algebra homomorphism 
\[
H_{\underline{n},\mathcal{O}} \to \End_{R_{\mathcal{O}}}(L_{\mathcal{O}}^e(\underline{n})). 
\]
Again by Lemma \ref{lem:Tiform}, this homomorphism is compatible with base change to any field. 
It remains to prove that this is an isomorphism. 
Note that, for $n \in \mathbb{Z}_{\geq 0}$ and $i \in \mathring{I}$, 
the action of $H_{n,\mathcal{O}}$ on $L_{i,\mathcal{O}}^e(\delta)^{\circ n}$ satisfies 
\[
T_w (u_1 \boxtimes \cdots \boxtimes u_n) = \tilde{\tau}_w (u_{w^{-1}(1)} \boxtimes \cdots \boxtimes u_{w^{-1}(n)}) \mod F_{<w} (L_{i,\mathcal{O}}^e(\delta)^{\circ n}) \ (w \in \mathfrak{S}_n). 
\]
Hence, the Mackey-filtration (Lemma \ref{lem:Mackey}) shows that that homomorphism is an isomorphism. 
\end{proof}

\section{Quantum imaginary Howe duality}

Let $\underline{n} = (n_i) \in \mathbb{Z}_{\geq 0}^{\mathring{I}}$ and put $n = \lvert \underline{n} \rvert = \sum_i n_i$. 
Recall $L^e(\underline{n}) = L_1^e(\delta)^{\circ n_1} \circ \cdots \circ L_N^e(\delta)^{n_N}$. 
We define a finite dimensional algebra
\[
\mathscr{S}_{\underline{n}} = R(\lvert \underline{n} \rvert \delta)/\Ann(L^e(\underline{n})). 
\]
In this section, we show that $\mathscr{S}_{\underline{n}}$ is Morita equivalent to a tensor product of $t$-Schur algebras. 
When $t = 1$, all the results in this section are already known \cite{MR3589160}, 
in which case the $t$-Schur algebra becomes the ordinary Schur algebra. 
We will see that even when $t \neq 1$ (with $X_N^{(1)} = A_N^{(1)}$), analogous results hold. 

\subsection{$t$-Schur algebras} \label{sub:tSchur}

In this subsection, we recall the definition of $t$-Schur algebras. 
Detailed information can be found in \cite{MR1804485}.

For $n, h \in \mathbb{Z}_{\geq 0}$, 
let $X(h,n)$ be the set of all compositions $\lambda = (\lambda_1,\ldots,\lambda_h)$ of $n$. 
(When $h=0$, $X(0,0)$ consists of a unique element while $X(0,n)$ is empty if $n \geq 1$.)
For a composition $\lambda \in X(h,n)$, 
we define 
\begin{equation} \label{eq:ylambda}
y_{\lambda} = (-1)^{\ell (w_{\lambda})} \sum_{w \in \mathfrak{S}_{\lambda}} (-t)^{\ell(w_{\lambda}) - \ell(w)} T_w \in H_{\lambda}, 
\end{equation}
where $H_{\lambda} = H_{\lambda_1} \otimes \cdots \otimes H_{\lambda_h} \subset H_n$ and $w_{\lambda}$ is the longest element of $\mathfrak{S}_{\lambda}$. 
We define a left $H_n$-module
\[
N^{\lambda} = H_n y_{\lambda}. 
\]
If $\lambda = (n)$, we write $y_n = y_{(n)}$.
In this case, the module $N^{(n)}$ is one-dimensional and $T_w y_n = (-1)^{\ell(w)} y_n$. 
For general $\lambda$, the module $N^{\lambda}$ is isomorphic to $H_n \otimes_{H_{\lambda}} H_{\lambda} y_{\lambda}$. 

We define the $t$-Schur algebra $S_{h,n}$ as the endomorphism algebra
\[
\End_{H_n} \left( \bigoplus_{\lambda \in X(h,n)} N^{\lambda} \right)^{\mathrm{op}}.  
\] 
(When $h=0$, $S_{0,0}$ is $\mathbf{k}$ while $S_{0,n}$ is the zero algebra if $n \geq 1$.)
For $\lambda, \mu \in X(h,n)$ and $u \in \mathfrak{S}_n$, 
the right multiplication in $H_n$ by 
\[
s_{\lambda,\mu}^u = \sum_{w \in \mathfrak{S}_{\lambda}u\mathfrak{S}_{\mu} \cap {}^{\lambda}\mathfrak{S}_n} T_w^{\sharp}
\]
induces a well-defined homomorphism of left $H_n$-modules $N^{\lambda} \to N^{\mu}$, 
where $\sharp$ is the automorphism of $H_n$ defined by $T_k^{\sharp} = -T_k + t-1$. 
Moreover, it yields an element $\varphi_{\lambda,\mu}^u$ of $S_{h,n}$ via extension by zero. 
For $\lambda \in X(h,n)$, we write $\varphi_{\lambda,\lambda}^e = \varphi(\lambda)$, which is an idempotent. 
It is known that 
\[
\{ \varphi_{\lambda,\mu}^u \mid \lambda, \mu \in X(h,n), u \in {}^{\lambda}\mathfrak{S}_n^{\mu} \}
\]
is a basis of $S_{h,n}$. 
When $h \geq n$, let $\kappa \colon H_n \to S_{h,n}$ be the nonunital algebra homomorphism given by 
\[
\kappa(T_w) = \varphi_{(1^n),(1^n)}^w \ (w \in \mathfrak{S}_n),
\]
and let $\kappa^{\sharp}$ be the composition $\kappa \circ \sharp$. 
These two homomorphisms restricts to isomorphisms $H_n \simeq \varphi(1^n) S_{h,n} \varphi(1^n)$. 

For $h',h'' \in \mathbb{Z}_{\geq 0}$ with $h' + h'' = h$,
there is a Levi subalgebra
\[
S_{(h',h''),n} = \prod_{n',n'' \geq 0, n' + n'' = n} (S_{h',n'} \otimes S_{h'',n''}) \subset S_{h,n},
\]
which allows us to define a functor
\[
\Ind_{S_{(h',h''),n}}^{S_{h,n}} \Infl_{S_{h',n'}\otimes S_{h'',n''}}^{S_{(h',h''),n}} \colon \Mod{(S_{h',n'}\otimes S_{h'',n''})} \to \Mod{S_{h,n}}. 
\]

There are $\mathring{I}$-tuple analogues of the definitions above. 
For $\underline{n} = (n_i), \underline{h} = (h_i) \in \mathbb{Z}_{\geq 0}^{\mathring{I}}$, 
we define 
\[
X(\underline{h},\underline{n}) = \prod_{i \in \mathring{I}} X(h_i,n_i), \ S_{\underline{h},\underline{n}} = \bigotimes_{i \in \mathring{I}} S_{h_i,n_i}. 
\]
Similarly, for $\underline{\lambda} = (\lambda^{(i)}), \underline{\mu} = (\mu^{(i)}) \in X(\underline{h},\underline{n})$ and $\underline{u} = (u_i) \in \mathfrak{S}_{\underline{n}}$, 
we define 
\begin{align*}
s_{\underline{\lambda},\underline{\mu}}^{\underline{u}} &= s_{\lambda^{(1)}, \mu^{(1)}}^{u_1} \otimes \cdots \otimes s_{\lambda^{(N)}, \mu^{(N)}}^{u_N} \in H_{\underline{n}}, \\
\varphi_{\underline{\lambda},\underline{\mu}}^{\underline{u}} &= \varphi_{\lambda^{(1)},\mu^{(1)}}^{u_1} \otimes \cdots \otimes \varphi_{\lambda^{(N)},\mu^{(N)}}^{u_N} \in S_{\underline{h},\underline{n}}, \\
\varphi(\underline{\lambda}) &= \varphi(\lambda^{(1)}) \otimes \cdots \otimes \varphi(\lambda^{(N)}) \in S_{\underline{h},\underline{n}}. 
\end{align*}
For $\underline{h}', \underline{h}'', \underline{n}', \underline{n}'' \in \mathbb{Z}_{\geq 0}^{\mathring{I}}$ with $\underline{h}' + \underline{h}'' = \underline{h}, \underline{n}' + \underline{n}'' = \underline{n}$, 
there is a Levi subalgebra $S_{(\underline{h}',\underline{h}''),\underline{n}}$ of $S_{\underline{h},\underline{n}}$ and an exact functor
\[
\Ind_{S_{(\underline{h}',\underline{h}''),\underline{n}}}^{S_{\underline{h},\underline{n}}} \Infl_{S_{\underline{h}',\underline{n}'}\otimes S_{\underline{h}'',\underline{n}''}}^{S_{(\underline{h}',\underline{h}''),\underline{n}}} \colon \Mod{(S_{\underline{h}',\underline{n}'}\otimes S_{\underline{h}'',\underline{n}''})} \to \Mod{S_{\underline{h},\underline{n}}}. 
\]
 
\subsection{Quantum imaginary Howe duality}

Let $\underline{n} = (n_i) \in \mathbb{Z}_{\geq 0}^{\mathring{I}}$ and set $n = \lvert \underline{n} \rvert$. 
By Theorem \ref{thm:heckeaction}, $L^e(\underline{n})$ is a left $H_{\underline{n}}$-module.  
By using the antiautomorphism of $H_{\underline{n}}$ fixing all the generators $T_k$, 
we obtain a right action of $H_{\underline{n}}$ on $L^e(\underline{n})$. 
In what follows, we regard $L^e(\underline{n})$ as an $(\mathscr{S}_{\underline{n}},H_{\underline{n}})$-bimodule in this way. 
For $i \in \mathring{I}$, let $\iota_i \colon H_{n_i} \to H_{\underline{n}}$ denote the inclusion into the $i$-th factor.  

Let $\underline{h} \in \mathbb{Z}_{\geq 0}^{\mathring{I}}$. 
For $\underline{\lambda} \in X(\underline{h},\underline{n})$, 
we define a submodule of $L^e(\underline{n})$: 
\[
Z^{\underline{\lambda}} = \{ v \in L^e(\underline{n}) \mid v \cdot \iota_i (T_w) = (-1)^{\ell(w)}v \ (i \in \mathring{I}, w \in \mathfrak{S}_{\lambda^{(i)}}) \}. 
\]
For $\underline{\lambda}, \underline{\mu} \in X(\underline{h},\underline{n})$ and $\underline{u} \in \mathfrak{S}_{\underline{n}}$, 
the right multiplication by $s_{\underline{\lambda},\underline{\mu}}^{\underline{u}}$ on $L^e(\underline{n})$ induces a homomorphism of $\mathscr{S}_{\underline{n}}$-modules $Z^{\underline{\lambda}}\to Z^{\underline{\mu}}$. 
It yields an endomorphism $\psi_{\underline{\lambda},\underline{\mu}}^{\underline{u}}$ of $Z^{\underline{h},\underline{n}} \coloneqq \oplus_{\underline{\lambda} \in X(\underline{h},\underline{n})} Z^{\underline{\lambda}}$ via extension by zero. 
We regard $S_{\underline{h},\underline{n}}$ and $H_{\underline{n}}$ as graded algebras concentrated in degree zero. 

In the rest of this subsection, we assume that $\mathbf{k}$ is a field. 

\begin{theorem} \label{thm:howe}
Assume that $\mathbf{k}$ is a field. 

(1) $Z^{\underline{h},\underline{n}}$ is a projective $\mathscr{S}_{\underline{n}}$-module, 
and we have an isomorphism of graded $\mathbf{k}$-algebras
\[
\End_{\mathscr{S}_{\underline{n}}} ( Z^{\underline{h},\underline{n}})^{\mathrm{op}} \simeq S_{\underline{h},\underline{n}}
\]
given by $\psi_{\underline{\lambda},\underline{\mu}}^{\underline{u}} \mapsto \varphi_{\underline{\lambda},\underline{\mu}}^{\underline{u}} \ (\underline{\lambda},\underline{\mu} \in X(\underline{h},\underline{n}), \underline{u} \in {}^{\underline{\lambda}}\mathfrak{S}_{\underline{n}}^{\underline{\mu}})$. 

(2) If $h_i \geq n_i$ for all $i \in \mathring{I}$, then $Z^{\underline{h},\underline{n}}$ is a projective generator of $\mathscr{S}_{\underline{n}}$, 
and the functors
\begin{align*}
\alpha_{\underline{h},\underline{n}} \coloneqq \Hom_{\mathscr{S}_{\underline{n}}}(Z^{\underline{h},\underline{n}},?) \colon \gmod{\mathscr{S}_{\underline{n}}} &\to \gmod{S_{\underline{h},\underline{n}}}, \\
\beta_{\underline{h},\underline{n}} \coloneqq Z^{\underline{h},\underline{n}}\otimes_{S_{\underline{h},\underline{n}}} ? \colon \gmod{S_{\underline{h},\underline{n}}} &\to \gmod{\mathscr{S}_{\underline{n}}}, 
\end{align*}
are mutually quasi-inverse equivalences. 
Moreover, we have a commutative diagram 
\[
\begin{tikzcd}
\gmod{\mathscr{S}_{\underline{n}}} \arrow[rr,yshift=0.7ex,"\alpha_{\underline{h},\underline{n}}"]\arrow[rd,"\gamma_{\underline{n}}"'] && \gmod{S_{\underline{h},\underline{n}}} \arrow[ll,yshift=-0.7ex,"\beta_{\underline{h},\underline{n}}"]\arrow[ld,"f_{\underline{h},\underline{n}}"] \\
& \gmod{H_{\underline{n}}} & 
\end{tikzcd}
\]
where 
\[
\gamma_{\underline{n}} = \Hom_{\mathscr{S}_{\underline{n}}}(L^e(\underline{n}),?), \ f_{\underline{h},\underline{n}}(X) = \varphi ((1^{n_i})_i) X, 
\]
and $H_{\underline{n}}$ is identified with $\varphi ((1^{n_i})_i) S_{\underline{h},\underline{n}} \varphi ((1^{n_i})_i)$ via $\kappa^{\sharp}$. 
\end{theorem}

\begin{proof}
The proof is completely analogous to \cite{MR3589160} (and is identical if $t = 1$), 
with the Schur algebras appropriately replaced by the $t$-Schur algebras. 
Specifically, we need to modify the following two parts in type $A$. 

First, \cite[Lemma 4.2.4]{MR3589160} is replaced by the following assertion. 
Using the notation in the proof of Lemma \ref{lem:Oform}, 
let $i \in \mathring{I}, \nu \in C_i$, and fix a nonzero vector $v \in e(\nu)L_i^e(\delta)$. 
Let $u_0 \in \mathfrak{S}_{2(N+1)}$ be the element of minimal length satisfying 
\[
u_0(\nu,\nu) = \nu^{(2)} \coloneq (\nu_1,\nu_1,\nu_2,\nu_2,\ldots,\nu_{N+1},\nu_{N+1}). 
\]
Then, we have 
\begin{equation} \label{eq:modify1}
\tau_{u_0}\tau_{w[N+1,N+1]}(v\boxtimes v) = -(t+1) \tau_{u_0} (v\boxtimes v). 
\end{equation}
To see this, note that the highest degree component of $e(\nu^{(2)})(L_i^e(\delta)^{\circ 2})$ is one-dimensional, which is spanned by $\tau_{u_0}(v \boxtimes v)$. 
Since the left hand side of (\ref{eq:modify1}) belongs to this space, 
it coincides with $a \tau_{u_0} (v\boxtimes v)$ for some $a \in \mathbf{k}$. 
Let $u' \in \mathfrak{S}_{2(N+1)}$ be the element such that $u'u_0 = w[N+1,N+1]$. 
Then, we have $\tau_{u'} \tau_{u_0} = \tau_{w[N+1,N+1]}$, and 
\begin{equation} \label{eq:modify2}
\tau_{w[N+1,N+1]}^2(v \boxtimes v) = a \tau_{w[N+1,N+1]}(v\boxtimes v).
\end{equation}
On the other hand, Lemma \ref{lem:Tiform} shows that
\[
(v \boxtimes v)T_1 = (\tau_{w[N+1,N+1]}+t) (v \boxtimes v).
\]
Hence, the quadratic relation of $T_1$ implies
\begin{equation}\label{eq:modify3}
(\tau_{w[N+1,N+1]}+t+1)\tau_{w[N+1,N+1]}(v \boxtimes v) = 0. 
\end{equation}
Since $\tau_{w[N+1,N+1]} (v \boxtimes v)$ is nonzero, 
(\ref{eq:modify2}) and (\ref{eq:modify3}) show that $a =  -(t+1)$,
which proves equation (\ref{eq:modify1}). 

Second, we modify the computation at the end of the proof of \cite[Proposition 5.1.3]{MR3589160} as follows:
\begin{align*}
\tau_{u_0}(v \boxtimes v) T_1 &= \tau_{u_0} (\tau_{w[N+1,N+1]}+t)(v \boxtimes v) \quad \text{by Lemma \ref{lem:Tiform}} \\
&= (-t-1+t)\tau_{u_0}(v\boxtimes v) \quad \text{by (\ref{eq:modify1})} \\
&= -\tau_{u_0}(v \boxtimes v). 
\end{align*}
\end{proof}

The following theorem is also proved by an argument parallel to that in \cite[Chapter 6]{MR3589160} using Theorem \ref{thm:heckeaction} (3). 

\begin{theorem} \label{thm:monoidal}
Assume that $\mathbf{k}$ is a field. 
Let $\underline{h}', \underline{h}'', \underline{n}', \underline{n}'' \in \mathbb{Z}_{\geq 0}^{\mathring{I}}$, and put $\underline{h}' + \underline{h}'' = \underline{h}, \underline{n}' + \underline{n}'' = \underline{n}$. 
Assume $h'_i \geq n'_i, h''_i \geq n''_i$ for all $i \in \mathring{I}$. 
The following assertions hold: 

(1) The induction functor 
\[
\Ind_{\lvert \underline{n}' \rvert \delta, \lvert \underline{n}'' \rvert \delta} \colon \gmod{(R(\lvert \underline{n}'\rvert \delta) \otimes R(\lvert \underline{n}''\rvert \delta))} \to  \gmod{R(\lvert \underline{n}\rvert \delta)}
\]
restricts to a functor
\[
\gmod{(\mathscr{S}_{\underline{n}'}\otimes \mathscr{S}_{\underline{n}''})} \to \gmod{\mathscr{S}_{\underline{n}}}. 
\]

(2) Under the Morita equivalences 
\[
\gmod{\mathscr{S}_{\underline{n}}} \simeq \gmod{S_{\underline{h},\underline{n}}},\ \gmod{(\mathscr{S}_{\underline{n}'}\otimes \mathscr{S}_{\underline{n}''})} \simeq \gmod{(S_{\underline{h}',\underline{n}'} \otimes S_{\underline{h}'',\underline{n}''})}
\]
of Theorem \ref{thm:howe}, 
the following two functors are naturally isomorphic to each other: 
\begin{align*}
\Ind_{\lvert \underline{n}' \rvert \delta,  \lvert \underline{n}'' \rvert \delta} \colon & \gmod{(\mathscr{S}_{\underline{n}'}\otimes \mathscr{S}_{\underline{n}''})} \to \gmod{\mathscr{S}_{\underline{n}}}, \\
\Ind_{S_{(\underline{h}',\underline{h}''),\underline{n}}}^{S_{\underline{h},\underline{n}}} \Infl_{S_{\underline{h}',\underline{n}'}\otimes S_{\underline{h}'',\underline{n}''}}^{S_{(\underline{h}',\underline{h}''),\underline{n}}} &\colon \gmod{(S_{\underline{h}',\underline{n}'}\otimes S_{\underline{h}'',\underline{n}''})} \to \gmod{S_{\underline{h},\underline{n}}}. 
\end{align*}

(3) We have a natural isomorphism
\[
\gamma_{\underline{n}} (M' \circ M'') \simeq \Ind_{H_{\underline{n}'} \otimes H_{\underline{n}''}}^{H_{\underline{n}}} (\gamma_{\underline{n}'}(M') \otimes \gamma_{\underline{n}''}(M'')) 
\]
for $M' \in \gmod{\mathscr{S}_{\underline{n}'}}, M'' \in \gmod{\mathscr{S}_{\underline{n}''}}$. 
\end{theorem}

It is known that simple $S_{\underline{h},\underline{n}}$-modules (concentrated in degree zero) are parametrized by the set $X_+(\underline{h},\underline{n})$ of multipartitions $\underline{\lambda} = (\lambda^{(i)}) \in X(\underline{h},\underline{n})$: 
\[
\lambda_1^{(i)} \geq \lambda_2^{(i)} \geq \cdots \geq \lambda_{h_i}^{(i)} \ (i \in \mathring{I}). 
\]
In fact, there exists a unique simple $S_{\underline{h},\underline{n}}$-module $L_{\underline{h}}(\underline{\lambda})$ of highest weight $\underline{\lambda}$, that is, 
$\varphi (\underline{\lambda}) L_{\underline{h}}(\underline{\lambda}) \neq 0$ and $\varphi (\underline{\mu}) L_{\underline{h}}(\underline{\lambda}) = 0$ for any $\underline{\mu} \in X(\underline{h},\underline{n})$ with $\underline{\mu} \leq \underline{\lambda}$ in the dominance order. 

For $\underline{n} \in \mathbb{Z}_{\geq 0}^{\mathring{I}}$, 
let $X_+(\underline{n})$ denote the set of all multipartitions $\underline{\lambda} = (\lambda^{(i)})_{i \in \mathring{I}}$ such that each $\lambda^{(i)}$ is a partition of $n^{(i)}$. 
Its elements are called multipartitions of $\underline{n}$. 
Taking $\underline{h} \in \mathbb{Z}_{\geq 0}^{\mathring{I}}$ satisfying $h_i \geq n_i \ (i \in \mathring{I})$, 
we may regard $\underline{\lambda}$ as an element of $X_+(\underline{h},\underline{n})$. 
As in \cite[Section 6.1]{MR3589160}, one can show that the $\mathscr{S}_{\underline{n}}$-module $\beta_{\underline{h},\underline{n}} (L_{\underline{h}}(\underline{\lambda}))$ is independent of the choice of $\underline{h}$. 
Let $L^e(\underline{\lambda})$ denote this simple module. 

Let $i \in \mathring{I}$ and let $\lambda$ be a partition. 
We write $L_i^e(\lambda) = L^e(\underline{\lambda})$, where $\underline{\lambda}$ is a multipartition given by 
\[
\lambda^{(i)} = \lambda, \ \lambda^{(j)} = \emptyset \ (j \neq i). 
\]

\begin{lemma} \label{lem:onecolor2}
Assume that $\mathbf{k}$ is a field. 
Let $\underline{\lambda}$ be a multipartition. 
Then, we have an isomorphism 
\[
L^e(\underline{\lambda}) \simeq L_1^e(\lambda^{(1)}) \circ \cdots \circ L_N^e(\lambda^{(N)}).  
\]
Moreover, each $L_i^e(\lambda^{(i)})$ belongs to $\Xi^i$.  
\end{lemma}

\begin{proof}
The former assertion immediately follows from the definition. 
Since the module $L_i^e(\lambda^{(i)})$ is a subquotient of $L_i^e(\delta)^{\circ \lvert \lambda^{(i)}\rvert}$ and $L_i^e(\delta) \in \Xi_1^i$ by Lemma \ref{lem:minusculeonecolor}, 
the latter assertion follows. 
\end{proof}

\begin{corollary} 
Assume that $\mathbf{k}$ is a field. 
The set $\{L^e(\underline{\lambda}) \mid \underline{\lambda} \in X_+(\underline{n}) \}$ is a complete set of representatives of self-dual simple $\mathscr{S}_{\underline{n}}$-modules up to isomorphism. 
\end{corollary}

\begin{corollary} \label{cor:imaginarysimple}
Assume that $\mathbf{k}$ is a field. 
Let $n \in \mathbb{Z}_{\geq 0}$. 
Then, the set 
\[
\{L^e(\underline{\lambda}) \mid \text{$\underline{\lambda}$ is a multipartition of $n$} \}
\]
is a complete set of representatives of self-dual simple $R^e(n\delta)$-modules up to isomorphism. 
\end{corollary}

\begin{proof}
For each $i \in \mathring{I}$ and partitions $\lambda, \mu$, 
the modules $L_i^e(\lambda)$ and $L_i^e(\mu)$ are not isomorphic to each other if $\lambda \neq \mu$. 
Hence, Theorem \ref{thm:onecolor} and Lemma \ref{lem:onecolor2} show that two modules $L^e(\underline{\lambda})$ and $L^e(\underline{\mu})$ are isomorphic to each other only when $\underline{\lambda} = \underline{\mu}$. 
By counting the numbers based on Theorem \ref{thm:cuspdecomp} (2), we see that all the simple $R^e(n\delta)$-modules are obtained in this way. 
\end{proof}

\begin{corollary}  \label{cor:ff}
Assume that $\mathbf{k}$ is a field. 
Let $n \in \mathbb{Z}_{\geq 0}$. 
The functor 
\[
\bigoplus_{\underline{n} \in \mathbb{Z}_{\geq 0}^{\mathring{I}}, \lvert \underline{n} \rvert =n} \gmod{\mathscr{S}_{\underline{n}}} \to \gmod{R^e(n\delta)}
\]
induced by inflation is fully faithful. 
\end{corollary}

\begin{proof}
We may regard each $\gmod{\mathscr{S}_{\underline{n}}}$ as a full subcategory of $\gmod{R^e(n\delta)}$ closed under taking subquotients. 
The self-dual simple modules in $\gmod{\mathscr{S}_{\underline{n}}}$ are parametrized by multipartitions of $\underline{n}$. 
Hence, for $\underline{m} \neq \underline{n}$,
the two subcategories $\gmod{\mathscr{S}_{\underline{m}}}$ and $\gmod{\mathscr{S}_{\underline{n}}}$ have no simple modules in common. 
Therefore, the corollary follows. 
\end{proof}

\subsection{Weyl modules and characters}

Assume that $\mathbf{k}$ is a field. 
Let $\underline{n},\underline{h} \in \mathbb{Z}_{\geq 0}^{\mathring{I}}$ and assume that $h_i \geq n_i$ for any $i \in \mathring{I}$. 
The $t$-Schur algebra $S_{h,n}$ is a quasi-hereditary algebra with the poset $X_+(\underline{h},\underline{n})$ endowed with the dominance order. 
For each $\underline{\lambda} \in X_+(\underline{h},\underline{n})$, let $W_{\underline{h}}(\lambda)$ denote the standard module of highest weight $\underline{\lambda}$. 
We define 
\[
W^e(\underline{\lambda}) = \beta_{\underline{h},\underline{n}}W_{\underline{h}}(\underline{\lambda}). 
\]
Since $\beta_{\underline{h},\underline{n}}$ is an equivalence, 
we have the following proposition: 

\begin{proposition}\label{prop:quasihereditary}
Assume that $\mathbf{k}$ is a field. 
Then, the algebra $\mathscr{S}_{\underline{n}}$ is a quasi-hereditary algebra with poset $X_+(\underline{n})$ endowed with the dominance order. 
The standard modules are $W^e(\underline{\lambda}) \ (\underline{\lambda} \in X_+(\underline{n}))$. 
\end{proposition}

In particular, $W^e(\underline{\lambda})$ is independent of $\underline{h}$. 
 
\begin{theorem} \label{thm:Weylcharacter}
Assume that $\mathbf{k}$ is a field, and let $\underline{n} \in \mathbb{Z}_{\geq 0}^{\mathring{I}}$. 
Then, for any $\underline{\lambda} \in X_+(\underline{n})$, the character $\chi(W^e(\underline{\lambda})) \in U_q^-(\mathfrak{g})$ belongs to the dual canonical basis. 
\end{theorem}

\begin{proof}
Take $\mathcal{O}$ as in Subsection \ref{sub:lattices}. 
Then, there exists an $\mathscr{S}_{\underline{n},\mathcal{O}}$-module $W_{\mathcal{O}}^e(\underline{\lambda})$ that is free of finite rank as an $\mathcal{O}$-module, 
and whose base change to any field $\mathbf{k}$ coincides with $W_{\mathbf{k}}(\underline{\lambda})$.
In fact, it is proved in \cite[Theorem 6.4.3]{MR3589160} except for type $A_N^{(1)}$, 
and analogous proof applies to type $A_N^{(1)}$ by virtue of the results in Subsection \ref{sub:lattices}. 
Therefore, we deduce that 
\[
\chi(W_{\mathbf{k}}^e(\underline{\lambda})) = \chi(W_{\mathcal{K}}^e(\underline{\lambda})), 
\]
where $\mathcal{K}$ is the fraction field of $\mathcal{O}$. 
Hence, it is enough to prove the theorem when the quantum characteristic of $t$ is infinite, 
which will be proved at the end of next subsection. 
\end{proof}

\begin{remark} \label{rem:character}
Assume that $\mathbf{k}$ is a field.
Let $\underline{n}, \underline{h} \in \mathbb{Z}_{\geq 0}^{\mathring{I}}$ and assume that $h_i \geq n_i$ for all $i \in \mathring{I}$. 
Since $\mathscr{S}_{\underline{n}}$ is graded Morita equivalent to $S_{\underline{h},\underline{n}}$ concentrated in degree zero, 
we have 
\[
[W^e(\underline{\lambda}):L^e(\underline{\mu})]_q = [W_{\underline{h}}(\underline{\lambda}):L^e(\underline{\mu})]
\]
for any $\underline{\lambda}, \underline{\mu} \in X_+(\underline{h},\underline{n})$. 
Note that the Laurent polynomial on the left is actually an integer. 
Hence, the characters of simple modules in the imaginary strata is computed from dual canonical basis and decomposition matrix of the $t$-Schur algebra.
In particular, if the quantum characteristic of $t$ is infinite, 
the decomposition matrix is the identity matrix. 
If the characteristic of $\mathbf{k}$ is zero and $t$ is a primitive $e$-th root of unity, 
the decomposition matrix is described in terms of coefficients of the canonical basis of the Fock space representation of $U_q(\hat{\mathfrak{sl}}_e)$ or certain parabolic Kazhdan-Lusztig polynomials of affine Symmetric groups \cite{MR1722955}. 
\end{remark}

\subsection{Semisimple case}

In this subsection, assume that $\mathbf{k}$ is a field and that the quantum characteristic of $t$ is infinite;
that is, $1 + t + \cdots + t^{e-1} \neq 0$ for any $e \in \mathbb{Z}_{\geq 2}$. 
Under this assumption, it is well known that $S_{\underline{n},\underline{n}}$ is semisimple and is Morita equivalent to $H_{\underline{n}}$ via the functor $f_{\underline{n},\underline{n}}$.  
Hence, Theorem \ref{thm:howe} yields the following result. 

\begin{theorem}
Assume that the quantum characteristic of $t$ is infinite. 
Then, the functor 
\[
\gamma_{\underline{n}} \colon \gmod{\mathscr{S}_{\underline{n}}} \to \gmod{H_{\underline{n}}}
\]
is an equivalence for $\underline{n} \in \mathbb{Z}_{\geq 0}^{\mathring{I}}$.
The quasi-inverse is given by 
\[
X \mapsto L^e(\underline{n}) \otimes_{H_{\underline{n}}} X \ (X \in \gmod{H_{\underline{n}}}). 
\]
\end{theorem}

For multipartitions $\underline{\lambda}$ of $\underline{m}$, $\underline{\mu}$ of $\underline{n}$ and $\underline{\nu}$ of $\underline{m} + \underline{n}$, 
we define 
\[
c_{\underline{\lambda}, \underline{\mu}}^{\underline{\nu}} = \prod_{i \in \mathring{I}} c_{\lambda^{(i)}, \mu^{(i)}}^{\nu^{(i)}}, \index{$c_{\underline{\lambda}, \underline{\mu}}^{\underline{\nu}}$}
\]
where each $c_{\lambda^{(i)}, \mu^{(i)}}^{\nu^{(i)}}$ is the Littlewood-Richardson coefficient. 

\begin{corollary}[{cf.\ \cite[Theorem 19.6]{MR3694676}}] \label{cor:LRrule}
Assume that the quantum characteristic of $t$ is infinite. 
 For multipartitions $\underline{\lambda}$ of $\underline{m}$ and $\underline{\mu}$ of $\underline{n}$, we have a decomposition
 \[
 L^e(\underline{\lambda}) \circ L^e(\underline{\mu}) \simeq \bigoplus_{\text{$\underline{\nu}$: a multipartition of $\underline{m}+\underline{n}$}} L^e(\underline{\nu})^{\oplus c_{\underline{\lambda}, \underline{\mu}}^{\underline{\nu}}}.   
 \]
 For $m,n \in \mathbb{Z}_{\geq 0}$ and a multipartition $\underline{\lambda}$ of $m+n$, we have a decomposition 
 \[ 
 \Res_{m\delta,n\delta} L^e(\underline{\lambda}) \simeq \bigoplus_{\text{$\underline{\mu}$ (resp.\ $\underline{\nu}$): a multipartition of $m$ (resp.\ $n$)}} (L^e(\underline{\mu}) \otimes L^e(\underline{\nu}))^{\oplus c_{\underline{\mu}, \underline{\nu}}^{\underline{\lambda}}}. 
 \]
\end{corollary}

Recall that $\Delta^e(\underline{n})$ is a right $H_{\underline{n}}$-module. 
For a multipartition $\underline{\lambda}$ of $\underline{n}$, we define 
\[
\Delta^e(\underline{\lambda}) = \Delta^e(\underline{n}) \otimes_{H_{\underline{n}}} Sp(\underline{\lambda}), 
\]
where $Sp(\underline{\lambda}) = f_{\underline{n},\underline{n}}(L_{\underline{n}}(\underline{\lambda}))$ is the simple Specht module of $H_{\underline{n}}$. 

\begin{remark} \label{rem:specht}
Let $\underline{n} \in \mathbb{Z}_{\geq 0}^{\mathring{I}}$. 
If $\lambda^{(i)} = (n_i)$ for all $i \in \mathring{I}$, 
then $Sp(\underline{\lambda}) \simeq H_{\underline{n}}y_{\underline{n}}$, 
where 
\[
y_{\underline{n}} = y_{n_1} \otimes \cdots \otimes y_{n_N}.
\] 
If $\lambda^{(i)} = (1^{n_i})$ for all $i \in \mathring{I}$, 
then $Sp(\underline{\lambda}) \simeq H_{\underline{n}}x_{\underline{n}}$, 
where 
\[
x_{\underline{n}} = x_{n_1}\otimes \cdots \otimes x_{n_N}, \ x_{n_i} = \sum_{w \in \mathfrak{S}_{n_i}} T_w \in H_{\lambda^{(i)}}. 
\] 
\end{remark}

\begin{corollary} \label{cor:LRrule2}
  Assume that the quantum characteristic of $t$ is infinite. 
 For multipartitions $\underline{\lambda}$ of $\underline{m}$ and $\underline{\mu}$ of $\underline{n}$, we have a decomposition
 \[
 \Delta^e(\underline{\lambda}) \circ \Delta^e(\underline{\mu}) \simeq \bigoplus_{\text{$\underline{\nu}$: a multipartition of $\underline{m}+\underline{n}$}} \Delta^e(\underline{\nu})^{\oplus c_{\underline{\lambda}, \underline{\mu}}^{\underline{\nu}}}.   
 \]
 For $m,n \in \mathbb{Z}_{\geq 0}$ and a multipartition $\underline{\lambda}$ of $m+n$, we have a decomposition 
 \[ 
 \Res_{m\delta,n\delta} \Delta^e(\underline{\lambda}) \simeq \bigoplus_{\text{$\underline{\mu}$ (resp.\ $\underline{\nu}$): a multipartition of $m$ (resp.\ $n$)}} (\Delta^e(\underline{\mu}) \otimes \Delta^e(\underline{\nu}))^{\oplus c_{\underline{\mu}, \underline{\nu}}^{\underline{\lambda}}}. 
 \]
\end{corollary}

\begin{proposition} \label{prop:semisimpledelta}
Assume that the quantum characteristic of $t$ is infinite. 
For a multipartition $\underline{\lambda}$ of $\underline{n}$, 
the module $\Delta^e(\underline{\lambda})$ is a projective cover of the simple module $L^e(\underline{\lambda})$ in $\gMod{R^e}$. 
\end{proposition}

\begin{proof}
By definition, $\Delta^e(\underline{\lambda})$ is a direct summand of $\Delta^e(\underline{n})$, 
which is a projective $R^e(\lvert \underline{n} \rvert)$-module by Proposition \ref{prop:imaginaryind} (2). 
Hence, $\Delta^e(\underline{\lambda})$ is a projective module. 

Put $\lvert \underline{n} \rvert = n$. 
By Lemma \ref{lem:tensorspacehom} and Theorem \ref{thm:heckeaction}, $\Hom_R(\Delta^e(\underline{n}),L^e(\underline{n}))$ is concentrated in degree zero and is isomorphic to $H_{\underline{n}}$. 
Moreover, for any $\underline{n'} \in \mathbb{Z}_{\geq 0}^{\mathring{I}}$ such that $\underline{n'} \neq \underline{n}$,  
$\Hom_R(\Delta^e(\underline{n}),L^e(\underline{n'}))$ is $0$. 
Hence, we have 
\[
\dim_q \Hom_R(\Delta^e(\underline{\lambda}),L^e(\underline{\mu})) = \delta_{\underline{\lambda},\underline{\mu}}
\]
for any multipartition $\underline{\mu}$ of $n$. 
Since  every self-dual simple $R^e(n\delta)$-module is of the form $L^e(\underline{\mu})$ by Corollary \ref{cor:imaginarysimple}, 
$\Delta^e(\underline{\lambda})$ is a projective cover of $L^e(\underline{\lambda})$. 
\end{proof}

Recall the imaginary root vector $S_{i,\lambda}^e$ from Subsection \ref{sub:affinePBW}. 
For a multipartition $\underline{\lambda}$, we write 
\[
S_{\underline{\lambda}}^e = \prod_{i \in \mathring{I}} S_{i,\lambda^{(i)}}^e. 
\]

The following theorem is a generalization of \cite[Theorem 23.4]{MR3694676}. 
Its proof is provided in Subsection \ref{sub:character}. 

\begin{theorem} \label{thm:characterformula}
Assume that the quantum characteristic of $t$ is infinite. 
For a multipartition $\underline{\lambda}$, we have 
\[
\chi(\Delta^e(\underline{\lambda})) = S_{\underline{\lambda}}^e. 
\]
\end{theorem}

Let $\preceq$ be a convex order on $\Phi_+^{\text{min}}$ of coarse type $e$. 
Recall the standard modules from Subsection \ref{sub:stratifications} and the affine PBW basis from Subsection \ref{sub:affinePBW}. 
By Corollary \ref{cor:imaginarysimple}, 
we may identify the set $\Omega^{\preceq}(\beta)$ with $\Upsilon^{\preceq}(\beta)$ for $\beta \in \mathsf{Q}_+$. 

\begin{corollary} \label{cor:characterformula}
Assume that the quantum characteristic of $t$ is infinite. 
Let $\preceq$ be a convex order of coarse type $e$. 
For $\beta \in \mathsf{Q}_+$ and $\boldsymbol{c} \in \Upsilon^{\preceq}(\beta)$, 
we have 
\[
\chi (\Delta^{\preceq}(\boldsymbol{c})) = f_{\boldsymbol{c}}^{\preceq}.
\]
\end{corollary}

\begin{remark}
This result can be generalized to arbitrary convex order $\preceq$ of coarse type $w \in \mathring{W}$. 
In fact, we have 
\[
S_{\underline{\lambda}}^w = S_w(S_{\underline{\lambda}}^e)
\]
for any multipartition $\underline{\lambda}$ by \cite[Theorem 4.13]{MR3874704}. 
On the other hand, Proposition \ref{prop:reduction} shows that the reflection functor $\mathcal{S}_w$ induces a bijection $\Xi^e \to \Xi^w$. 
Hence, the set $\Xi^w$ is also parametrized by multipartitions.
When a simple module $L \in \Xi^w$ is parametrized by a multipartition $\underline{\lambda}$, we define
\[
\Delta^w(\underline{\lambda}) = \Delta^w(L). 
\]
Then, we have
\[
\chi (\Delta^w(\underline{\lambda})) = S_{\underline{\lambda}}^w. 
\]
\end{remark}

Now, we can complete the proof of Theorem \ref{thm:Weylcharacter}. 

\begin{proof}[End of the proof of Theorem \ref{thm:Weylcharacter}]
We have to prove, under the assumption that the quantum characteristic of $t$ is infinite,  
that $\chi(W^e(\underline{\lambda}))$ belongs to the dual canonical basis for any multipartition $\underline{\lambda}$. 
Since the $t$-Schur algebra is semisimple, we have $W^e(\underline{\lambda}) = L^e(\underline{\lambda})$. 
Let $\preceq$ be a convex order of coarse type $e$. 
For $\boldsymbol{c} = (\boldsymbol{c}_-, \mu, \boldsymbol{c}_+) \in \Upsilon^{\preceq}(\beta)$ for some $\beta \in \mathsf{Q}_+$, 
we have 
\begin{align*}
(\chi(\Delta^e(\boldsymbol{c})),\chi(L^e(\underline{\lambda}))) &= \overline{(\overline{\chi(\Delta^e(\boldsymbol{c}))}, c(\chi(L^e(\underline{\lambda}))))} \quad \text{by the definition of $c$ in Subsection \ref{sub:quantumgroups}} \\
&= \overline{(\overline{\chi(\Delta^e(\boldsymbol{c}))}, \chi(D(L^e(\underline{\lambda}))))} \quad \text{by Theorem \ref{thm:categorification}} \\
&= \overline{(\overline{\chi(\Delta^e(\boldsymbol{c}))}, \chi(L^e(\underline{\lambda})))} \quad \text{since $L^e(\underline{\lambda})$ is self-dual} \\
&= \overline{\Extform{\Delta^e(\boldsymbol{c})}{L^e(\underline{\lambda})}} \quad \text{by Theorem \ref{thm:categorification}} \\
&= \delta_{\boldsymbol{c},(0,\lambda,0)} \quad \text{by Theorem \ref{thm:stratification}}. 
\end{align*}
On the other hand, Corollary \ref{cor:characterformula} shows that $\chi(\Delta^e(\boldsymbol{c})) = f_{\boldsymbol{c}}^{\preceq}$. 
Hence, the assertion follows from the fact that the transition matrix between the canonical basis and PBW basis is unitriangular with respect to the bilexicographic order on $\mathcal{P}(\beta)$ \cite[Lemma 5.6]{MR2066942}, \cite[Theorem 4.28]{MR3874704}. 
\end{proof}

The rest of this paper is devoted to the proof of Theorem \ref{thm:characterformula}. 
While our proof largely follows \cite{MR3694676},
additional care is required because the original argument relies on a geometric result \cite[Lemma 7.5]{MR3694676}, 
which is not available in our general setting. 
To overcome this, we utilize several consequences from the theory of $R$-matrices developed by Kang, Kashiwara, Kim, Oh and Park \cite{MR3314831,MR3758148,MR3790066}, 
which we recall in the following subsection. 

\subsection{Consequences of the theory of $R$-matrices} \label{sub:Rmatrix}

In this subsection, we recall necessary results of the theory of $R$-matrices. 
Assume that $\mathbf{k}$ is a field. 

\begin{definition}[{\cite[Definition 2.1]{MR3790066}}] \label{def:affinization}
 Let $\beta \in \mathsf{Q}_+$ and $M \in \gmod{R(\beta)}$ be a simple module.

 (1) $M$ is called real if $M \circ M$ is also simple. 

 (2) An affine object of $M$ is a pair $(\widehat{M}, z = z_{\widehat{M}})$ of an $R(\beta)$-module $\widehat{M}$ and an injective endomorphism $z \in \End_{R(\beta)}(\widehat{M})$ of positive degree satisfying
 \begin{itemize}
  \item as a $\mathbf{k}[z]$-module, $\widehat{M}$ is free of finite rank, and
  \item there is an isomorphism of $R(\beta)$-modules $\widehat{M}/z\widehat{M} \simeq M$. 
 \end{itemize}

 (3) An affine object $(\widehat{M}, z)$ of $M$ is said to be an affinization if it additionally satisfies 
 \begin{itemize}
  \item $\mathfrak{p}_{i, \beta} \widehat{M} \neq 0$ for all $i \in I$, 
 \end{itemize}
 where $\mathfrak{p}_{i,\beta}$ is an element of the center of $R(\beta)$ defined by 
 \[
 \mathfrak{p}_{i, \beta} = \sum_{\nu \in I^{\beta}} \left( \prod_{1 \leq k \leq \height (\beta), \nu_k = i} x_k \right) e(\nu).  
 \] 
\end{definition}

\begin{theorem} \label{thm:affreal}
Let $\preceq$ be a convex order on $\Phi_+^{\mathrm{min}}$. 
For any real root $\beta \in \Phi_+^{\mathrm{re}}$, 
the cuspidal module $L^{\preceq}(\beta)$ is real,
and the endomorphism algebra $\End_R(\Delta^{\preceq}(\beta))$ is isomorphic to the polynomial ring $\mathbf{k}[z]$ with $\deg z = (\beta,\beta)$. 
(Note that the action of $z$ on $\Delta^{\preceq}(\beta)$ is unique up to a scalar multiple since $\mathbf{k}[z]_{(\beta,\beta)}$ is one-dimensional.)
Furthermore, the pair $(\Delta^{\preceq}(\beta),z)$ is an affinization of $L^{\preceq}(\beta)$. 
\end{theorem}

\begin{proof}
The former assertions follow from Theorem \ref{thm:cuspdecomp} (1) and Theorem \ref{thm:stratification}. 
In \cite[Section 6.2]{murata2025affinehighestweightstructures}, 
the cuspidal module $L^{\preceq}(\beta)$ (resp.\ the standard module $\Delta^{\preceq}(\beta)$) is identified with a certain determinantial module $M$ (resp.\ $\widehat{M}$, a determinantial module defined over $\mathbf{k}[z]$ with $\deg z = (\beta,\beta)$; see \cite[Section 3.6]{murata2025affinehighestweightstructures} for its precise definition). 
It is proved in \cite[Theorem 3.26]{MR4359265} that $(\widehat{M},z)$ is an affinization of $M$. 
\end{proof}

\begin{theorem} \label{thm:rmatrix} 
Let $M,N \in \gmod{R}$ be simple modules. 
Assume that $M$ is real and admits an affinization $(\widehat{M},z)$ with $\deg z = d > 0$.  
Then, the following assertions hold: 
 
(1) The graded vector space $\Hom_R(M\circ N, N \circ M)$ is one-dimensional. 
Let $\mathbf{r}_{M,N}$ be a spanning vector of it (unique up to a scalar multiple),
and let $\Lambda(M,N)$ be the degree of $\mathbf{r}_{M,N}$. 
Hence, $\mathbf{r}_{M,N}$ is a homomorphism in $\gmod{R}$ from $q^{\Lambda(M,N)} M\circ N$ to $N \circ M$. 

(2) The head $M \mynabla N$ of $M \circ N$ and the socle $N \myDelta M$ of $N \circ M$ are simple. 
Moreover, we have 
\[
q^{\Lambda(M,N)} M \mynabla N \simeq \Image \mathbf{r}_{M,N} = N \myDelta M. 
\]

(3) We have $\Lambda(M, M \mynabla N) = \Lambda(M,N)$ and $\Lambda(M,L) < \Lambda(M,N)$ for any composition factor $L$ of $\Ker (\mathbf{r}_{M,N})$. 
In particular, $[M\circ N: M \mynabla N]_q = 1$. 
Dually, we have $[N \circ M : N \myDelta M]_q = 1$. 

(4) The number $\tilde{\Lambda}(M,N) \coloneqq (\Lambda(M,N) + (\wt M, \wt N))/2$ is an integer. 
If both $M$ and $N$ are self-dual, then $q^{\tilde{\Lambda}(M,N)} M\mynabla N$ is also self-dual. 

\end{theorem} 

\begin{proof}
(1) (2) and (3) are proved in \cite{MR3790066} for instance. 
(4) is proved in \cite[Lemma 3.14]{MR3758148}. 
Note that it is proved in \cite[Lemma 3.12]{MR3758148} that $\tilde{\Lambda}(M,N)$ is an integer under some additional assumptions. 
However, the proof of \cite[Lemma 3.14]{MR3758148} simultaneously shows that $\tilde{\Lambda}(M,N)$ is an integer without using these additional assumptions. 
\end{proof}

\begin{definition}
Let $M \in \gMod{R(\alpha)}$ and $N \in \gMod{R(\beta)}$ be simple modules. 
The pair $(M,N)$ is called unmixing if 
\[
\Res_{\alpha,\beta} (M \circ N) = M \otimes N. 
\]
\end{definition}

\begin{theorem}[{\cite[Section 2.1]{MR4655348}}] \label{thm:unmixing}
Let $M, N \in \gMod{R}$ be simple modules, and assume that $M$ is real and admits an affinization. 
Assume that the pair $(M,N)$ is unmixing. 
Then, we have $\Lambda(M,N) = -(\wt M, \wt N)$. 
Furthermore, if both $M$ and $N$ are self-dual, then $M \mynabla N$ is self-dual. 
\end{theorem}

\subsection{Proof of Theorem \ref{thm:characterformula}} \label{sub:character}

In this subsection, we prove Theorem \ref{thm:characterformula}.
In particular, we assume that $\mathbf{k}$ is a field and that the quantum characteristic of $t$ is infinite. 

\begin{lemma} \label{lem:choice}
For $i \in \mathring{I}$, we have  
\[
\Res_{2(\delta-\alpha_i),2\alpha_i} L_i^e(1,1) = 0. 
\]
\end{lemma}

\begin{proof}
Suppose that $\Res_{2(\delta-\alpha_i),2\alpha_i}L_i^e(1,1) \neq 0$. 
By \cite[Propositoin 6.16 (1)]{murata2025affinehighestweightstructures}, we have $\Res_{\delta-\alpha_i,\alpha_i} L_i^e(\delta) \simeq L^e(\delta-\alpha_i) \otimes L(i)$, 
and $\Res_{\delta-2\alpha_i,\alpha_i} L^e(\delta-\alpha_i) = 0$. 
Hence, we have the following one-step Mackey filtration (hence an isomorphism) ignoring grading shift: 
\[
\Res_{2(\delta-\alpha_i),2\alpha_i} L_i^e(\delta)^{\circ 2} \simeq L_i^e(\delta-\alpha_i)^{\circ 2} \otimes L(i)^{\circ 2}. 
\]
Note that this module is simple by Theorem \ref{thm:cuspdecomp}. 
Hence, its nonzero submodule $\Res_{2(\delta-\alpha_i),2\alpha_i}L_i^e(1,1)$ coincides with $\Res_{2(\delta-\alpha_i),2\alpha_i}L_i^e(\delta)^{\circ 2}$. 

Similarly to \cite[Proposition 5.1.3 (iii)]{MR3589160}, there exists a non-zero vector $v \in e(2(\delta-\alpha_i),2\alpha_i) L_i^e(\delta)^{\circ 2}$ such that $v (T_1+1) = 0$. 
On the other hand, since $L_i^e(1,1) = L_i^e(\delta)^{\circ 2}(T_1 + 1)$ by Remark \ref{rem:specht},
we have $v(T_1 -t) = 0$. 
Hence, we have $(t+1)v = 0$ and $t = -1$. 
This contradicts the assumption that the quantum characteristic of $t$ is infinite.  
\end{proof}

\begin{remark}
Lemma \ref{lem:choice} ensures that our labelling of simple $R^e$-modules by multipartitions is consistent with \cite[Page 277]{MR3694676}. 
\end{remark}

\begin{lemma} \label{lem:width1}
For $i \in \mathring{I}$ and $n \in \mathbb{Z}_{\geq 1}$, 
we have 
\[
\Res_{n\delta-\alpha_i,\alpha_i}L_i^e(1^n) \simeq (L_i^e(1^{n-1})\mynabla L^e(\delta-\alpha_i)) \otimes L(i).  
\]
\end{lemma}

\begin{proof}
In view of Lemma \ref{lem:choice}, it is parallel to \cite[Theorem 21.6]{MR3694676}. 
\end{proof}

\begin{lemma}[{cf.\ \cite[Lemma 21.8]{MR3694676}}] \label{lem:innerprodSES}
  Let $i \in \mathring{I}$ and $n \in \mathbb{Z}_{\geq 1}$. 
  Then, $\Lambda(L^e(n\delta-\alpha_i), L_i^e(\delta)) = (\alpha_i, \alpha_i)$, and there is a short exact sequence
\[
 0 \to q_i L^e((n+1)\delta-\alpha_i) \to L_i^e(\delta) \circ L^e(n\delta-\alpha_i) \to L_i^e(\delta) \mynabla L^e(n\delta-\alpha_i) \to 0. 
 \]
\end{lemma}

\begin{proof}
  Take a convex order $\preceq$ of carse type $e$. 
  Put $\gamma = n\delta - \alpha_i$ and let $\boldsymbol{c} = (\boldsymbol{c}_-,\boldsymbol{c}_{\delta},\boldsymbol{c}_+)$ be an element of $\Upsilon^{\preceq}(\gamma + \delta)$ defined by setting 
  \[
  \boldsymbol{c}_-(\gamma) = 1, \ \boldsymbol{c}_{\delta}^{(i)} = (1), 
  \]
  and taking all the other entries to be trivial. 
  Let $\boldsymbol{c}'$ be the element of $\Upsilon^{\preceq}(\gamma + \delta)$ defined by setting $\boldsymbol{c}'_-(\gamma+\delta) =1$ and taking all the other entries to be trivial. 
  Then, we see that only $\boldsymbol{c}' \in \Upsilon^{\preceq}(\gamma+\delta)$ satisfies $\rho(\boldsymbol{c}') < \rho(\boldsymbol{c})$ in the bilexicographic order. 
  Hence, all the composition factors of $\overline{\Delta}(\boldsymbol{c}) = L_i^e(\delta) \circ L^e(\gamma)$ other than the head $L_i^e(\delta) \mynabla L^e(\gamma)$ are grading shifts of $L^e(\gamma+\delta)$. 
  Since $L^e(\gamma)$ is real and admits an affinization by Theorem \ref{thm:affreal}, 
  we see using Theorem \ref{thm:rmatrix} that $L^e(\gamma+\delta)$ appears in the socle of $L_i^e(\delta) \circ L^e(\gamma)$ with multiplicity one up to grading shift.  
  More precisely, we have an isomorphism
  \[
  L^e(\gamma+\delta) \simeq q^{\tilde{\Lambda}(L^e(\gamma),L_i^e(\delta))} L^e(\gamma)\mynabla L_i^e(\delta) \simeq q^{\tilde{\Lambda}(L^e(\gamma),L_i^e(\delta)) - \Lambda(L^e(\gamma),L_i^e(\delta))} L_i^e(\delta)\myDelta L^e(\gamma), 
   \] 
  and a short exact sequence
  \[
  0 \to q^c L^e(\gamma + \delta) \to L_i^e(\delta) \circ L^e(\gamma) \to L_i^e(\delta)\mynabla L^w(\gamma) \to 0, 
  \]
  where $c = -(\tilde{\Lambda}(L^e(\gamma),L_i^e(\delta)) - \Lambda(L^e(\gamma),L_i^e(\delta))) = \Lambda (L^e(\gamma), L_i^e(\delta))/2$.
 
  To complete the proof, it suffices to prove that $\Lambda(L^e(\gamma), L_i^e(\delta)) = (\alpha_i,\alpha_i)$. 
  We proceed by an induction on $n$. 
  First, assume that $n = 1$, that is, $\gamma = \delta-\alpha_i$. 
 We compute
  \begin{align*}
  &\Lambda (L^e(\delta-\alpha_i), L_i^e(\delta))  \\
  &= \Lambda (L^e(\delta-\alpha_i), L^e(\delta-\alpha_i) \mynabla L(i)) \quad \text{by Definition \ref{def:minuscule}} \\
  &= \Lambda (L^e(\delta-\alpha_i), L(i)) \quad \text{by Theorem \ref{thm:rmatrix}} \\
  &= -(\delta-\alpha_i, \alpha_i) \quad \text{by Theorem \ref{thm:unmixing} and \cite[Proposition 6.16]{murata2025affinehighestweightstructures}} \\
  &= (\alpha_i, \alpha_i). 
  \end{align*}
  Hence, the proposition holds in this case. 

  Next, assume that $n>1$ and that the proposition is proved for smaller numbers. 
  By the induction hypothesis, we have a surjective homomorphism 
  \[
  q_i L^e(\gamma - \delta) \circ L_i^e(\delta) \to L^e(\gamma). 
  \]
  By induction-restriction adjunction, we obtain a nonzero homomorphism
  \[
  q_i L^e(\gamma-\delta) \otimes L_i^e(\delta) \to \Res_{\gamma-\delta, \delta} L^e(\gamma), 
  \]
  which is injective since $L^e(\gamma-\delta)\otimes L_i^e(\delta)$ is simple. 
  It yields a sequence of injective homomorphisms
  \begin{align*}
  q_i^2 L^e(\gamma) \otimes L_i^e(\delta) &\to q_i(L_i^e(\delta) \circ L^e(\gamma-\delta)) \otimes L_i^e(\delta) \quad \text{by the induction hypothesis} \\
  &= q_i (\Ind_{\delta, \gamma-\delta} \otimes \Id_{\delta}) (L_i^e(\delta) \otimes L^e(\gamma-\delta) \otimes L_i^e(\delta)) \\ 
  &\to (\Ind_{\delta, \gamma-\delta} \otimes \Id_{\delta})(L_i^e(\delta) \otimes \Res_{\gamma-\delta, \delta}L^e(\gamma))  \quad \text{by the discussion above} \\
  &\to \Res_{\gamma,\delta} (L_i^e(\delta) \circ L^e(\gamma)) \quad \text{by Mackey filtration}. 
  \end{align*}
  By induction-restriction adjunction, we obtain a nonzero homomorphism
  \[
  q_i^2 L^e(\gamma) \circ L_i^e(\delta) \to L_i^e(\delta) \circ L(\gamma), 
  \]
  which means $\Lambda(L^e(\gamma), L_i^e(\delta)) = (\alpha_i,\alpha_i)$. 
  The induction step is complete.  
\end{proof}

\begin{lemma}[{cf.\ \cite[Proposition 21.9]{MR3694676}}] \label{lem:inductionSES}
  Let $m,n \in \mathbb{Z}_{\geq 1}$ and $i \in \mathring{I}$. 
  Then, we have a short exact sequence 
  \[
  0 \to L_i^e(1^m) \mynabla L^e(n\delta-\alpha_i) \to L^e(n\delta-\alpha_i)\circ L_i^e(1^m) \to q_i^{-1} L^e(1^{m-1}) \mynabla  L_i^e((n+1)\delta-\alpha_i) \to 0. 
  \]
\end{lemma}

\begin{proof}
It is proved exactly in the same way as \cite[Proposition 21.9]{MR3694676}, using Lemma \ref{lem:innerprodSES}.  
The apparent difference in the form of our short exact sequence is merely due to the application of $D$. 
\end{proof}

\begin{lemma}[{cf.\ \cite[Lemma 22.1]{MR3694676}}] \label{lem:innerprod0}
Let $i, j \in \mathring{I}$ and $n \in \mathbb{Z}_{\geq 1}$. 
Then, 
\[
(\overline{\psi_{i,n}^e}, \chi(L_j^e(1^n))) = \delta_{i,j} (-q_i)^{-n+1}
\]
\end{lemma}

\begin{proof}
Recall $\psi_{i,n}^e = f_{n\delta-\alpha_i}^e f_i - q_i^2 f_i f_{n\delta-\alpha_i}^e$ and $\chi(\Delta^e(n\delta-\alpha_i)) = f_{n\delta-\alpha_i}^e, \chi(R(\alpha_i)) = f_i$. 
Hence, using Theorem \ref{thm:categorification}, we have 
\begin{align*}
&(\overline{\psi_{i,n}^e}, \chi(L_j^e(1^n))) \\
&= \Extform{\Delta^e(n\delta-\alpha_i) \circ R(\alpha_i)}{L_j^e(1^n)} - q_i^{-2} \Extform{R(\alpha_i) \circ \Delta^e(n\delta-\alpha_i)}{L_j^e(1^n)} \\ 
&= \Extform{\Delta^e(n\delta-\alpha_i) \otimes R(\alpha_i)}{\Res_{n\delta-\alpha_i,\alpha_i}L_j^e(1^n)} \\
& \quad - q_i^{-2}\Extform{R(\alpha_i)\otimes \Delta^e(n\delta-\alpha_i)}{\Res_{\alpha_i,n\delta-\alpha_i}L_j^e(1^n)}, 
\end{align*}
where the last equality follows from induction-restriction adjunction. 
Since $L_j^e(1^n)$ is a cuspidal module, the latter restriction is zero. 
Since $L_j^e(1^n)$ belongs to $\Xi^j$, the former restriction is also zero unless $j = i$ (Lemma \ref{lem:onecolor}). 
From now on, we assume $j = i$. 
By Lemma \ref{lem:width1}, the former restriction is $(L_i^e(1^{n-1})\mynabla L^e(\delta-\alpha_i)) \otimes L(i)$. 
Moreover, using Lemma \ref{lem:inductionSES} we see that 
\[
\chi(L_i^e(1^{n-1})\mynabla L^e(\delta-\alpha_i)) = \sum_{1\leq s \leq n} (-q_i)^{-s+1} \chi(L^e(s\delta-\alpha_i)\circ L_i^e(1^{n-s})). 
\]
Hence, 
\begin{align*}
&(\overline{\psi_{i,n}^e}, \chi(L_i^e(1^n))) \\
&= \sum_{1 \leq s \leq n} (-q_i)^{-s+1} \Extform{\Delta^e(n\delta-\alpha_i)}{L^e(s\delta-\alpha_i) \circ L_i^e(1^{n-s})} \Extform{R(\alpha_i)}{L(i)} \\
&= \sum_{1\leq s \leq n} (-q_i)^{-s+1} \Extform{\Res_{(n-s)\delta,s\delta-\alpha_i}\Delta^e(n\delta-\alpha_i)}{L_i^e(1^{n-s})\otimes L^e(s\delta-\alpha_i)} \\
&\quad \text{by restriction-coinduction adjunction}. 
\end{align*}
Since $\Delta^e(n\delta-\alpha_i)$ is a cuspidal module, the restriction is zero unless $s = n$. 
Therefore, we obtain
\begin{align*}
(\overline{\psi_{i,n}^e}, \chi(L_i^e(1^n))) &= (-q_i)^{-n+1} \Extform{\Delta^e(n\delta-\alpha_i)}{L^e(n\delta-\alpha_i)} \\
&= (-q_i)^{-n+1} \quad \text{by Theorem \ref{thm:stratification}}. 
\end{align*}
\end{proof}

\begin{lemma}[{cf.\ \cite[Lemma 22.2]{MR3694676}}] \label{lem:innerprod}
 Let $i,j \in \mathring{I}$ and $n \in \mathbb{Z}_{\geq 1}$. 
 Then, 
 \[
 (\overline{P_{i,n}^e},\chi(L_j^e(1^n))) = \begin{cases}
 1 & \text{if $i=j, n =1$}, \\
 0 & \text{otherwise}. 
 \end{cases}
 \]
\end{lemma}

\begin{proof}
We proceed by induction on $n$. 
If $n = 1$, $P_{i,1}^e = \psi_{i,1}^e$ and it follows from Lemma \ref{lem:innerprod0}. 
Assume $n\geq 2$. 
By the definition of $P_{i,n}^e$, we have 
\begin{align*}
&(\overline{P_{i,n}^e},\chi(L_j^e(1^n))) \\
&= \frac{1}{[n]_i} \sum_{s=1}^n q_i^{-(n-s)} (\overline{\psi_{i,s}^e P_{i,n-s}^e}, \chi(L_j^e(1^n))) \\
&= \frac{1}{[n]_i} \sum_{s=1}^n q_i^{-(n-s)} (\overline{\psi_{i,s}^e} \otimes \overline{P_{i,n-s}^e}, r_{s\delta,(n-s)\delta}\chi(L_j^e(1^n))). 
\end{align*}
Here, we have 
\begin{align*}
r_{s\delta,(n-s)\delta} \chi(L_j^e(1^n)) &= \chi(\Res_{s\delta,(n-s)\delta}L_j^e(1^n)) \quad \text{by Theorem \ref{thm:categorification}} \\
&= \chi(L_j^e(1^s) \otimes L_j^e(1^{n-s})) \quad \text{by Corollary \ref{cor:LRrule}}. 
\end{align*}
Hence, 
\[
(\overline{P_{i,n}^e},\chi(L_j^e(1^n))) = \frac{1}{[n]_i} \sum_{s=1}^n q_i^{-(n-s)} (\overline{\psi_{i,s}^e}, \chi(L_j^e(1^s))) (\overline{P_{i,n-s}^e}, \chi(L_j^e(1^{n-s}))).
\]
By the induction hypothesis, this is zero unless $j = i$. 
If $j = i$, Lemma \ref{lem:innerprod0} shows that it is equal to 
\[
\frac{1}{[n]_i} (q_i^{-1} (-q_i)^{-n+2} + (-q_i)^{-n+1}) =0. 
\] 
The lemma is proved. 
\end{proof}

Let $\Lambda$ denote the ring of symmetric functions over $\mathbb{Q}(q)$. 
Let $e_n$ be the elementary symmetric function, $h_n$ the complete symmetric function, and $s_{\lambda}$ the Schur function.
Let $\langle \cdot, \cdot \rangle$ be the symmetric bilinear form on $\Lambda$, for which the Schur functions form an orthonormal basis. 
Let $\Delta$ be the coproduct on $\Lambda$ defined by 
\[
\Delta(h_n) = \sum_{k=0}^n h_k \otimes h_{n-k}. 
\]
We also have 
\[
\Delta(s_{\lambda}) = \sum_{\mu,\nu} c_{\mu,\nu}^{\lambda} s_{\mu} \otimes s_{\nu}. 
\]
For $f,g, h \in \Lambda$, we have 
\[
\langle fg, h \rangle = \langle f \otimes g, \Delta(h) \rangle. 
\]
For each $i \in \mathring{I}$, let $\Lambda_i$ be a copy of $\Lambda$. 
In $\Lambda_i$, we write $h_{i,n} = h_n, e_{i,n} = e_n, s_{i,\lambda} = s_{\lambda}$. 
We define
\[
\Lambda^{\otimes \mathring{I}} = \Lambda_1 \otimes \cdots \otimes \Lambda_N.
\]
For a multipartition $\underline{\lambda}$, we write 
\[
s_{\underline{\lambda}} = s_{1, \lambda^{(1)}} \otimes \cdots \otimes s_{N,\lambda^{(N)}}.
\] 

Let $U_q^-(\mathfrak{g})_{\mathrm{im}} = \bigoplus_{n \in \mathbb{Z}_{\geq 0}} U_q^-(\mathfrak{g})_{-n\delta}$, 
and let $d \colon U_q^-(\mathfrak{g})_{\mathrm{im}} \to U_q^-(\mathfrak{g})_{\mathrm{im}} \otimes_{\mathbb{Q}(q)} U_q^-(\mathfrak{g})_{\mathrm{im}}$ be the direct sum of all $r_{m\delta,n\delta} \ (m,n \in \mathbb{Z}_{\geq 0})$. 
Note that the multiplication on $U_q^-(\mathfrak{g})_{\mathrm{im}} \otimes_{\mathbb{Q}(q)} U_q^-(\mathfrak{g})_{\mathrm{im}}$ induced from (\ref{eq:twistedmultiplication}) coincides with the usual one since $(\delta,\delta) = 0$. 
Hence, $U_q^-(\mathfrak{g})_{\mathrm{im}}$ is a bialgebra. 

Let $B$ (resp.\ $B'$) be the $\mathbb{Q}(q)$-subspace of $U_q^-(\mathfrak{g})_{\mathrm{im}}$ spanned by $\chi(L^e(\underline{\lambda}))$ (resp.\ $\chi(\Delta^e(\underline{\lambda}))$) for all multipartitions $\underline{\lambda}$. 
By Theorem \ref{thm:categorification}, Corollary \ref{cor:LRrule} and Corollary \ref{cor:LRrule2}, 
both $B$ and $B'$ are sub-bialgebras of $U_q^-(\mathfrak{g})_{\mathrm{im}}$. 

Let $\eta \colon \Lambda^{\otimes \mathring{I}} \to B$ be the $\mathbb{Q}(q)$-linear isomorphism given by 
\[
\eta(s_{\underline{\lambda}}) = \chi(L^e(\underline{\lambda})). 
\]
By Corollary \ref{cor:LRrule}, it is an isomorphism of bialgebras. 

Recall $\psi_{i,n}^e, P_{i,n}^e$ and $S_{\underline{\lambda}}^e$ from Subsection \ref{sub:affinePBW}. 
Proposition \ref{prop:psimodule} shows that $\psi_{i,n}^e$ belongs to $B'$. 
Hence, $P_{i,n}^e$ and $S_{\underline{\lambda}}^e$ also belong to $B'$. 
Let $\eta' \colon \Lambda^{\otimes \mathring{I}} \to B'$ be the $\mathbb{Q}(q)$-algebra homomorphism defined by 
\[
\eta'(h_{i,n}) = \overline{P_{i,n}^e} \ (i \in \mathring{I}, n \in \mathbb{Z}_{\geq 1}). 
\]
By Jacobi-Trudi identity, we have 
 \[
  \eta'(s_{\underline{\lambda}}) = \overline{S_{\underline{\lambda}}^e}
  \]
  for any multipartition $\underline{\lambda}$.
  In particular, $\eta'$ is injective. 

\begin{lemma}
$\eta'$ is an isomorphism of bialgebras. 
\end{lemma}

\begin{proof}
Note that Theorem \ref{thm:stratification} shows that the vectors $\chi(\Delta^e(\underline{\lambda}))$ ($\underline{\lambda}$ is a multipartition) are linearly independent and form a basis of $B'$. 
By comparing the dimensions of each weight spaces, 
we see that $\eta'$ is an isomorphism. 
To prove that $\eta'$ respects coproducts, it suffices to show that $d(\overline{P_{i,n}^e}) = \sum_{k=0}^n \overline{P_{i,k}^e} \otimes \overline{P_{i,n-k}^e}$.
By the relation between $r$ and $\overline{(\cdot)}$ \cite[Lemma 1.2.11]{MR2759715} and the fact $(\delta,\delta) = 0$, 
this is equivalent to $d(P_{i,n}^e) = \sum_{k=0}^n P_{i,k}^e \otimes P_{i,n-k}^e$. 
By \cite[Proposition 3.22]{MR2066942} (note that $\wt f_i = - \alpha_i$), we have 
\[
r (P_{i,n}^e) \equiv \sum_{k=0}^n P_{i,k}^e \otimes P_{i,n-k}^e \mod \quantum{-}_{p^{-1}(\mathring{\mathsf{Q}}_+ \setminus \{0\})} \otimes \quantum{-}_{p^{-1}(\mathring{\mathsf{Q}}_- \setminus \{0\})}. 
\]
Hence, the assertion follows. 
\end{proof}

\begin{lemma}[{cf.\ \cite[Lemma 23.2]{MR3694676}}] \label{lem:bilinearforms}
For $f,g \in \Lambda^{\otimes \mathring{I}}$, we have 
\[
(\eta'(f), \eta(g)) = \langle f,g \rangle. 
\]
\end{lemma}

\begin{proof}
Since both $\eta$ and $\eta'$ are homomorphisms of bialgebras, 
the lemma is reduced to Lemma \ref{lem:innerprod} using
\[
(xx',y) = (x\otimes x',d(y)), \ \langle ff',g\rangle = \langle f\otimes f', \Delta(g) \rangle, 
\]
for $x,x',y \in U_q^-(\mathfrak{g})_{\mathrm{im}}$ and $f,f',g \in \Lambda^{\otimes \mathring{I}}$. 
\end{proof}

\begin{proof}[Proof of Theorem \ref{thm:characterformula}]
By Theorem \ref{thm:stratification} and Proposition \ref{prop:semisimpledelta},
we have 
\[
\Extform{\Delta^e(\underline{\lambda})}{L^e(\underline{\mu})} = \delta_{\underline{\lambda},\underline{\mu}}, 
\]
for multipartitions $\underline{\lambda}$ and $\underline{\mu}$. 
the left hand side coincides with 
\begin{align*}
(\overline{\chi(\Delta^e(\underline{\lambda}))}, \chi(L^e(\underline{\mu}))) &= (\overline{\chi(\Delta^e(\underline{\lambda}))}, \eta(s_{\underline{\mu}})) \quad \text{by Theorem \ref{thm:categorification}} \\
&= \langle \eta'^{-1}(\overline{\chi(\Delta^e(\underline{\lambda}))}), s_{\underline{\mu}} \rangle \quad \text{by Lemma \ref{lem:bilinearforms}}.
\end{align*}
This implies $\eta'^{-1}(\overline{\chi(\Delta^e(\underline{\lambda}))}) = s_{\underline{\lambda}}$ for any multipartition $\underline{\lambda}$. 
Hence, 
\[
\chi(\Delta^e(\underline{\lambda})) = \overline{\eta'(s_{\underline{\lambda}})} = \overline{\overline{S_{\underline{\lambda}}^e}} = S_{\underline{\lambda}}^e. 
\]
\end{proof}

\appendix
\section{Conventions for quantum groups}

Conventions for quantum groups often vary across different references. 
We summarize in a table below how to translate other references into our setup. 
To translate one reference into another, simply replace each symbol with the corresponding symbol appearing in the same row. 

\begin{table}[h]
\centering
\begin{tabular}{cccc} \toprule
This paper & \cite{MR2066942, MR3874704} & \cite{MR2759715} \\ \midrule 
$q$ & $q^{-1}$ & $v^{-1}$ \\
$e_i$ & $F_i$  & $F_i$ \\
$f_i$ & $E_i$ & $E_i$ \\
$q^h$ & $q^h$  & $K_h$ \\
$\quantum{}$ & $\quantum{}$  & $\mathbf{U}$ \\
$T_i$ & $T_i$ & $T''_{i,1}$\\
$T_i^{-1}$ & $T_i^{-1}$  & $T'_{i,-1}$ \\
convex order & opposite & \\ \bottomrule
\end{tabular}
\end{table}

\bibliographystyle{amsalpha}
\bibliography{library.bib}

\end{document}